\documentclass[11pt, reqno]{amsart} 

\usepackage[T1]{fontenc}
\usepackage{lmodern}
\usepackage{mathtools, amsfonts, amssymb, mathrsfs, bm, stmaryrd}

\usepackage{graphicx}
\usepackage{tikz}
\usetikzlibrary{shapes, arrows.meta, positioning, calc, backgrounds, lindenmayersystems, decorations.pathreplacing}
\usepackage{pgfplots, pgfplotstable}
\pgfplotsset{compat=1.17}

\usepackage{float}
\usepackage{svg}
\usepackage{wrapfig}
\usepackage{subcaption}

\usepackage{hyperref}

\pgfdeclarelindenmayersystem{cantor set}{
  \rule{F -> FfF}
  \rule{f -> fff}
}
\numberwithin{equation}{section}
\raggedright
\allowdisplaybreaks

\usepackage[letterpaper,hmargin=1.0in,vmargin=1.0in]{geometry}
\parskip 	\smallskipamount
\newcounter{dummy} \numberwithin{dummy}{section}
\newtheorem{corollary}[dummy]{Corollary}
\newtheorem{proposition}[dummy]{Proposition}
\newtheorem{theorem}[dummy]{Theorem}
\newtheorem{definition}[dummy]{Definition}
\newtheorem{lemma}[dummy]{Lemma}
\newtheorem*{remark}{Remark}

\newcommand{\N}{\ensuremath{\mathbb{N}}}
\newcommand{\R}{\ensuremath{\mathbb{R}}}
\newcommand{\Z}{\ensuremath{\mathbb{Z}}}

\newcommand{\C}{\ensuremath{\mathbb{C}}}
\newcommand{\PP}{\ensuremath{\mathbb{P}}}
\newcommand{\F}{\mathcal{F}}

\DeclareMathOperator{\supp}{supp}

\newcommand{\norm}[1]{\left\lVert#1\right\rVert}

 \newcommand\restrict[2]{% make the whole thing an ordinary symbol
   \left.\kern-\nulldelimiterspace % automatically resize the bar with \right
   #1% the function
   \littletaller % pretend it's a little taller at normal size
   \right|_{#2}%
   }
 \newcommand{\littletaller}{\mathchoice{\vphantom{\big|}}{}{}{}}

 \makeatletter
 \newcommand*{\scaleddelims}[3]{%
   \ensuremath{%
     \mathpalette{\@scaleddelims{#1}{#2}}{#3}%
   }%
 }   
 \newcommand*{\@scaleddelims}[4]{%
   \begingroup
     #3%
     \sbox0{$\m@th#3\vphantom{A}#4$}%
     \setbox2\vbox{\hbox{$\m@th#3#1$}\kern\z@}%
     \setbox4\vbox{\hbox{$\m@th#3#2$}\kern\z@}%
     \setbox6\hbox{$#3\vcenter{}$}%
     \ifx\downharpoonleft#1\relax  
       \let\DelimLeft=L%
     \else\ifx\upharpoonleft#1%
       \let\DelimLeft=L%
     \else\ifx\downharpoonright#1%
       \let\DelimLeft=R%
     \else\ifx\upharpoonright#1%
       \let\DelimLeft=R%
     \fi\fi\fi\fi
     \ifx\downharpoonleft#2\relax
       \let\DelimRight=L%
     \else\ifx\upharpoonleft#2\relax
       \let\DelimRight=L%
     \else\ifx\downharpoonright#2\relax
       \let\DelimRight=R%
     \else\ifx\upharpoonright#2\relax
       \let\DelimRight=R%
     \fi\fi\fi\fi
     \ifx\DelimLeft L%
       \wd2=.6\wd2
     \fi
     \ifx\DelimRight L%
       \wd4=.6\wd4
     \fi
     \ifx\DelimLeft R%
       \sbox2{\kern-.4\wd2\box2}%
     \fi
     \ifx\DelimRight R%
       \sbox4{\kern-.4\wd4\box4}%
     \fi
     \dimen0=\ht0 %
     \advance\dimen0 by -\ht6 %
     \dimen2=\dp0 %
     \advance\dimen2 by \ht6 %
     \ifdim\dimen2>\dimen0 %  
       \dimen0=\dimen2 %
     \else
       \dimen0=\dimen0 %
     \fi
     \dimen2=\ht6 %
     \advance\dimen2 by -\dimen0 %
     \dimen0=2\dimen0 %
     \def\DelimCorr{%  
       \mskip.5\thinmuskip
       \nonscript\mskip.5\thinmuskip
     }%
     \mathopen{%
       \ifx\DelimLeft R\DelimCorr\fi
       \raisebox{\dimen2}{\resizebox{!}{\dimen0}{\box2}}%
       \ifx\DelimLeft L\DelimCorr\fi
     }%
     \begingroup
       #3#4%
     \endgroup
     \mathclose{%
       \ifx\DelimRight R\DelimCorr\fi
       \raisebox{\dimen2}{\resizebox{!}{\dimen0}{\box4}}%
       \ifx\DelimRight L\DelimCorr\fi
     }%
   \endgroup
 }\makeatother

 \newcommand{\restr}[2]{#1\scaleddelims{\kern-0.5\nulldelimiterspace\upharpoonright}{\vphantom{.}}{_{#2}}}
\newcommand{\diff}{\ensuremath{\operatorname{d}\!}}

\title{Radon-Nikodym derivative of inhomogeneous Brownian last passage percolation}

\author{Pantelis Tassopoulos}
\address{Department of Pure Mathematics and Mathematical Statistics, 
University of Cambridge, Cambridge, United Kingdom}
\email{pkt28@cam.ac.uk}

\author{Sourav Sarkar}
\address{Department of Pure Mathematics and Mathematical Statistics, 
University of Cambridge, Cambridge, United Kingdom}
\email{ss2871@cam.ac.uk}

\begin{document}

%{$82C22$,  and  $60H15$.}
\date{}

\begin{abstract}
We show that the Radon-Nikodym derivative of the law of the spatial increments (with endpoints away from the origin) of inhomogeneous Brownian last passage percolation (LPP) with non-decreasing initial data against the Wiener measure $\mu$ on compacts is in $L^{\infty-}(\mu)$; and for any fixed $p>1$, the $L^p$ norm is at most of the order $O_p(\mathrm{e}^{d_pm^2\log m})$ for some $p$-dependent constant $d_p>0$.

Furthermore, when the initial data is homogeneous, we establish optimal growth on $L^p$ norms ($\asymp O(\exp(dm^2))$) of the Radon-Nikodym derivative of the Brownian LPP (i.e. top line of an $m$-level Dyson Brownian motion) away from the origin, as the number of curves $m$ tends to infinity, for all $p>1$ sufficiently large.

As an application of our framework, we show that the Radon-Nikodym derivative of certain toy models for the KPZ fixed point lies in $L^{\infty-}(\mu)$, inspired by its variational characterisation in terms of the directed landscape.
\end{abstract}

\maketitle

\tableofcontents

\section{Introduction}
For a collection of $m\in \N$ independent Brownian motions $B_1, \ldots, B_m$ defined on a compact interval $[0,t]$ for $t>0$, one can define their last passage percolation (\textbf{LPP} for short) value,
\[
B[(0,m)\to (t,1)]
\]
as the maximum `length' across all non-increasing integer valued paths $\pi:[0,t]\to \N$ with $\pi(0) = m$ and $\pi(t)=1$. Here the length of such a path with respect to the random environment is the total cost incurred by traversing each curve in the  environment; where for every index $i$ from $1$ to $m$, the contribution of the $i$-th Brownian motion to the cost is the increment in the value of the curve along the path, see Definition \ref{Definition: last passage} and Figure \ref{fig: last passage}. 

One can construct the Brownian LPP (or \textbf{BLPP} for short) by starting from $B_m$ and recursively reflecting upwards an independent Brownian motion off it. More generally, when one carries this recursive construction of upward reflections for Brownian motions with initial values $g_1
\ge \ldots 
\ge g_m$, one obtains at the $k^{\mathrm{th}}$ iteration 
\[
\displaystyle\max_{k\le \ell \le m}(g_{\ell} + B[(0,\ell)\to (t, k)])\,.
\]
We refer to the resulting process at the final $m$-th iteration, namely,
\begin{equation}\label{eq: inhom BLPP}
\displaystyle\max_{1\le \ell \le m}(g_{\ell} + B[(0,\ell)\to (t, 1)])
\end{equation}
as the \textbf{inhomogeneous Brownian LPP} and the tuple
\[
(\displaystyle\max_{k\le \ell \le m}(g_{\ell} + B[(0,\ell)\to (t, k)]))_{k=1}^m
\]
as the \textbf{Brownian TASEP}, where TASEP stands for totally asymmetric simple exclusion process. 

Intuitively, one can think of this system as a collection of interacting Brownian particles with asymmetric (upward) collisions. Indeed, Brownian TASEP can be obtained as the low density limit of the totally asymmetric simple exclusion process (TASEP), see \cite{gorin2015limits}. Briefly, in this model, one considers an infinite sequences of particles with positions
\[
\dots<X_t(2)<X_t(1)<X_t(0)<X_t(-1)<X_t(-2)<\dots \quad \text{ on } \Z\cup\pm \infty.
\] 
The evolution of the system is driven by each particle independently attempting to occupy its right-adjacent spot, provided it is \textit{empty}, with exponential waiting time with rate \(1\). To each integer \(u\), one can associate \(X^{-1}_t(u) = \min\{k\in \Z: X_t(k)\leq u\}\), that the label of the right-most particle with position up to \(u\). This is an example of an exactly solvable model and was one of the first models shown to lie in the so-called KPZ universality class. This universality class is a collection of random growth models, polymer models and stochastic PDEs all having similar qualitative characteristics. For an introduction to the KPZ universality class, see\cite{quastel2011introduction}, \cite{ferrari2010random}, \cite{romik2015surprising}, \cite{corwin2016kardar}, \cite{weiss2017reflected}, \cite{ganguly2021random} and \cite{ZygourasalgKPZ}.

When the initial data is homogeneous, that is $g_i=0$ for all $i=1,\ldots,m$, the above construction coincides with that of Dyson's Brownian motion through the RSK correspondence and was established in \cite{O2002representation}. In that case, for all $p>1$ sufficiently large, we establish optimal growth on $L^p$ norms ($\asymp O(\exp(dm^2))$) of the Radon-Nikodym derivative of the top level of Dyson Brownian motion away from the origin, as the number of levels tends to infinity by observing that the measure of Dyson Brownian motion paths can be constructed by the $h$-transform of a Brownian motion conditioned to never leave the Weyl chamber.

Now if the Brownian motion starting points are arbitrary, in \cite[Theorem 4.3]{Sarkar2021Brownian}, it was shown that away from zero, inhomogeneous Brownian LPP is absolutely continuous with respect to Brownian motion on compacts. When the initial conditions are non-increasing, we substantially strengthen this comparison and show using diffusion interlacing arguments and the theory of Markov processes that the Radon-Nikodym derivative of the law of the spatial increments (with endpoints away from zero) of the inhomogeneous BLPP against the Wiener measure $\mu$ on compacts is in $L^{\infty-}(\mu)$. Also, for any fixed $p>1$, one has that the $L^p$ norm is at most of the order $O_p(\mathrm{e}^{d_pm^2\log m})$ for some $p$-dependent constant $d_p>0$, which is the main technical result of this paper (Theorem \ref{thm: tasep bm abs cont}). We state it below.

\begin{theorem}(Radon-Nikodym derivative estimates)\label{thm: main informal} Fix $m\ge 1$, and let $H(\cdot)$ denote the inhomogeneous Brownian LPP started from initial data $g_1
\ge \ldots 
\ge g_m$ as defined in \eqref{eq: inhom BLPP}. Then, for all $0<\ell<r<\infty$, we have that the Radon-Nikodym derivative
     of the law of $H(\cdot)$ against a rate two Brownian motion starting from the origin $\mu$ on $[\ell,r]$ is in $L^{\infty-}(\mu\vert_{[\ell, r]})$. In particular, with $\xi_{\ell, r, m, \underline{b}}$ denoting the law of $H$ as defined above on $[\ell, r]$,
     \begin{equation*}
         \norm{\frac{\diff \xi_{\ell, r, m, \underline{b}}}{\diff\mu\vert_{[\ell, r]}}}_{L^p(\mu\vert_{[\ell, r]})} =  O_p(\mathrm{e}^{d_pm^2\log m}),\quad \mbox{for all}\quad p>1.
     \end{equation*}
     for some universal in $m\in \N$ (though possibly $p$-dependent) constant $d_p>0$ for all $p>1$.

\end{theorem}

The key observation that enables us to proceed with the proof is that the inhomogeneous Brownian LPP with non-increasing initial data is a version of the regular conditional distribution of the Warren process  (see \cite{warren2007dyson}), in the Gelfand-Tsetslin cone. This arises from the iteration of a certain diffusion interlacing procedure consisting of upward and downward Skorokhod reflections to `interlace' newly added independent Brownian motions between existing processes. One can thus obtain explicit densities, and owing to the Markovian nature of the inhomogeneous Brownian LPP, the proof of Theorem \ref{thm: tasep bm abs cont} is reduced to estimating the ratio of densities using the estimates obtained for Dyson Brownian motion. The final ingredient consists of several integral estimates, relegated to the Appendix, that are crucial in estimating the Radon-Nikodym derivative.

As an application, we provide a framework to obtain $L^p$ estimates for the Radon-Nikodym derivative of various toy models for the KPZ fixed point inspired by its variational characterisation in terms of the directed landscape. 

Moreover, in \cite[Theorem 6.6]{tassopoulos2025quantitativebrownianregularitykpz}, we use these estimates to obtain a form of quantitative Brownian regularity for the spatial increments of the KPZ fixed point started from arbitrary, that is, all finitary initial data (see \cite{Sarkar2021Brownian} for a definition of finitary initial data). We believe our arguments can be strengthened to obtain $L^p$ norm estimates for the Radon-Nikodym derivative of spatial increments of the KPZ fixed point on compacts for all such initial data.

\iffalse
\begin{theorem}(Quantitative Brownian regularity, \cite[Theorem 6.6]{tassopoulos2025quantitativebrownianregularitykpz})\label{thm: main informal companion}
Let $\F$ be any class of continuous (with respect to the subspace topology) and uniformly bounded from above initial data that are uniformly bounded with countable 'max-plus' support contained in a fixed, reference compact set. Then the spatial increments of the KPZ fixed point on a fixed interval withstand a form of quantitative Brownian regularity with rate function of the form
\[
\exp\left(-d\log^{r}\log\big(1/\mu(A)\big)\right)\,, 
\]
for all $A$ Borel measurable sets on paths and some positive (independent of $A$) $d>0, r\in (0,1)$ where $\mu$ denotes an appropriate restriction of the rate two Wiener measure.
\end{theorem}
\fi

\subsection{Organisation of the paper}
In Section \ref{sec: notation}, we set up the notation we will be using throughout. In Section \ref{sec: preliminaries} we provide necessary background material. Section \ref{sec: dyson bm estimates} is devoted to establishing pathwise Radon-Nikodym derivative estimates for the top line of Dyson Brownian motion, see Proposition \ref{prop: top line pitman pointwise bound general case}. This leads to a characterisation of the growth of $L^p$ norms for all $p>1$ sufficiently large of the Radon-Nikodym derivative away from the origin, as the number of curves in the Dyson Brownian motion tends to infinity, see Theorem \ref{thm: melon lb rn}. Then, in Section \ref{sec: inhom blpp} we introduce Brownian TASEP and investigate some of its structural properties leading to its characterisation as a semi-martingale satisfying a singular stochastic differential equation, see Theorem \ref{thm: semi-mg decomp blpp}. Next, we introduce the Warren process in Section \ref{sec: BM gelfand-tsetslin cone} and establish its close connection to inhomogeneous BLPP. This allows us to compute explicit densities of the iterated reflection processes, and obtain an explicit expression for the Radon-Nikodym derivative of inhomogeneous BLPP against that of Dyson Brownian motion in Theorem \ref{thm: Radon-Nikodym  density inhom BLPP}. In Section \ref{sec: BM tasep rn estimates}, the main result of this paper is established in Theorem \ref{thm: tasep bm abs cont}. It provides pathwise and $L^{\infty-}$ estimates for the Radon-Nikodym derivative of inhomogeneous BLPP with respect to Brownian motion on compacts and the $L^p$ norm of the derivative for any $p>1$. The technical integral estimates used here have been relegated to the Appendix (Section \ref{app: int est}). Finally, in Section \ref{sec: misc}, we apply our estimates to obtain pathwise and $L^{\infty-}$ estimates for a `toy' model of the KPZ fixed point, namely for inhomogeneous BLPP of ``random depth''; see Theorem \ref{thm: tasep bm abs cont}. Finally, we briefly discuss \cite[Theorem 6.6]{tassopoulos2025quantitativebrownianregularitykpz}, and how we use the estimates in this paper to obtain a form of quantitative Brownian regularity for the spatial increments of the KPZ fixed point started from arbitrary, or ($t$-)finitary (for $t>0$) initial data, that are exactly those initial data that prevent a finite time `blow-up' of the KPZ fixed point at times up to $t>0$.

\subsection{Related work}
In \cite{warren2007dyson, assiotis2019interlacing}, the authors construct systems of diffusions with singular drifts corresponding to collisions induced by `interlacing' various diffusion semigroups (see also \cite{rogers1981markov}). In particular, they construct Brownian motion in the Gelfand-Tsetslin cone, whose diagonal section is shown to be identical in \cite[Proposition 6]{warren2007dyson} in law to Dyson Brownian motion. Various determinantal formulae are obtained for transition densities related to such interlaced diffusions. This is a crucial input in the proof of Theorem \ref{thm: tasep bm abs cont}.

In \cite[Theorem 4.3]{Sarkar2021Brownian} it is shown that general inhomogeneous BLPP is locally absolutely continuous with respect to Brownian motion away from the origin. In our work, Theorem \ref{thm: tasep bm abs cont} is a considerable strengthening of the comparison of inhomogeneous BLPP against Brownian motion, giving the aforementioned $L^{\infty-}$ control against the Wiener measure.

In \cite[Theorem 6.6]{tassopoulos2025quantitativebrownianregularitykpz}, we use these estimates to obtain a form of quantitative Brownian regularity for the spatial increments of the KPZ fixed point started from finitary initial data. 

\subsection{Acknowledgement} SS would like to thank B\'alint Vir\'ag for some initial helpful discussions. 

\section{Notation}\label{sec: notation}

We introduce some notation and conventions we will be using throughout.

When in some estimates a constant appears that will depend on some parameters $a,b,c,\cdots$, it will be denoted by $C_{a,b,c,\cdots}$ unless otherwise specified. Constants without subscripts are deemed to be universal, unless otherwise stated. Additionally, for ease of notation, such constants are allowed to change from line to line. Moreover, for ease of notation such constants may be dropped and instead replaced with the symbols $\lesssim_{a,b,c,\cdots}(\equiv O_{a,b,c,\cdots}(\cdot))$ and $\gtrsim_{a,b,c,\cdots}$ for some parameters $a,b,c,\cdots$ which stand for $\le C_{a,b,c,\cdots}$ and $\ge C'_{a,b,c,\cdots}$ for some constants $C_{a,b,c,\cdots}, C'_{a,b,c,\cdots}$ respectively.

We take the set of natural numbers $\N$ to be $\{1,2, \ldots \}$. For $k\in \N$, we use an underbar to denote a $k$-vector, that is, $\underline{x}\in \R^k$. We denote the integer interval $\{i,i+1, \ldots, j\}$ by $\llbracket i, j\rrbracket$. A $k$-vector $\underline x = (x_1, \ldots, x_k)\in \R^k$ is called a $k$-decreasing list if $x_1>x_2> \ldots >x_k$. For a set $I\subseteq \R$, let $I^k_> \subseteq I^k$ be the set of $k$-decreasing lists of elements of $I$, and $I^k_{\geq}$ be the analogous set of $k$-non-increasing lists.

The symbols $\cdot \land \cdot , \cdot \lor \cdot$ denote $\min\{\cdot, \cdot \}$ and $\max\{\cdot, \cdot \}$ respectively. For any $a
\in \R$, $a_+$ denotes $a\lor 0$.

We now turn to some notational conventions for the path spaces that will be used throughout. For general domains of paths $J$, we denote the space of continuous paths, in the usual topologies, by $C_{*,*}(J, \R)$. More specifically, if the domain is an interval $[a,b]\subseteq \R$, we denote the space of continuous functions with domain $[a,b]$ which vanish at $a$ by $C_{0,*}([a,b], \R)$. For random functions taking values in these spaces, we will always endow them with their respective Borel $\sigma$-algebras generated by the topologies of uniform convergence (which makes them into Polish spaces). Similarly, for $k\ge 1, a < b$, define $C^k_{*,*}([a,b], \R):= \bigtimes_{i=1}^k C_{*,*}([a,b], \R)$ and equip it with the product of the uniform topologies. 

For $k\in \N$, $a<b$ and $\underline{x},\underline{y}\in \R^k_{>}$, the \textit{non-crossing} event on any fixed union of finite sub-intervals $J \subseteq [a,b]$ is denoted by
\begin{equation}\label{eq: noint}
\mathrm{NoInt}(J) \coloneqq \Big\{g \in C^k_{*, *}(J, \R): \forall r\in J,  g_{i}(r) > g_{j}(r) \textrm{ for all } 1\leq i<j\leq k\Big\}\,.
\end{equation} 

We say that a Brownian motion or a Brownian bridge has \textit{rate $v$} if its quadratic variation in an interval $[s,t]$ is equal to $v(t-s)$. We say that a Dyson's Brownian motion or a Brownian $k$-melon has rate $v$ if the component Brownian motions have rate $v$. From now on, all Brownian motions are rate two unless stated otherwise. 

\section{Preliminaries}\label{sec: preliminaries}
We will now collect some basic definitions and provide necessary background material that will be crucial in setting up pathwise Radon-Nikodym derivative estimates for Brownian last passage percolation (BLPP). To this end, we start with the definition of random line ensembles and last passage percolation (LPP), that underlie the constructions we will be henceforth considering.

\begin{definition}[Random ensemble]\label{def: random ensemble}
    Let $\Sigma$ be a (possibly infinite) interval of $\Z$, and let $\Lambda$ be an interval of $\R$. Consider the set \(X:=C^\Sigma\) of continuous functions $f:\Sigma\times \Lambda \rightarrow \R$. We endow it with the topology of uniform convergence on compact subsets of $\Sigma\times\Lambda$. Let $\mathscr{C}$ denote the sigma-field  generated by Borel sets in $X$.

    A {\it $\Sigma$-indexed line ensemble} $L$ is a random variable defined on a probability space $(\Omega,\mathfrak{B},\PP)$, taking values in $X$ such that $L$ is a $(\mathfrak{B},\mathscr{C})$-measurable function. Furthermore, we write $L_i:=(L(\omega))(i,\cdot)$ for the line indexed by $i\in\Sigma$. 
\end{definition}% Random Ensemble

\subsection{Last passage percolation} We begin with the collection of some preliminary facts regarding last passage percolation (sometimes abbreviated as LPP in the paper) over ensembles of functions.

Formally, let \(I\subset \mathbb{Z}\) be a possibly finite index set and define the space \(C^{I}\) of sequences of continuous functions with real domains, that is, the space
\[f: \mathbb{R}\times I\to \mathbb{R}\quad (x,i)\mapsto f_i(x)\,.\]

\begin{definition}[Path] Let \(x\leq y \in \mathbb{R}\), and \(m\leq k \in \mathbb{Z}\) respectively. A \textbf{path}, from \((x,k)\) to \((y,m)\) is a non-increasing  function \(\pi: [x,y] \to \mathbb{N}\) which is cadlag on \((x,y)\) and takes the values \(\pi(x)= k\) and \(\pi(y)= m\).
\label{def: path}
\end{definition}%Path definition

\begin{remark}
    The convention that the paths be non-increasing is so that they match the natural indexing of the Airy line ensemble, see the seminal paper of \cite{Corwin2014Brownian}. 
\end{remark}

We now define an important quantity associated to each such path, namely, its \textit{length} as the sum of increments of \(f\) along \(\pi\). This also leads one to naturally define a derived quantity, namely the \textit{last passage value}.

\begin{definition}[Length]\label{def: length} Let \(x\leq y \in \mathbb{R}\) and \(m < k\in\mathbb{Z}\). For each \(m\leq i <k\), let \(t_{k-i}\) denote the jump of the path \(\pi\), on an ensemble \((f_i)_{i\in I}\), from \(f_{i+1}\) to \(f_{i}\). Then the length of \(\pi\) is defined as
\[
\ell(\pi) = f_m(y)-f_m(t_{k-m}) + \displaystyle\sum_{i = 1}^{k-m-1}(f_{k-i}(t_{i+1})-f_{k-i}(t_{i}))+f_{k}(t_{1})-f_{k}(x)\,.
\]
\end{definition}%Length definition

\begin{figure}
  \centering
\begin{tikzpicture}[scale=1.0]

% Parameters
\def\startx{1}
\def\endx{9}
\def\startg{2.5}
\def\endg{7.5}
\def\spacing{2.0}
\def\amp{0.15}     % vertical noise amplitude
\def\nsegments{100}

\pgfmathsetseed{2025}

% Compute special x-locations
\pgfmathsetmacro{\tOne}{\startg+1*(\endg-\startg)/4}
\pgfmathsetmacro{\tTwo}{\startg+2*(\endg-\startg)/4}
\pgfmathsetmacro{\tThree}{\startg+3*(\endg-\startg)/4}

% Generate four curves
\foreach \i in {1,...,4} {

  % ----- Black noisy curve -----
  \def\coords{}%
  \foreach \k in {0,...,\nsegments} {
    \pgfmathsetmacro{\xx}{\startx+(\k/\nsegments)*(\endx-\startx)}
    % baseline: simple linear decreasing with i, different slope per level
    \pgfmathsetmacro{\trend}{\i*\spacing - 2.5 + 0.3*(\i-2)*(\xx-5)/4}
    \pgfmathsetmacro{\noise}{(2*rnd-1)*\amp}
    \pgfmathsetmacro{\yy}{\trend+\noise}
    \xdef\coords{\coords (\xx,\yy)}%
  }
  \draw[thick, black!60] plot coordinates {\coords};

  % ----- Blue highlighted segment -----
  \pgfmathsetmacro{\segA}{\startg+(\i-1)*(\endg-\startg)/4}
  \pgfmathsetmacro{\segB}{\startg+\i*(\endg-\startg)/4}
  \def\bluecoords{}%
  \foreach \k in {0,...,12} {
    \pgfmathsetmacro{\xx}{\segA+(\k/12)*(\segB-\segA)}
    \pgfmathsetmacro{\trend}{\i*\spacing - 2.5 + 0.3*(\i-2)*(\xx-5)/4}
    \pgfmathsetmacro{\noise}{(2*rnd-1)*0.05} % smaller jaggedness for clarity
    \pgfmathsetmacro{\yy}{\trend+\noise}
    \xdef\bluecoords{\bluecoords (\xx,\yy)}%
  }
  \draw[thick, blue] plot coordinates {\bluecoords};

  % ----- Label above interval -----
  \pgfmathsetmacro{\xmid}{0.5*(\segA+\segB)}
  \pgfmathsetmacro{\ytoplabel}{\i*\spacing - 2.5 + 0.6}
  \node[blue, above] at (\xmid, \ytoplabel) {$\Delta_{\i}$};

  % ----- Vertical dashed connectors -----
  \ifnum\i>1
    \ifnum\i<5
      \pgfmathsetmacro{\jumpx}{\startg+(\i-1)*(\endg-\startg)/4}
      \pgfmathsetmacro{\ytop}{(\i-1)*\spacing - 2.5 + 0.3*((\i-1)-2)*(\jumpx-5)/4}
      \pgfmathsetmacro{\ybot}{\i*\spacing - 2.5 + 0.3*(\i-2)*(\jumpx-5)/4}
      \draw[blue, dashed] (\jumpx,\ytop) -- (\jumpx,\ybot);
      \draw[blue, thick, fill=white] (\jumpx,\ytop) circle (2pt);
      \fill[blue] (\jumpx,\ybot) circle (2pt);
    \fi
  \fi
}

% ----- Start/end blue dots -----
\pgfmathsetmacro{\ystart}{1*\spacing - 2.5 + 0.3*(1-2)*(\startg-5)/4}
\pgfmathsetmacro{\yend}{4*\spacing - 2.5 + 0.3*(4-2)*(\endg-5)/4}
\fill[blue] (\startg,\ystart) circle (2pt);
\fill[blue] (\endg,\yend) circle (2pt);

% Labels near start/end
\node[left] at (\startg, \ystart+0.4) {$(x,4)$};
\node[right] at (\endg, \yend-0.2) {$(y,1)$};

% ----- Horizontal time axis -----
\draw[->] (\startg-0.5,-1.5) -- (\endg+0.5,-1.5);

% ----- Vertical dotted lines for special x -----
\pgfmathsetmacro{\ytOne}{2*\spacing - 2.5 + 0.3*(2-2)*(\tOne-5)/4}
\pgfmathsetmacro{\ytTwo}{3*\spacing - 2.5 + 0.3*(3-2)*(\tTwo-5)/4}
\pgfmathsetmacro{\ytThree}{4*\spacing - 2.5 + 0.3*(4-2)*(\tThree-5)/4}

\foreach \px/\py/\lab in {
    \startg/\ystart/$x$,
    \tOne/\ytOne/$t_1$,
    \tTwo/\ytTwo/$t_2$,
    \tThree/\ytThree/$t_3$,
    \endg/\yend/$y$
} {
  \draw[dotted] (\px,-1.5) -- (\px,\py);
  \filldraw[black] (\px,-1.5) circle (2pt);
  \node[below] at (\px,-1.5) {\lab};
}

\end{tikzpicture}  \caption{Visualisation of a possible path (\color{blue}blue\color{black}) `embedded' on the underlying environment (random ensemble $F:[x,y]\times \llbracket 1, 4\rrbracket\to \R$), here \((F_1, F_2, F_3, F_4)\) from top to bottom, and \(m = 1, k = 4\). Here $\Delta_1 = F_4(t_1)-F_4(x)$, $\Delta_2 = F_3(t_2)-F_3(t_1)$, $\Delta_3 = F_2(t_3)-F_2(t_2)$, $\Delta_4 = F_1(y)-F_1(t_3)$ and $k = \sum_{i=1}^4 \Delta_i.$}
\label{fig: last passage}
\end{figure}
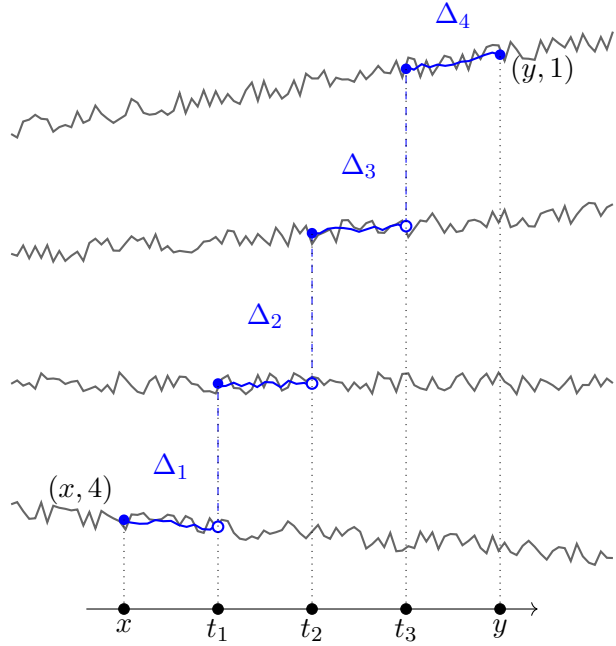

\begin{definition}[Last passage value]\label{Definition: last passage}
    With \(x\leq y, m<k\) as before and \(f\in C^{I}\), define the \textbf{last passage value} of \(f\) from \((x,k)\) to \((y,m)\) as
    \[
    f[(x,k)\to(y,m)] \coloneqq \displaystyle \sup_{\pi}k(\pi)\,,
\]

where the supremum is over precisely the paths \(\pi\) from \((x,k)\) to \((y,m)\).
\end{definition}%Last Passage Value
\begin{remark}
    Any path \(\pi\) from \((x,k)\) to \((y,m)\) such that its length is equal to its last passage value is called a \textbf{geodesic}. To establish the existence of geodesics one can proceed by first noticing that the length of a path \(\ell(\pi)\), can be viewed as a function on the subset \(\mathcal{Z}\) of non-increasing cadlag functions with fixed endpoints, of the space of cadlag functions \(\mathbf{D}\coloneqq \mathbf{D}([x,y], \N)\). When endowed with respect to the  Skorokhod topology, which is metrisable, the above function is continuous. Since \(\mathcal Z\) is closed with respect to the above topology of \textquotedblleft jump times\textquotedblright, a compactness argument using Arzela-Ascoli, see \cite[ch. 3]{billingsley2013convergence}, implies that the supremum over admissible paths is indeed attained.
\end{remark}

Last passage percolation enjoys the following \textbf{metric composition law}, Lemma 3.2 in DOV \cite{DOV}.

\begin{lemma}[Metric composition law]\label{Lemma: Metric Composition}
    Let \(x\leq y \in \mathbb{R}\), \(m < \ell\in\mathbb{Z}\) and \(f\in C^I\). If \(k\in \{m, \dots, \ell\}\), then we have
    \[
    f[(x,\ell)\to(y,m)] = \displaystyle \sup_{z\in[x,y]}(f[(x,\ell)\to(z,k)]+f[(z,k)\to(y,m)])\,,
    \]
    and if \(k\in \{m+1, \dots, \ell\}\), then 
    \[
    f[(x,\ell)\to(y,m)] = \displaystyle \sup_{z\in[x,y]}(f[(x,\ell)\to(z,k)]+f[(z,k-1)\to(y,m)])\,.
    \]
    Furthermore for any \(z\in [x,y]\), 
    \begin{equation}\label{eq: composition}
    f[(x,\ell)\to(y,m)] = \displaystyle \sup_{k\in \{m, \dots, \ell\}}(f[(x,\ell)\to(z,k)]+f[(z,k)\to(y,m)])
    \end{equation}
\end{lemma}%Metri composition Law

\subsection{The Pitman transform}

Recall that with $f=(f_1,f_2)$ where $f_i:[0,\infty)\mapsto \R$ for $i=1,2$, for $f\in C^2_{*,*}([0,\infty))$, we define $\mathrm \mathrm{W}f=(\mathrm{W} f_1,\mathrm{W} f_2)\in C^2_{*,*}([0,\infty))$, the \textbf{Pitman transform} of $f$ as follows. For $x<y\in [0,\infty)$, define the maximal gap size
\[G(f_1,f_2)(x,y):=\max\left(\max_{s\in [x,y]}\big(f_2(s)-f_1(s)\big)\,,\,0\right)\,.\]
Then define
\begin{equation}\label{eq: pitmantrans}
\mathrm{W} f_1(t)=f_1(t)+G(f_1,f_2)(0,t)\,,
\end{equation}
$$\mathrm{W} f_2(t)=f_2(t)-G(f_1,f_2)(0,t)\,,$$
for all $t\in [0,\infty)$.

One can express the top line of the Pitman transform in terms of last passage values. 

\begin{lemma}\label{lemma: Pitman melon}
    Let \(f\in C^2_+\) and let \(Wf = (Wf_1, Wf_2)\) be as above. Then for all \(t\in[0,\infty)\),
    \[
    Wf_1(t) = \displaystyle \max_{i=1,2}\{f_i(0) + f [(0, i) \to (t, 1)]\} \,.
    \]
\end{lemma}%Two curve Melon

\begin{proof}
    By definition, 
    \[
    Wf_1(t) = f_1(t)+G(f_1, f_2)(0,t)\]
    \[
    = f_1(t) + \max\{\displaystyle \max_{s\in[x,y]}(f_2(s)-f_1(s)), 0 \}
    \]
    \[
    = \max\{\displaystyle \max_{s\in[x,y]}(f_2(s)+f_1(t)-f_1(s)), f_1(t) \}\,.
    \]
    From \ref{eq: composition}, we get \(f_1(t) = f_1(0) + f [(0, 1) \to (t, 1)]\) and \[ \displaystyle \max_{s\in[0,t]}(f_2(s) + f_1(t) - f_1(s))\} = f_2(0) + f [(0, 2) \to (t, 1)]\,.\] Combining the above gives the result. 
\end{proof}

Particularly in the case where \(f_1(0) = f_2(0) = 0\), we obtain that 
\[
Wf_1(t) = f [(0, 2) \to (t, 1)]\,.
\]
\(Wf\) is commonly referred to as the \textbf{2-melon} (which will be generalised in the following section to so-called $n$-melons) of $f$, since paths in \(Wf\) avoid each other and thus resemble the stripes of a watermelon. For an illustration involving two Brownian motions, see Figure \ref{fig: pitman}.

\begin{figure}%Pitman Transform
\centering
\includegraphics[width = 0.6\linewidth]{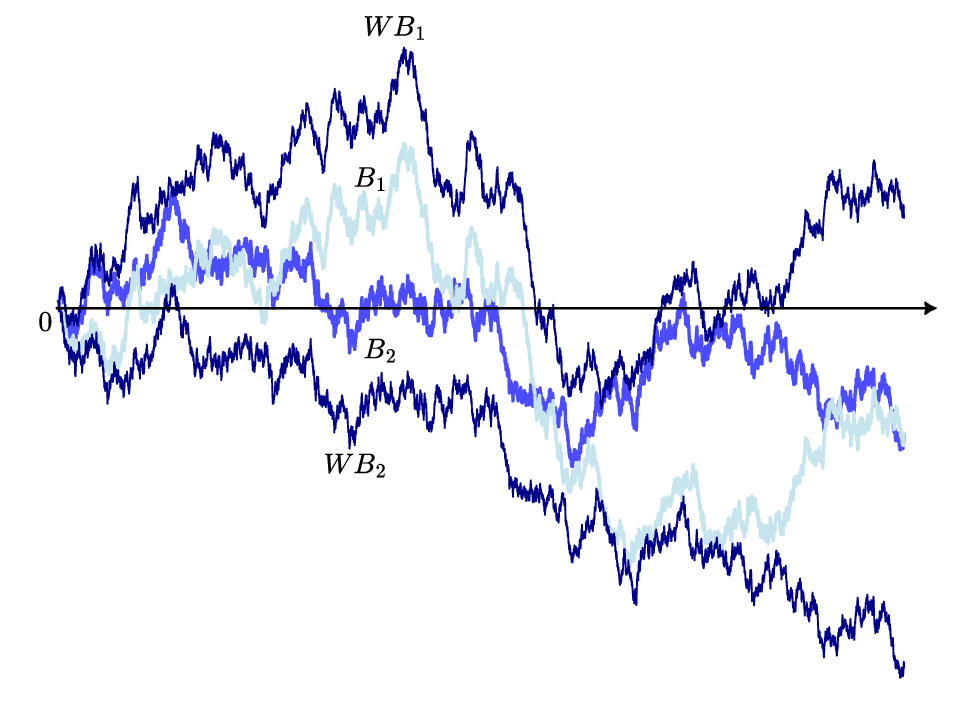}
\caption{An illustration of the Pitman transform $WB$ of two Brownian motions $B_1, B_2$.}
    \label{fig: pitman}
\end{figure}

\subsection{Dyson Brownian motion} 

Fix any $\epsilon, t > 0$ and let $B^n$ be the collection of $n$ independent Brownian motions with initial conditions $B^n_i(0) = 0$ conditioned not to intersect on $[\epsilon, t]$ (note the non-intersection event has positive probability). Then, as $\epsilon \searrow 0$, $t \nearrow \infty$, Kolmogorov's extension theorem gives that the $B^n$ converges in law to a limiting process, namely, $n$-level Dyson Brownian motion.

An alternative construction is to first take $x\in \R^n_{>}$ and with $ \PP_x$ denoting the law of $n$ independent Brownian motions $B$ started at $x$ and $\hat{\PP}_x$ the law of the $h$-transform of $B$ started at $x$ where $h(x_1, x_2, \cdots, x_n) = \prod_{1\le i < j\le n}(x_i-x_j)_+$. Then the weak limit of $\hat{\PP}_x$ as $\R^n_> \ni x \to 0$ can be realised as a random ensemble with law on paths $\hat{\PP}_{0^+}$ which agrees with the $n$-level Dyson Brownian motion starting from the origin. The advantage of this construction is that it is more amenable to Radon-Nikodym estimates.

It is worth mentioning that the Dyson Brownian motion was initially described in \cite{dyson1962brownian} as the eigenvalues of $n \times n$ time-dependent Hermitian matrices with entries independent complex-valued Brownian motion. 

\subsection{Melons}\label{subsec: melons}An application of the above that is of interest is that of two independent standard Brownian motions (starting from zero) \(B = (B_1, B_2)\).  Let \(\hat{B}=(\hat{B}_1, \hat{B}_2)\) be two independent Brownian motions conditioned not to collide, in the sense of Doob (a $2$-Dyson Brownian motion). Then, the law of the melon \(WB\) as defined above in (\ref{eq: pitmantrans}) is the same as that of \(\hat{B}\). In \cite{O2002representation}, a generalisation was proved for \(n\) Brownian motions, using a continuous analogue of the Robinson–Schensted–Knuth (RSK) correspondence, where each level in the \textit{$n$-melon} $WB^n = (WB^n_1, WB^n_2, \cdots, WB^n_n)$ is obtained from a family of $n$ Brownian motions by a sequence of deterministic operations that are analogous to the sorting algorithm `bubble sort' where the top curve $WB^n_1$ coincides with the top level of an $n$-Dyson Brownian motion. The term melon comes from the ordering of paths: for some continuous $n-$tuple $f$, $(Wf)^n_1 \ge (Wf)^n_2 \ge \dots \ge (Wf)^n_n$ and their initial value which is $0$, which means they look like stripes on a watermelon. When clear from context, we will abuse notation and drop the superscript, writing instead $Wf$. 

The RSK correspondence, group representation theoretic combinatorial procedure has helped establish strong connections between the theory of random matrices and certain combinatorial models, including last passage percolation, 
random growth models and queueing systems, see \cite{baik1999distribution} and \cite{johansson2000shape}, amongst many others. See K\"{o}nig, \cite{konig2005orthogonal} for a recent survey. For more details on the RSK correspondence, the interested reader can consult \cite{ZygourasalgKPZ} (see also  \cite{dauvergne2022rsk} and \cite{oconnell2003conditioned}).

In particular, \cite[Proposition 4.1]{DOV} gives an important property of melon paths in that they preserve last passage values (with no restriction on their starting point). In particular, 
\[
WB[(0,n)\to (t,1)] = B[(0,n)\to (t,1)]\,,\forall t\ge 0\,.
\]
Using the fact that $WB^n(0) = \underline{0}$ and the ordering of melon paths, one gets that the left-hand-side of the above equation is just $WB_1(t)$. Thus the top line of melon paths is completely characterised in terms of Brownian last passage percolation. For a more complete definition of melons involving the remaining lines, see \cite[sec. 2]{DOV} and \cite{O2002representation}.

After appropriate rescaling, \(WB^n\) converges in law to a non-intersecting ensemble on \(C^{\mathbb{N}}\) (with respect to the product of the uniform-on-compact topologies on \(C^{\mathbb{N}}\)), called the parabolic Airy line ensemble; see Theorem 2.1 in \cite{DOV} for more on this. 

\section{Dyson Brownian motion: Radon-Nikodym derivative estimates}\label{sec: dyson bm estimates}
 In this section, we consider iterating the Pitman transform on a family of independent Brownian motions starting from the origin. We obtain pathwise, as well as optimal $L^p$ Radon-Nikodym derivative estimates for all $p>1$ against Brownian motion on compacts away from zero. This approach is inspired by \cite{O2002representation} (see Subsection \ref{subsec: melons}), wherein the authors introduce a `sorting' procedure for a family of independent Brownian motions involving Pitman transforms and (with inputs from queueing theory) show that BLPP is equal in distribution with the top curve in a Dyson ensemble, \cite[Theorem 7]{O2002representation}. This allows us to exploit the representation of the Dyson ensemble as the Doob $h$-transform of Brownian motions conditioned to not leave the Weyl chamber, and obtain the desired estimates.
 
We first establish pointwise estimates for the Radon-Nikodym derivatives of the top line of the Dyson Brownian motion in Proposition \ref{prop: top line pitman pointwise bound general case} leading to Theorem \ref{thm: melon lb rn} establishing optimal growth of the $L^p$ norms of the Radon-Nikodym derivative of Brownian LPP away from the origin, as the number of curves tends to infinity.

\begin{proposition}\label{prop: top line pitman pointwise bound general case}Fix $n\in \N$, $0<\ell<r$ and let $H(\cdot)$ denote the top line of an $n$-Dyson Brownian motion. Then the law of $H$ restricted to $[\ell,r]$ is absolutely continuous with respect to that of a standard Brownian motion starting from $0$ restricted to $[\ell,r]$. The Radon-Nikodym derivative of $H$ at a path $\xi$ is bounded above by
\begin{equation*}
    \frac{c^{n(n-1)}n^{n(n-1)}}{\prod_{j=1}^{n-1}j!}\cdot (\xi(\ell)_+/\sqrt{\ell}+1)^{n-1}\cdot (\xi(r)_+/\sqrt{r}+1)^{n-1}.
\end{equation*}
for some universal constant $c>1$.
\end{proposition}
\begin{proof}

Let $x\in \R^n_{>}$ and $ \PP_x$ denote the law of a Brownian motion $B$ started at $x$ and $\hat{\PP}_x$ the law of the $h$-transform of $B$ started at $x$ where $h(x_1, x_2, \cdots, x_n) = \prod_{1\le i < j\le n}(x_i-x_j)_+$. With $T = \inf\{t\ge 0: B\; \text{hits} \; \partial r^n_>\}$, we have for $A\in \mathcal{F}_t$, the natural filtration of $B$, 
\begin{equation*}\label{eq: dyson rn basic estimate}
    \hat{\PP}_x(A) = \mathbb{E}_x\left[\frac{h(B_t)}{h(x)}\mathbf{1}_{\{t<T\}\cap A}\right] \le \mathbb{E}_x\left[\frac{h(B_t)}{h(x)}\mathbf{1}_{A}\right]
\end{equation*}
where $\mathbb{E}_x$ denotes the expectation with respect to $\PP_x$. Now also observe that we have established a pointwise bound on the Radon-Nikodym derivative of $\hat{\PP}_x \ll \PP_x$ on $C_{*,*}([0,t])$ for all $t\ge 0$. It is also not hard to see that upon taking the limit $\R^n_> \ni x \to 0$, $H$ can be realised as the top curve of a random ensemble with $\hat{\PP}_{0^+}$ satisfying (\cite{O2002representation})
\begin{equation*}
    \hat{\PP}_{0^+}(A) = \PP_0 \left[C_t h(B_t)^2 \hat{\PP}_{B_t}(\theta_t A)\right], \, A\in \sigma(B_u, u\ge t),
\end{equation*}
where $\theta_t$ is the shift operator on path space $A\in \sigma(B_u, u\ge t)\mapsto \theta_t A\in \sigma(B_u, u\ge 0)$ and 
\begin{equation*}
    C_t = \left[t^{n(n-1)/2}\prod_{j=1}^{n-1}j!\right]^{-1}
\end{equation*}
Now let $0<\ell<r$ be given and let $A\subseteq C_{*,*}([\ell,r])$ be Borel measurable, then with $t = \ell>0$, we compute
\begin{equation*}
\begin{array}{cc}
     & \hat{\PP}_{0^+}(A) = \PP_0 \left[C_\ell h(B(\ell))^2 \hat{\PP}_{B_t}(\theta_L A)\right]\\
     & \le \mathbb{E}_0\left[C_\ell h(B(\ell))^2 \mathbb{E}_{0}\left[\frac{h(\bar{B}(r-\ell)+B(\ell))}{h(B(\ell))}\mathbf{1}_{\theta_L A}\right]\right] \\
     & = \mathbb{E}_0\left[C_\ell h(B(\ell)) \mathbb{E}_{0}\left[h(\bar{B}(r-\ell)+B(\ell))\mathbf{1}_{\theta_L A}\right]\right] \\
     & = \mathbb{E}_0\left[C_\ell h(B(\ell))h(B(r))\mathbf{1}_{A}\right]
\end{array}
\end{equation*}
    where $\bar{B}$ denotes a standard Brownian motion independent from $B$ and for the last equality the Markov property for Brownian motion was used (independent increments). Thus, the Radon-Nikodym derivative of $H$ against standard Brownian motion on $[\ell,r]$ is pointwise bounded for $\mu$-a.a. paths $\xi$ by (essentially computing the marginal)
\begin{align*}
& C_\ell \mathbb{E}_0\left[ h(\xi(\ell), B_2(\ell), \cdots, B_n(\ell)) \cdot h(\xi(r), B_2(r), \cdots, B_n(r)) \right] \\
&= C_\ell \mathbb{E}_0\left[ \prod_{1 < j \le n} (\xi(\ell) - B_j(\ell))_+ \cdot \prod_{1 < j \le n} (\xi(r) - B_j(r))_+ \cdot \prod_{2 \le i < k \le n} (B_i(\ell) - B_k(\ell))_+ (B_i(r) - B_k(r))_+ \right] \\
&\le C_\ell \mathbb{E}_0 \Bigg[ \prod_{1 < j \le n} (\xi(\ell)_+ + B_j(\ell)_-) \prod_{1 < j \le n} (\xi(r)_+ + B_j(r)_-) \\
&\quad \times \prod_{2 \le i < k \le n} (B_i(\ell)_+ + B_k(\ell)_-) (B_i(r)_+ + B_k(r)_-) \Bigg] \\
&\le C_\ell \prod_{1 < j \le n} \left(\xi(\ell)_+ + \| B_j(\ell)_- \|_{n(n-1)} \right) \cdot \left(\xi(r)_+ + \| B_j(r)_- \|_{n(n-1)} \right) \\
&\quad \times \prod_{2 \le i < k \le n} \left( \| B_i(\ell)_- \|_{n(n-1)} + \| B_k(\ell)_- \|_{n(n-1)} \right) \cdot \left( \| B_i(r)_- \|_{n(n-1)} + \| B_k(r)_- \|_{n(n-1)} \right) \,,
\end{align*}
by generalised H\"{o}lder, where $\norm{\cdot}_{n(n-1)}$ denotes the $L^{n(n-1)}(\PP)$ norm. Now observe that for all $1\le j \le n$, 
\begin{equation*}
    \norm{B_j(\ell)_-}_{n(n-1)} \le cn\sqrt{\ell},
\end{equation*}
where $c>0$ is a universal constant, and similarly for $B(r)$ (which follows from the asymptotics of the moments of the gaussian distribution). Thus, we have the further estimate for the Radon-Nikodym  derivative  for $\mu$-a.a. paths $\xi$
\begin{equation*}
    \begin{array}{cc}
         & \le \displaystyle C_\ell \prod_{1 < j\le n}(\xi(\ell)_++cn\sqrt{\ell})\cdot (\xi(r)_++cn\sqrt{r})\cdot \prod_{2\le i < k\le n}(cn\sqrt{\ell}+cn\sqrt{\ell})\cdot(cn\sqrt{r}+cn\sqrt{r})\\
     & \le \frac{c^{n(n-1)}n^{n(n-1)}}{\prod_{j=1}^{n-1}j!}\cdot(\xi(\ell)_+/\sqrt{\ell}+1)^{n-1}\cdot (\xi(r)_+/\sqrt{\ell}+1)^{n-1}.
    \end{array}
\end{equation*}
\end{proof}

\begin{remark}
    These pathwise estimates also yield some uniform estimates for the Radon-Nikodym derivative processes of homogeneous BLPP against Brownian motion on compacts, see Section \ref{sec: uniform path blpp hom} in the Appendix.
\end{remark}

A simple calculation gives the prefactor in the pointwise bound of the Radon-Nikodym derivative in Proposition \ref{prop: top line pitman pointwise bound general case}
\[
 \frac{c^{n(n-1)}n^{n(n-1)}}{\prod_{j=1}^{n-1}j!} = O(\mathrm{e}^{dn^2})\,,
\]
for some $d>0$. It turns out that for large enough $p>1$, one can show that this growth is optimal, which is the content of the following theorem.

\begin{theorem}\label{thm: melon lb rn}
Fix $n\in \N$,  $0<\ell<r$ and let $H(\cdot)$ be the top line of an $n-$Dyson Brownian motion. Then the Radon-Nikodym derivative of the law of $H$ against the Wiener measure $\mu$ restricted to $[\ell,r]$ satisfies the following asymptotic behaviour for all $p>1$ sufficiently large 
\begin{equation*}
   \norm{\frac{\diff \mathrm{Law}H}{\diff \mu}}_{L^p(\mu)}\ge c_p \mathrm{e}^{d_p n^2}\,,\qquad n\ge 1\,.
\end{equation*}
for some $p$-dependent constants $c_p, d_p>0$.
\end{theorem}

\begin{proof}
    Now, arguing as in the proof of Proposition \ref{prop: top line pitman pointwise bound general case}, we can obtain that the Radon-Nikodym derivative of $H$ against standard Brownian motion on $[\ell,r]$ on paths $\xi$ by 
\begin{equation*}
    C_\ell\mathbb{E}_0\left[ h(\xi(\ell), B_2(\ell), \cdots, B_n(\ell))\cdot h(\xi(r), B_2(r), \cdots, B_n(r))\cdot \mathbf{1}(\mathrm{NoInt}([\ell, r]))(\xi, B_{2:n})\right]\,,
\end{equation*}
where
\begin{equation*}
    C_\ell = \left[\ell^{n(n-1)/2}\prod_{j=1}^{n-1}j!\right]^{-1}\,.
\end{equation*}
This gives the norm estimates
\[
\norm{\frac{\diff \mathrm{Law}H^n}{\diff \mu}}_{L^p(\mu)} =  C_\ell\mathbb{E}_0\left[ h^p(B(\ell))\cdot h^p(B(r))\cdot \mathbf{1}(\mathrm{NoInt}([\ell, r])(B))\right]^{1/p}\,,\qquad p>1\,.
\]
Now, conditioning on $B(\ell)$ and using the Karlin-McGregor formula (see \cite{karlin1959coincidence}) and a Monotone class argument allows us to express the above as
\begin{align}\label{eq: L^p expression}
    C_\ell\left(\displaystyle\int_{\R^n_<\times \R^n_<} h^p(x)h^p(y)\mathrm{det}\big(\mathrm{e}^{-\frac{(x_i-y_j)^2}{4(r-\ell)}}\big)_{1\le i , j \le n}\mathrm{e}^{-\frac{\norm{x}^2}{4\ell}}\diff x\diff y\right)^{1/p}\,.
\end{align}
Now, the Harish-Chandra formula from \cite{harish1957differential} and the multi-linearity of the determinant give
\begin{align*}
\det\biggl(\mathrm{e}^{-\frac{(x_i - y_j)^2}{4(r - \ell)}}\biggr)_{1 \le i,j \le n} 
= & \frac{
    \mathrm{e}^{-\frac{\|x\|^2}{4(r - \ell)}} \mathrm{e}^{-\frac{\|y\|^2}{4(r - \ell)}} 
}{
    (4(r - \ell))^{\frac{n(n-1)}{2}} \prod_{j=1}^{n-1} j!
} \\
& \times \int_{\mathrm{U}(m)} 
    \mathrm{e}^{\frac{1}{4(r - \ell)} \mathrm{tr} \bigl( \mathrm{diag}(\underline{x}) U \mathrm{diag}(\underline{y}) U^* \bigr)} 
    \, \mu_{\mathrm{U}(m)}(\mathrm{d}U) \, h(x) h(y) \,,
\end{align*}

where $\mu_{\mathrm{U}(m)}$ is the normalised Haar measure on the group $\mathrm{U}(m)$ of unitary transformations on $\C^m$. Observe that for any $U\in \mathrm{U}(m)$ we can estimate using Cauchy-Schwarz
 \begin{equation*}
     |\mathrm{tr}(\mathrm{diag}(\underline{x})U\mathrm{diag}(\underline{y})U^*)|\le \norm{x}_2\cdot\norm{y}_2\le \frac{1}{2}\norm{x}_2^2+ \frac{1}{2}\norm{y}_2^2\,.
 \end{equation*}
 Thus, we can estimate equation (\ref{eq: L^p expression}) as a product of two integrals (where we have expanded the domains of integral to $\R^n$ by symmetry)
 \[
(\ref{eq: L^p expression}) \ge \frac{c_{\ell,r} \mathrm{e}^{-d_{\ell,r} n^2}}{\left(\prod_{j=1}^{n}j!\right)^2} \left(\displaystyle\int_{\R^n} \Delta^p(x)\mathrm{e}^{-c'_{\ell,r}\norm{x}^2}\diff x\right)^{1/p}\cdot \left(\displaystyle\int_{\R^n} \Delta^p(y)\mathrm{e}^{-c''_{\ell,r}\norm{y}^2}\diff y\right)^{1/p}\,,
\]
where for $x\in \R^n$, $\Delta(x) = \prod_{1\le i < j\le n}|x_i-x_j|$. Observe that the RHS above is the product of two Mehta integrals, see \cite[Eq. (1.1)]{forrester2008importance}, which can be evaluated exactly in terms of the Gamma function resulting in the estimates
\[
(\ref{eq: L^p expression}) \ge c_{\ell,r, p} \mathrm{e}^{-d_{\ell,r} n^2}\frac{\left(\prod_{j=1}^{n}\Gamma\left(1+j\frac{p}{2}\right)\right)^{2/p}}{\left(\prod_{j=1}^{n}j!\right)^2}\ge c_{\ell,r, p} \mathrm{e}^{d'n^2\log p -d_{\ell,r} n^2}\,,
\]
by the asymptotics of the Gamma function $\Gamma(\cdot)$, for some positive constants $c_{\ell,r, p}, c'_{\ell,r}, c''_{\ell,r},  d_{\ell,r}, d'$, which concludes the proof.
\end{proof}

\begin{remark}
    Note that Proposition \ref{prop: 3} in Section \ref{sec: inc reg} in the Appendix shows that in principle the growth for the Radon-Nikodym derivative of spatial increments of the top line of Dyson Brownian motion can be improved, though at present it is not clear how to translate the estimates thus obtained to this context.
\end{remark}

\section{Inhomogeneous Brownian LPP}\label{sec: inhom blpp}
In this section, we will consider the operation of iterating Skorokhod reflections on a family of independent Brownian motions $B_1, \cdots, B_n$ with non-decreasing initial conditions $g_1 \ge g_2 \ge \cdots \ge g_n$ (Brownian TASEP). In particular, we will uncover the Markovian structure of Brownian TASEP in Proposition \ref{prop: generalised Markov} and then in Theorem \ref{thm: semi-mg decomp blpp} uncover its semi-martingale structure by showing that it solves a particular system of singular stochastic differential equations.

Now let $n\in \N$ and consider the Brownian LPP (or BLPP) $(H_{k})_{k=1}^n$ as in Proposition \ref{prop: top line pitman pointwise bound general case}. A simple induction unravelling the definition of the Pitman transform (\ref{eq: pitmantrans}) gives 
\[
(H_{n-k+1})_{k=1}^n(\cdot)\stackrel{d}{=} (B[(0, 1)\to (\cdot, k)])_{k=1}^n
\]
as processes on $[0,\infty)$, where $B$ is a family of independent Brownian motions starting from the origin. Now, by the metric composition law enjoyed by LPP, we obtain 
\[
B[(0, 1)\to (\cdot+1, k)] = \max_{1\le \ell \le k}(B[(0, 1)\to (1, \ell )] + B[(1,\ell)\to (\cdot+1, k)])\quad \mathrm{on}\quad [0,\infty)\,.
\]
Thus, by independence of increments of Brownian motion, conditioning on $(B[(0, 1)\to (1, \ell )])_{1\le \ell\le k}\equiv(b_\ell)_{1\le \ell\le k}$, we are hinted to the Markovian structure of $(H_k)_{k=1}^n$ (with respect to its own filtration, to be made precise below) and are led to the following definition of Brownian LPP with \textbf{inhomogeneous} boundary data, as a way of studying its regular conditional distributions, or Markov kernel.

\begin{definition}(Brownian TASEP)\label{def: brownian tasep}
    Fix $m\ge 1$, $B_1, \cdots, B_m$ be independent Brownian motions starting from the origin, $\underline{g} = (g_\ell)_{\ell =1}^m \in\R^m_{\ge}$ and define the \textbf{Brownian TASEP} started from initial data $(g_\ell)_{\ell =1}^m$ as the random ensemble
     $$\displaystyle\max_{k\le \ell \le m}(g_{\ell} + B[(0,\ell)\to (y, k)])\,, \qquad  y\in [0,\infty)\,, k\in\llbracket 1, m \rrbracket$$ where $B[(0,\cdot')\to (\cdot, 1)]$ denotes Brownian LPP. 

     Furthermore, the process 
     \[
     \displaystyle\max_{1\le \ell \le m}(g_{\ell} + B[(0,\ell)\to (y, 1)])\,, \qquad  y\in [0,\infty)
     \]
     is called \textbf{inhomogeneous BLPP} started from initial data $\underline{g}$.
\end{definition}

We start with a lemma that gives a deterministic procedure for constructing Brownian TASEP of any depth in terms of iterated Skorokhod reflections of independent Brownian motions. This will be useful in obtaining a better understanding of its structural properties.

\begin{lemma}\label{lemma: iterated skorokhod}
    Let $B = (B_1, B_2, \cdots)\in C^{\N}$ be an ensemble of independent rate two Brownian motions starting from the origin. Fix $m\ge 2$ and $\underline{g}=(g_1,g_2,\cdots, g_m)\in\R^m_\ge$ and let $H(\cdot)$ denote an inhomogeneous Brownian LPP started from initial data $\underline{g}=(g_1,g_2,\cdots, g_m)\in\R^m_\ge$. Then, $H$ satisfies
    \[
    H = W\mathcal{L}_1(y)\qquad \mathrm{ for\quad all}\, y\ge 0\,,
    \]
    where $\underline{g}' = (g_2,\cdots, g_m)\in\R^{m-1}$, $\mathcal{L} = (g_1+B_1, H')$ and for all $y\ge 0$,
    \[
    H'(y) \coloneqq \displaystyle\max_{2\leq \ell\leq m}(g_\ell+B[(0,\ell)\to(y,2)]) \,.
    \] 
    In other words, $H$ is the top line of the melon of a Brownian motion starting from $g_1$ and an inhomogeneous BLPP starting from data $\underline{g}'$.
\end{lemma} 
\begin{proof}
    Straightforward application of Metric composition law for last passage values, Lemma \ref{Lemma: Metric Composition}.
\end{proof}

We obtain some elementary structural properties of the laws of Brownian TASEP and in doing so, obtain a natural generalisation of Lemma $4.2$ in \cite{Sarkar2021Brownian} by showing that Brownian TASEP enjoys a description as a Markov process.
\begin{proposition}\label{prop: generalised Markov}
     Fix $m\ge 1$, a sequence $(g_\ell)_{\ell =1}^m \in\R^m_{\ge}$ and let $(H_1, H_2, \cdots, H_m)(\cdot)$ denote Brownian TASEP starting from initial data $\underline{g}$. Then, we have for all $A\subseteq C_{*, *}^m ([x,y])$ Borel measurable, the following
\begin{equation*}
\begin{array}{cc}
     &\mathbb{E}\left[ \mathbf{1}((H_1, \cdots, H_{m})\in A) | (H_1, \cdots, H_m)|_{[0,x]}\right]\\ &=  \mathbb{E}\left[ \mathbf{1}((H_1, \cdots, H_{m})\in A) | (H_1(x), \cdots, H_m(x))\right].
\end{array}
\end{equation*}
In other words, the process $(H_k)_{k=1}^m$ is a continuous Markov process with respect to its own filtration on $[0,\infty)$.
\end{proposition}
\begin{proof}
    We proceed via induction.
    
     \underline{$m=1$:} This case follows from the Markov property of Brownian motion and the fact that $H_1 = g_1 + B_1$.
    
    \underline{$m\ge 2$:} Suppose now the claim holds for some $m\ge 2$. Now we wish to show it for $m+1$. That is fix $(g_\ell)_{\ell =1}^{m+1} \in\R^{m+1}_{\ge}$ and let 
     let $(H_1, H_2, \cdots, H_m)(\cdot)$ denote Brownian TASEP starting from initial data $\underline{g}$. Observe that from Lemma \ref{lemma: iterated skorokhod}, we can express 
    \[
    H_{1}(y) = W\hat{\mathcal{L}}_1(y)\,,
    \]
where $\hat{\mathcal{L}} = (g_1+B_1, H_2)$ and $B_1$ is a rate two Brownian motion starting from the origin independent of $H_2 = \displaystyle\max_{2\le \ell \le m+1}(g_{\ell} + B[(0,\ell)\to (\cdot, 2)])$. Now, from the formula of the top line of the Pitman transform we have
\begin{equation*}
    \begin{array}{cc}
        &H_{1}(z) = B_1(z)-B_1(x)\\
        &+\displaystyle\max\left(\max_{x\leq r\leq z}(H_2(r)-B_1(r)+B_1(x)),H_{1}(x)\right), \quad z\in [x,y].
    \end{array}
\end{equation*}
Hence, we can express 
\begin{equation*}
    H_{1}|_{[x,y]}(\cdot) = F(H_{1}(x), H_2|_{[x,y]}, (B_1(\cdot)-B_1(x))|_{[x, y]})(\cdot),
\end{equation*}
where $F$ is a measurable functional and $(B_1(\cdot)-B_1(x))|_{[x, y]}$ is independent of $H_2$ and $H_{1}|_{[0,x]}$.

Now, fix $A\subseteq C_{*, *} ([x,y])$ Borel measurable, then by the induction hypothesis applied to $H_{2:m+1}$, we have
\begin{equation*}
\begin{array}{cc}
     &\mathbb{E}\left[ \mathbf{1}(H_{1}\in A) \Big| H_{1:m+1}|_{[0,x]}\right]\\ &=  \mathbb{E}\left[ \mathbf{1}(F(H_1(x), H_2|_{[x,y]}, (B_1(\cdot)-B_1(x))|_{[x, y]})(\cdot)\in A) \Big| H_{1:m+1}|_{[0,x]}\right]\\
     & =  \mathbb{E}\left[ \mathbf{1}(H_1\in A) \Big| H_{1:m+1}(x)\right],
\end{array}
\end{equation*}
by disintegration of measures, see \cite{kallenberg1997foundations}, and independence, and a monotone class argument now allows us to complete the proof.
\end{proof}

\begin{remark}\label{rmk: markov generator}
    Note that the Markov generator does not depend on the choice of sequence $(g_\ell)_{\ell=1}^\infty$ and that it is a homogeneous-time Markov process. It turns out that it can be computed explicitly using determinants and so called diffusion interlacing arguments, as will be done in the forthcoming sub-sections.
\end{remark}

\subsection{Semi-martingale decomposition of Brownian TASEP}

The construction of Brownian TASEP by inductively iterating the Skorokhod reflection on a family of Brownian motions suggests it has a natural decomposition as a semi-martingale. That this is so is the content of the following theorem which follows from an inductive argument and the deterministic Lemma 2.1 in \cite{revuz2013continuous}. 

\begin{theorem}(Semi-martingale decomposition of Brownian TASEP)\label{thm: semi-mg decomp blpp}
     Fix $m\ge 1$, a sequence $(g_\ell)_{\ell =1}^m \in\R^m_{\ge}$, a family of independent Brownian motions $B_1, \cdots, B_m$ and let $(H_1, H_2, \cdots, H_m)(\cdot)$ be a Brownian TASEP started from initial data $\underline{g}$. Then, there exist continuous non-decreasing processes $(\alpha^k(\cdot))_{k=2}^m$ such that for all $2\le k \le m$, the Stieltjes measure $\diff \alpha^{k-1}$ is almost surely supported on the set $\{H_{k-1} = H_k\}$ and
     \begin{equation*}
     \begin{array}{cc}
          & H_{k-1}(t) = B_{k-1}(t) + \alpha^{k-1}(t) \\
          & = B_{k-1}(t)+ \displaystyle\int_{(0,t]}\mathbf{1}([T_{k-1},\infty))(s)\diff \alpha_{k-1}(s),\, \text{a.s. for all }\, t\ge 0, 
     \end{array}
     \end{equation*}
     where $T_{k-1} = \inf\{t\ge 0:  H_{k-1}(t) =  H_{k}(t)\}$ and is almost surely positive. In other words, for all $1\le k \le m$, $H_k$ is a semi-martingale.
\end{theorem}

\begin{proof}
    Let $B_1, B_2$ be two continuous-time processes such that $B_1(0)\ge B_2(0)$. Then, the Skorokhod reflection lemma, Lemma 2.1 in \cite{revuz2013continuous} gives that the process
    \begin{equation*}
        Z(t) = B_1(t)-B_2(t) + \alpha(t), \, t\ge 0
    \end{equation*}
    is positive and $\alpha(t) = \sup_{s\le t}((B_2(s)-B_1(s))\lor 0)$ is continuous, non-decreasing and the (random) measure $\diff \alpha$ induced by the increments of $\alpha$ is almost surely supported on the set $\{Z = 0\} = \{WB_1 = B_2\}$ upon noticing that 
    \begin{equation*}
        \sup_{s\le t}((B_2(s)-B_1(s))\lor 0) =  \sup_{s\le t}(B_2(s)-B_1(s))\lor 0,\, \text{for all}\, t\ge 0.
    \end{equation*}
    Thus, we have the semi-martingale decomposition of the top line of the Pitman transform 
    \begin{equation*}
          WB_1(t) = B_1 + \alpha(t) = B_{k+1}(t)+ \displaystyle\int_{(0,t]}\mathbf{1}([T,\infty))(s)\diff \alpha(s),\, \text{a.s. for all} \, t\ge 0, 
     \end{equation*}
     where $T = \inf\{t\ge 0:  WB_1 = B_2\}$ and is almost surely positive. A quick induction, replacing $B_1$ with $B_{k-1}$, $B_2$ with $H_{k}$ and $\alpha$ with $\alpha_{k-1}$ in the inductive step allows us to conclude.
\end{proof}

\section{Diffusion interlacing and Brownian motion in the Gelfand-Tsetslin cone}\label{sec: BM gelfand-tsetslin cone}
Having obtained some elementary structural properties of the Brownian TASEP, we aim to compute explicitly its transition probability densities. Remark \ref{rmk: markov generator} would then give an explicit representation of the Radon-Nikodym derivative of Brownian TASEP as a ratio of the aforementioned transition densities. This would then allow us transfer the pathwise estimates for Brownian LPP to Brownian TASEP. To this end, we now introduce the Warren process and establish its close connection to Brownian TASEP using Theorem \ref{thm: semi-mg decomp blpp}. 

Let ${\mathbf K}$ be the cone of  points ${\mathbf x}= 
\bigl(x^{1},x^{2}, \ldots x^{n}\bigr)$ with  $x^{k}=\bigl(x^{k}_1,x^{k}_2, \ldots ,x^{k}_k\bigr) \in {\R}^k$ satisfying 
the inequalities 
\begin{equation*}
 x^{k+1}_{i} \leq x^{k}_i \leq x^{k+1}_{i+1}.
\end{equation*}
${\mathbf K}$ is sometimes called the Gelfand-Tsetlin cone, and arises in representation theory.
We will consider a process ${\mathbf X}(t)=\bigl( X^{1}(t), X^{2}(t), \ldots X^{n}(t)\bigr)$ taking values in ${\mathbf K}$ so that 
\begin{equation*}
\label{gts:sde}
X^{k}_i(t)= x^{k}_i+\gamma^{k}_i(t)+ L^{k,-}_i(t)-L^{k,+}_i(t),
\end{equation*}
where $\bigl(\gamma^k_i(t); t \geq 0 \bigr)$ for $1\leq k \leq n, 1 \leq i \leq k $ are independent Brownian motions, and 
$\bigl(L^{k,+}_i(t); t \geq 0 \bigr)$ and  $\bigl(L^{k,1}_i(t); t \geq 0 \bigr)$ are continuous, increasing processes growing only when 
$X^k_i(t)=X^{k-1}_i(t)$ and $X^k_i(t)=X^{k-1}_{i-1}(t)$ respectively,  the exceptional cases  $L^{k,+}_k(t)$ and $L^{k,-}_1(t)$  being 
identically zero for all $k$.

Now, observe that upon extracting the `diagonal' subprocess $(X^1_1, \cdots, X^n_n)$, the semi-martingale decomposition of Brownian TASEP just established in Proposition \ref{thm: semi-mg decomp blpp} and the deterministic Skorokhod Lemma, \cite[Lemma 2.1]{revuz2013continuous}, give through a quick induction argument the following proposition. 

\begin{proposition}\label{prop: interlacing inhm blpp equality}
The process $(X^n_n, \cdots, X^1_1)(\cdot)$ as the same law on paths as a Brownian TASEP started from the origin.
\end{proposition}

This identification allows one to use the following lemma from \cite{warren2007dyson} to obtain an explicit form of the one point marginals of homogeneous BLPP, which is the content of the following lemma.

\begin{lemma}(\cite[Proposition 8]{warren2007dyson})\label{lemma: density Warren} Fix $r>0$, $m\in \N$ and let $(H_1, H_2, \cdots, H_m)(\cdot)$ be a Brownian TASEP started from the origin. Then, the density of  $q_r$ of $(H_1(r), H_{2}(r), \cdots, H_m (r))$ on $\R^m_\ge$ against the Lebesgue measure is given by
\begin{equation*}
    q_r(x_1, \cdots, x_m) = \det \bigl\{ \Phi_r^{(i-j)}(x_{m-j+1}); 1 \leq i,j \leq m \bigr\}.
\end{equation*}
where for $\Phi^{(m)}_t$ denotes the $m$th order ($m\ge 1$) iterated integral of the rate two Gaussian density $\varphi_t$ $t>0$,
\begin{equation*}
\begin{array}{cc}
      \Phi^{(m)}_t(y) & = \displaystyle\int_{-\infty}^y \frac{(y-x)^{m-1}}{(m-1)!} \varphi_t(x)dx = \int_0^\infty \frac{z^{m-1}}{(m-1)!} \varphi_t(z-y)dz \\
     & = \displaystyle\varphi_t(y)\int_0^\infty \frac{z^{m-1}}{(m-1)!} \mathrm{e}^{-\frac{z^2}{4t}}\mathrm{e}^{\frac{yz}{2t}}dz,
\end{array}
\end{equation*}
and for $k \geq 0 $ $\Phi^{(-k)}_t$ denotes the $k$th order derivative of $\varphi_t$, $t > 0$, which is equal to
\begin{equation*}
    \Phi^{(-k)}_t(y) = (-1)^k \left(\frac{1}{4t}\right)^{\frac{k}{2}}H_k\left(\frac{y}{\sqrt{4t}}\right) \varphi_t (y)
\end{equation*}
where $(H_n)_{n\in \N}$ are the Hermite polynomials as defined in \cite{weisstein2002hermite}.
\end{lemma}

\begin{remark}
    Observe that $\Phi^{(m)}_t$, $m\in \Z, t>0$ are real analytic. 
\end{remark}

The following lemma finally allows one to obtain an explicit form of transition kernel of Brownian TASEP by a conditioning argument, that is, the essential uniqueness of regular conditional distributions and the metric composition law enjoyed by LPP.

\begin{lemma}\label{lemma: Warren Markov density}
Fix $r>0$, Fix $m\ge 1$, a sequence $\underline{b} = (b_\ell)_{\ell =1}^m \in\R^m_{\ge}$, let $(H_1, H_2, \cdots, H_m)(\cdot)$ be a Brownian TASEP started from initial data $\underline{b}$ and  let $(G_1, G_2, \cdots, G_m)(\cdot)$ be a Brownian TASEP started from the origin. Then, we have that $\mathrm{Law}(H_\ell)_{\ell =1}^m(r)$ is a version of the regular conditional distribution\\
\[\mathrm{Law}\left((G_\ell)_{\ell =1}^m(r + 1)\right. \;\mathrm{conditioned\quad on }\;\left.(G_\ell)_{\ell =1}^m(1) = (b_\ell)_{\ell =1}^m\right)\] 
and in particular 
\begin{equation*}
     \PP((H_1, H_2, \cdots, H_m)(r)\in A)  = \displaystyle\int_{A} q_r(x_1, \cdots, x_m; b_1, \cdots, b_m)\lambda(\diff x_1\cdots \diff x_m), \quad A\in \mathcal{B}(\R^m_\ge)
\end{equation*}
for all $r>0$, where $q_r$ is as in Lemma \ref{lemma: density Warren}.
\end{lemma}
\begin{proof}
    Observe that one can express $H_k(y) = \max_{m-k+1\le \ell \le m}(b_{k} + B[(0,\ell)\to (y, k)])$, for $y\in [0,\infty)$, $1\le k \le m$, where $B[(\cdot, \cdot)\to (\cdot, \cdot)]$ denotes Brownian LPP. By translation invariance we have that $H_k(y) \stackrel{d}{=}\max_{m-k+1\le \ell \le m}(b_{\ell} + B[(1,\ell)\to (y+1, k)])$, $y\ge 0$. If we replace $b_{\ell}$ with $B[(0, 0)\to (1, m-\ell+1)]$, by metric composition, we obtain that \[
    \max_{m-k+1\le \ell \le m}(B[(0, 0)\to (1, \ell )] + B[(1,\ell)\to (y+1, k)]) = B[(0, 0)\to (y+1, k)] \stackrel{d}{=}G(y+1).
    \]
    Thus, by independence of increments of Brownian motion, we see that for any $r>0$ and choice of $(b_\ell)_{1\le \ell\le m}$, the law of $(H_\ell)_{1}^m(r)$ is a version of the regular conditional distribution of $\mathrm{Law}\left((G_\ell)_{\ell =1}^m(r + 1)\,\text{conditioned}\, (G_\ell)_{\ell =1}^m(1)\right)$.
    
    From Lemma \ref{lemma: density Warren}, we have that $(G_\ell)_{\ell =1}^m$ is a Markov process with transition density $q_r$ of $(H_k)_{k=1}^{m}(r) $ on $R^m_\ge \times \R^m_\ge$. Hence, we have by uniqueness of regular conditional distributions on Polish spaces, see \cite{kallenberg1997foundations}, that for almost all with respect to the law of $(G_k)_{k=1}^{m}(1)$ $\stackrel{d}{=}(b_\ell)_{\ell = 1}^m\in \R^m_\ge $
    \begin{equation*}\label{eq: markov reg cond}
        \PP((H_\ell)_{\ell=1}^m(r)\in A) =  \displaystyle\int_A q_r(x_1, \cdots, x_m; b_1, \cdots, b_m)\lambda(\diff x_1\cdots \diff x_m).
    \end{equation*}
    for all Borel $A\subseteq \R^m_\ge$. We now claim that this equality in distribution holds true for all $(b_\ell)_{\ell = 1}^m\in \R^m_\ge$. Indeed, observe by Proposition \ref{prop: support density warren} that the law of $(G_\ell)_{\ell = 1}^m$ is mutually absolutely continuous with respect to the Lebesgue measure and so the above equality holds for almost all $(b_\ell)_{\ell = 1}^m\in \R^m_\ge$ with respect to the Lebesgue measure. Now, fix arbitrary $(b_\ell)_{\ell = 1}^m\in \R^m_\ge$ and by density obtain a sequence $((b^n_\ell)_{\ell = 1}^m)_{n\in\N}\subseteq \R^m_\ge, n\in \N$ converging to $(b_\ell)_{\ell = 1}^m\in \R^m_\ge$ such that \eqref{eq: markov reg cond} holds for all $n\in \N$. To show \eqref{eq: markov reg cond} holds for the limit, simply observe that the Markov density $q_r$ is continuous in all of its arguments, use dominated convergence and the convergence in distribution 
    \begin{equation*}
        \max_{m-k+1\le \ell \le m}(b^n_{k} + B[(0,\ell)\to (r, k)]) \stackrel{d}{\longrightarrow}  \max_{m-k+1\le \ell \le m}(b_{k} + B[(0,\ell)\to (r, k)]),\quad n\to \infty
    \end{equation*}
    which holds by continuity; conclude noting that limits in distribution on Polish spaces are unique.
\end{proof}

\begin{remark}
    The same argument shows that $\mathrm{Law}(H_\ell)_{\ell =1}^m(\cdot)$ is a version of the regular conditional distribution $\mathrm{Law}\left((G_\ell)_{\ell =1}^m(\cdot + 1)\, \mathrm{conditioned\, on }\, (G_\ell)_{\ell =1}^m(1)\right)$ for all understood as random continuous functions on $[0,\infty)$. Particularly, as a Markov process it has the Markov transition densities $q_r, r>0$ as in Lemma \ref{lemma: density Warren}. Alternatively, we could have used the semi-martingale decomposition of the above ensemble in terms of independent Brownian motions and local times and use the argument in \cite[Proposition 8]{warren2007dyson}. For an illustration, see Figure \ref{fig: Brownian LPP conditioned}.
\end{remark}

\begin{figure}
    \centering
    \includegraphics[width=0.6\linewidth]{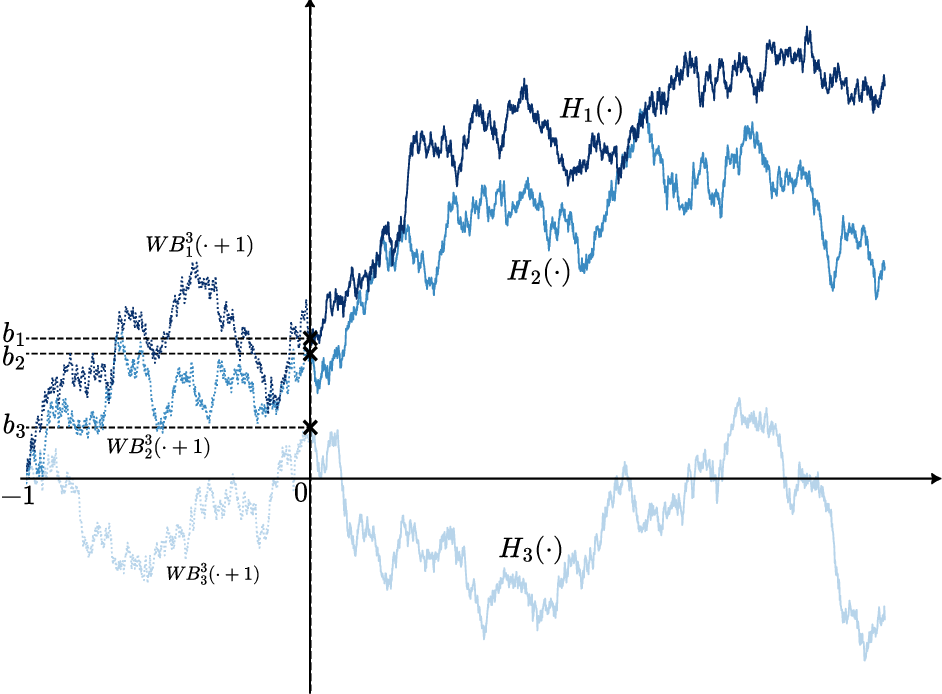}
    \caption{Here, the contents of Lemma \ref{lemma: Warren Markov density} with $m=3$ are schematically depicted. It states that the law of the inhomogeneous Brownian LPP on the positive reals $(H_\ell)_{\ell =1}^3(\cdot)$ is a version of the regular conditional distribution of the Brownian melon $WB^3(\cdot + 1)$ conditioned on $WB^3(1) = (b_\ell)_{\ell =1}^3$. Thus, if one `traces back' the inhomogeneous Brownian LPP by one unit to the left one recovers the Brownian $3$-melon $WB^3(\cdot + 1)$.}
    \label{fig: Brownian LPP conditioned}
\end{figure}

Thus, from the above propositions, we are able to compute the density of the non-homogeneous ensemble in the following lemma.

\begin{lemma}
    Fix $r>0$, $m\ge 1$, a sequence $\underline{b}= (b_\ell)_{\ell =1}^m \in\R^m_{\ge}$ and let $(H_1, H_2, \cdots, H_m)(\cdot)$ be a Brownian TASEP started from initial data $\underline{b}$. Then, the density  $q_r(x_1, \cdots, x_m)$ of $(H_k)_{k=1}^{m}(r) $ on $\R^m_\ge $ against the Lebesgue measure is given by
\begin{equation*}
    q_r(x_1, \cdots, x_m; b_1, \cdots, b_m) = \det \bigl\{ \Phi_r^{(j-i)}(x_{m-i+1}-b_{m-j+1}); 1 \leq i,j \leq m \bigr\},\, x\in \R^m_>
\end{equation*}
with $\Phi^{(m)}_r$, $\Phi^{(-m)}_r$, $m\in \N, r>0$ as in Lemma \ref{lemma: density Warren}.
\end{lemma}

We also record the following proposition which gives mutual absolute continuity of the Lebesgue densities $q_r$, being real analytic in several variables.

\begin{proposition}\label{prop: support density warren}
    Fix $r>0$, $m\ge 1$, a sequence $\underline{b}= (b_\ell)_{\ell =1}^m \in\R^m_{\ge}$ and let $(H_1, H_2, \cdots, H_m)(\cdot)$ be a Brownian TASEP started from initial data $\underline{b}$. Then, the support is
\begin{equation*}
    \supp (\nu_r) = \overline{\{\underline{x}\in \R^m_>: \nu_r(\mathrm{B}_{\underline{x}, \epsilon})>0\,, \forall \epsilon > 0 \}} = \R^m_\ge\,,
\end{equation*}
where for $x\in \R^m_>$ and $\epsilon > 0$, $B_{x, \epsilon}$ denotes the open ball of radius $\epsilon$ centred at $x$. 

Furthermore, the law of $(H_\ell)_{\ell = 1}^m$ is mutually absolutely continuous with respect to (equivalent to) the Lebesgue measure on $\R^m_\ge$.
\end{proposition}

\begin{proof}
 Observe that one can express $H_k(y) = B[(0,m)\to (y, m-k+1)]$, for $y\in [0,\infty)$, $1\le k \le m$, where $B[(\cdot, \cdot)\to (\cdot, \cdot)]$ denotes Brownian LPP. Furthermore, fix $r>0$ and observe that LPP is continuous in the product uniform topology on $C^n_{*,*}([0,r])$. It is also a well known fact that with positive probability a Brownian motion starting from zero can approximate any continuous function starting from the origin on $[0,r]$. Thus, by independence, we have for all $\delta > 0$ that
 \begin{equation*}
     \PP(\norm{B_1(t)-a_1t}_{\infty,[0,r]}<\delta,\cdots, \norm{B_m(t)-a_mt}_{\infty,[0,r]} <\delta) >0
 \end{equation*}
 where $(a_\ell)_{\ell = 1}^m\in \R^m_>$. Now, observe that with the ensemble $L: \llbracket 1, m \rrbracket \times [0,\infty):(n,x)\mapsto L(n, x) = a_{m-n+1}x$, the last passage percolation is simply $L[(0,m)\to (r, k)] = r a_k \in \R^m_>$, by the ordering of the coefficients. The above combined easily imply that 
\begin{equation*}
    \R^m_{<}\subseteq \supp (\nu_r) = \overline{\{\underline{x}\in \R^m_>: \nu_r(\mathrm{B}_{\underline{x}, \epsilon})>0 \forall \epsilon > 0 \}} = \R^m_\ge.
\end{equation*}
 Now, $ \overline{\{q>0\}} = \supp (\nu_r) = \R^m_\ge$ and we have that $q\not\equiv 0$ on $\R^m_\ge$; since the density $q_r(x_1, \cdots, x_m)$ is real analytic (in several variables), its zero set $\{q = 0\}$ must have zero Lebesgue measure on $\R^m_\ge$. We thus conclude that the law of $(H_\ell)_{\ell = 1}^m$ is mutually absolutely continuous (that is, equivalent) with respect to the Lebesgue measure on $\R^m_\ge$, as required.
\end{proof}

We thus have the following theorem which gives an expression for the Radon-Nikodym  derivative of the inhomogeneous ensemble at a point and the homogeneous one at zero.
\begin{theorem}\label{thm: Radon-Nikodym  density inhom BLPP}
Fix $r>0$, $m\ge 1$, a sequence $\underline{b}= (b_\ell)_{\ell =1}^m \in\R^m_{\ge}$  let $(H_1, H_2, \cdots, H_m)(\cdot)$ be a Brownian TASEP started from initial data $\underline{b}$ and  let $(G_1, G_2, \cdots, G_m)(\cdot)$ be a Brownian TASEP started from the origin. Then, the Radon-Nikodym  derivative between the distributions of two processes at $r>0$ on $\R^m_\ge$ is almost surely given by
\begin{equation*}
      \frac{\diff \mathrm{Law}(H_\ell)_{\ell =1}^m(r)}{\diff\mathrm{Law}(G_\ell)_{\ell =1}^m(r)}(x_1, \cdots, x_m)  = \frac{\det \bigl\{ \Phi_r^{(i-j)}(x_{m-j+1}-b_{m-i+1}); 1 \leq i,j \leq m \bigr\}}{\det \bigl\{ \Phi_r^{(i-j)}(x_{m-j+1}); 1 \leq i,j \leq m \bigr\}}
\end{equation*}
\begin{equation}\label{eq: RN ratio}
     = \left(\displaystyle\prod_{j=1}^n\mathrm{e}^{-\frac{b_j^2}{4r}} \right)\frac{\det \bigl\{ \mathrm{e}^{\frac{x_{m-j+1}b_{m-i+1}}{2r}}{F_r^{(i-j)}}(x_{m-j+1}-b_{m-i+1}); 1 \leq i,j \leq m \bigr\}}{\det \bigl\{ F_r^{(i-j)}(x_{m-j+1}); 1 \leq i,j \leq m \bigr\}}
    ,\, x\in \R^m_>
\end{equation}
with $H_m, \Phi^{(m)}_r$, $\Phi^{(-m)}_r$, $m\in \N, r>0$ as in Lemma \ref{lemma: density Warren} and 
\begin{equation*}
    F_r^{k}(y) = \begin{cases}
        \int_0^\infty \frac{z^{k-1}}{(k-1)!} \mathrm{e}^{-\frac{z^2}{4r}}\mathrm{e}^{\frac{yz}{2r}}\diff z, \quad & k\ge 1\\
        (-1)^k \left(\frac{1}{4r}\right)^{-\frac{k}{2}}H_{-k}\left(\frac{y}{\sqrt{4r}}\right), \quad & k \le 0.
    \end{cases}
\end{equation*}
\end{theorem}

\begin{proof}
    Using Proposition \ref{prop: support density warren}, we see that the law of $(H_\ell)_{\ell =1}^m(r)$ for all choices of non-decreasing $(b_\ell)_{\ell =1}^m$ are mutually absolutely continuous with respect to the Lebesgue measure on $\R^m_>$. Hence, one can take their pointwise ratios as the Radon-Nikodym derivative of the Markov processes defined above to conclude.
\end{proof}

\section{Inhomogeneous BLPP: Radon-Nikodym derivative estimates}\label{sec: BM tasep rn estimates}

In this section, the main result of this paper providing pathwise and $L^{\infty-}$ estimates for the Radon-Nikodym derivative of inhomogeneous BLPP with respect to Brownian motion on compacts, namely Theorem \ref{thm: tasep bm abs cont}, is established. Technical integral estimates have been relegated to the Appendix (Section \ref{app: int est}). 

Recall that in Theorem \ref{thm: Radon-Nikodym  density inhom BLPP}, we obtained an explicit form of the Radon-Nikodym derivative of Brownian TASEP with respect to its homogeneous counterpart in terms of the transition densities given in Lemma \ref{lemma: density Warren}. Hence, to obtain pathwise Radon-Nikodym derivative bounds, we now estimate the transition densities of the Warren process. In \cite{warren2007dyson}, the author deduces that the distribution of the diagonal section of
${\mathbf X}(t)$ given by the SDEs (\ref{gts:sde}) has the density
\begin{equation*}
\label{entrance:gts}
\boldsymbol{\mu}^n_t({\mathbf x}) = (2\pi)^{-n/2} (2t)^{-n^2/2} \exp \left\{ -\sum_{i} (x^n_i)^2/(4t)\right\} \left\{ \displaystyle\prod_{i<j} (x^n_j-x^n_i) 
\right\},\, t > 0
\end{equation*}
with respect to the Lebesgue measure on ${\mathbf K}$. We thus obtain by taking the marginal, the following proposition. 

\begin{proposition}\label{prop: lb density 1}
    Fix $\ell>0$, $n\in \N $ and let $q_\ell (\cdot; \underline{0})$ be as in Lemma \ref{lemma: density Warren}. Then we have for all $(x^1_1, \cdots, x^n_n)\in \R^n_{\le}$
    \begin{equation*}
        q_{\ell}(x^n_n, \cdots, x^1_1; \underline{0}) =  \displaystyle\int_{\boldsymbol{K}} \boldsymbol{\mu}^n_{r}(\underline{x})\diff\, \left(x^k_j\, \mathrm{where}\, \left.\begin{cases}
       &2\le k \le n\\
       &1\le j \le k-1
   \end{cases}\right\} \right)
\end{equation*}
where 
\begin{equation*}
\boldsymbol{\mu}^n_{r}(x_n, x_{n-1}, \cdots, x_2, x_1) = (2\pi)^{-n/2} (2r)^{-n^2/2}(4\pi r)^{n/2} \displaystyle\prod_{i=1}^n \phi_r\big(x^n_i\big) \cdot\left\{ \displaystyle\displaystyle\prod_{i<j} (x^n_j-x^n_i) 
\right\}.
\end{equation*}
\end{proposition}

Now, by estimating the determinants in the definition of the densities $q_r, r>0$, see Lemma \ref{lemma: density Warren}, we arrive at the following proposition.
\begin{proposition}\label{prop: lb density 2}
     Fix $r>0$, $m\ge 1$, a sequence $(b_\ell)_{\ell =1}^m \in\R^m_{>}$ with $b_m = 0$. Then we have for $\ell > 0$ and $(x^1_1, \cdots, x^n_n)\in \R^n_{\le}$ the pointwise upper bound on $q_\ell(x_{1:m}; b_{1:m})$ as defined in Lemma \ref{lemma: density Warren}
\begin{equation*}
        q_r(x^n_n, \cdots x^1_1; b_1, \cdots, b_n)\le \displaystyle\int_{\boldsymbol{K}} \boldsymbol{\nu}^n_{r}(\underline{x})\diff\, \left(x^k_j\, \mathrm{where}\, \left.\begin{cases}
       &2\le k \le n\\
       &1\le j \le k-1
   \end{cases}\right\} \right)
\end{equation*}
where 
\begin{equation*}
\begin{array}{cc}
     \boldsymbol{\nu}^n_r(x_1^n, \cdots, x^n_n) & = \displaystyle\prod^n_{i=1}\exp\big(-b_i^2/(4r)\big) \big((4r)^{(i-n)/2}\big)^n \cdot\displaystyle\prod_{i=1}^n \phi_r\big(x^n_i\big) \cdot\displaystyle\prod_{i<j} (x^n_j-x^n_i)  \\
     &  \cdot \displaystyle\sum_{\stackrel{0\le k_i\le n-i}{1\le i \le n}}G^{\underline{k}}(x^n_1,\cdots, x^n_n)\displaystyle\prod^n_{i=1}\left(2|b_{n-i+1}|/\sqrt{4r}\right)^{n-i-k_i}{\binom{n-i}{k_i}}
\end{array}
\end{equation*}
and 
\begin{equation*}
   G^{\underline{k}}(x^n_1,\cdots, x^n_n) = \left\vert\frac{\mathrm{det}\big(H_{k_i}\big(x^n_j/\sqrt{4r}\big)\cdot \exp\big(x^n_j b_{n-i+1}/(2r)\big)_{1\le i, j\le n}}{\displaystyle\prod_{i<j} (x^n_j-x^n_i) }\right\vert
\end{equation*}
where $\underline{k} = (0\le k_i\le n-i)_{1\le i \le n}$.
\end{proposition}

\begin{proof}
Recall the definition of ${\mathbf K}$ as the cone of  points $\mathbf{x}= 
\bigl(x^{1},x^{2}, \ldots x^{n}\bigr)$ with  $x^{k}=\bigl(x^{k}_1,x^{k}_2, \ldots ,x^{k}_k\bigr) \in {\R}^k$ satisfying 
the inequalities 
\begin{equation*}
 x^{k+1}_{i} \leq x^{k}_i \leq x^{k+1}_{i+1}.
\end{equation*}
Notice that for all $1\le i \le n$ and $1\le j \le n-1$, we have 
\begin{equation*}
    \Phi^{(i-j)}(x^j_j-b_{n-i+1}) = \displaystyle\int^{x^j_j}_{-\infty}\int^{x^{j+1}_{j}}_{-\infty}\cdots\displaystyle\int^{x^{n-1}_j}_{-\infty}\Phi^{(i-n)}(x^n_j-b_{n-i+1}) \diff x^n_{j}\cdots \diff x^{j+2}_j \diff x^{j+1}_j
\end{equation*}
Now, by the multi-linearity of the determinant, we obtain
\begin{equation*}
\begin{array}{cc}
    q_\ell(x^m_m, \cdots, x^1_1; b_{1:m}) & = \displaystyle\int_{\boldsymbol{D}}\mathrm{det}\big(\Phi^{(i-n)}(x^n_j-b_{n-i+1})\big)_{1\le i, j\le n} \diff\, \left(x^k_j\, \mathrm{where}\, \left.\begin{cases}
       &2\le k \le n\\
       &1\le j \le k-1
   \end{cases}\right\} \right)
    \end{array}
\end{equation*}
where
\begin{equation*}
    \boldsymbol{D} = \{x^k_i, 1\le i\le k \le n :  x^{k+1}_i\le x^k_i\; \mathrm{for }\; 1\le k \le n-1 \}
\end{equation*}
Since the determinant in the integral above is antisymmetric in $(x_1^n, \cdots, x^n_n)$, we conclude using \cite[Lemma 5.6]{Weiss_2017} that 
\begin{equation*}
   q_{\ell}(x^n_n, \cdots, x^1_1; b_1, \cdots, b_n) = \displaystyle \int_{\boldsymbol{K}}\mathrm{det}\big(\Phi^{(i-n)}(x^n_j-b_{n-i+1})\big)_{1\le i, j\le n}\diff\, \left(x^k_j\, \mathrm{where}\, \left.\begin{cases}
       &2\le k \le n\\
       &1\le j \le k-1
   \end{cases}\right\} \right)
\end{equation*}
for all $(x^1_1, \cdots, x^n_n)\in \R^n_{\le}$. Now using an identity satisfied by Hermite polynomials, see \cite{weisstein2002hermite}, observe that for $1\le i, j \le n$
\begin{equation*}
    \begin{array}{cc}
         & \Phi^{(i-n)}(x^n_j-b_{n-i+1}) = (-1)^{i-n}(4r)^{(i-n)/2}H_{n-i}\big((x^n_j-b_{n-i+1})/\sqrt{4r}\big)\phi_{r}\big(x^n_j-b_{n-i+1}\big)\\
         & = (-1)^{i-n}(4r)^{(i-n)/2}\displaystyle\sum^{n-i}_{k=0}{\binom{n-i}{k}}H_{k}\big(x^n_j/\sqrt{4r}\big)\left(-2b_{n-i+1}/\sqrt{4r}\right)^{n-i-k}\\
         & \cdot \phi_{r}\big(x^n_j)\exp\big(x^n_j b_{n-i+1}/(2r)\big)\exp\big(-b_{n-i+1}^2/(4r)\big)
    \end{array}
\end{equation*}
Thus, by multi-linearity, we obtain
\begin{equation*}
    \begin{array}{cc}
        &\mathrm{det}\big(\Phi^{(i-n)}(x^n_j-b_{n-i+1})\big)_{1\le i, j\le n}\\
        & = \displaystyle\sum_{\stackrel{0\le k_i\le n-i}{1\le i \le n}} \mathrm{det}\bigg({\binom{n-i}{k}_i}(-1)^{i-n}(4r)^{(i-n)/2}H_{k_i}\big(x^n_j/\sqrt{4r}\big)\left(-2b_{n-i+1}/\sqrt{4r}\right)^{n-i-k_i}\\
         & \displaystyle\cdot \phi_{r}\big(x^n_j)\exp\big(x^n_j b_{n-i+1}/(2r)\big)\exp\big(-b_{n-i+1}^2/(4r)\big)\bigg)_{1\le i, j\le n}\\[2ex]
         & = \displaystyle\prod^n_{i=1}\exp\big(-b_{n-i+1}^2/(4r)\big)\cdot  \prod^n_{j=1}\phi_{r}\big(x^n_j)\displaystyle\sum_{\stackrel{0\le k_i\le n-i}{1\le i \le n}}\big((-1)^{i-n}(4r)^{(i-n)/2}\big)^n\\
         & \displaystyle\cdot \displaystyle\prod ^n_{i=1}{\binom{n-i}{k}_i}\left(-2b_{n-i+1}/\sqrt{4r}\right)^{n-i-k_i}\cdot \mathrm{det}\big(H_{k_i}\big(x^n_j/\sqrt{4r}\big)\cdot \exp\big(x^n_j b_{n-i+1}/(2r)\big)\big)_{1\le i, j\le n}\\
    \end{array}
\end{equation*}
Observe now that the map
\begin{equation*}
(x^n_1, \cdots, x^n_n)\mapsto \mathrm{det}\big(H_{k_i}\big(x^n_j/\sqrt{4r}\big)\cdot \exp\big(x^n_j b_{n-i+1}/(2r)\big)_{1\le i, j\le n}
\end{equation*}
is antisymmetric and analytic in all variables in $\R^n_{\le}$. Thus, by multilinearity, that is, the map
\begin{equation*}
    (x_1^n, \cdots, x^n_n)\mapsto \frac{\mathrm{det}\big(H_{k_i}\big(x^n_j/\sqrt{4r}\big)\cdot \exp\big(x^n_j b_{n-i+1}/(2r)\big)_{1\le i, j\le n}}{\displaystyle\prod_{i<j} (x^n_j-x^n_i) }
\end{equation*}
for all $(x_1^n, \cdots, x^n_n)\in \R^n_{<}$, can be expressed in terms of divided differences and by smoothness has a continuous extension up to the boundary of $\R^n_{\le}$.
\end{proof}

We thus have the following  estimating the ratio of transition densities of inhomogeneous and homogeneous BLPP in terms of more analytically tractable functions.
\begin{proposition}\label{prop: density estimate Warren}
Fix $r>0$, $n\in \N$, a sequence $(b_\ell)_{\ell =1}^n\in \R^n_{\ge}$ such that $b_n = 0$  and let ${\mathbf K}$ be the cone of points $\mathbf{x}= 
\bigl(x^{1},x^{2}, \ldots x^{n}\bigr)$ with  $x^{k}=\bigl(x^{k}_1,x^{k}_2, \ldots ,x^{k}_k\bigr) \in {\R}^k$ satisfying 
the inequalities 
\begin{equation*}
 x^{k+1}_{i} \leq x^{k}_i \leq x^{k+1}_{i+1}.
\end{equation*}
Then we have that with $q_r$ as in Lemma \ref{lemma: density Warren} the following pointwise estimate
    \begin{equation*}
\begin{array}{cc}
    &\frac{q_r(x^n_n, \cdots x^1_1; b_{1:n})}{q_r(x^n_n, \cdots x^1_1; \underline{0})}\le \frac{\displaystyle\int_{\boldsymbol{K}} \boldsymbol{\nu}^n_{r}(\underline{x})\diff\, \left(x^k_j\, \mathrm{where}\, \left.\begin{cases}
       &2\le k \le n\\
       &1\le j \le k-1
   \end{cases}\right\} \right)}{\displaystyle\int_{\boldsymbol{K}} \boldsymbol{\mu}^n_r (\underline{x})\diff\, \left(x^k_j\, \mathrm{where}\, \left.\begin{cases}
       &2\le k \le n\\
       &1\le j \le k-1
   \end{cases}\right\} \right)}
   \end{array}
\end{equation*}
where 
\begin{equation*}\label{eq: nu}
\begin{array}{cc}
     \boldsymbol{\nu}^n_r(x_1^n, \cdots, x^n_n) & = \displaystyle\prod^n_{i=1}\exp\big(-b_i^2/(4r)\big) \big((4r)^{(i-n)/2}\big)^n\cdot G^{\underline{k}}(x^n_1,\cdots, x^n_n)\\
     &  \cdot \displaystyle\sum_{\stackrel{0\le k_i\le n-i}{1\le i \le n}}\displaystyle\prod ^n_{i=1}\left(2|b_{n-i+1}|/\sqrt{4r}\right)^{n-i-k_i}{\binom{n-i}{k}_i}\cdot\displaystyle\prod_{i=1}^n \phi_r\big(x^n_i\big) \cdot\displaystyle\prod_{i<j} (x^n_j-x^n_i)  
\end{array}
\end{equation*}
\begin{equation*}
\boldsymbol{\mu}^n_r(x_1^n, \cdots, x^n_n) = (2\pi)^{-n/2} (2r)^{-n^2/2}(4\pi r)^{n/2} \displaystyle\prod_{i=1}^n \phi_r\big(x^n_i\big) \cdot\left\{ \displaystyle\displaystyle\prod_{i<j} (x^n_j-x^n_i) 
\right\} 
\end{equation*}
and 
\begin{equation*}\label{eq: G^k}
   G^{\underline{k}}(x^n_1,\cdots, x^n_n) = \left\vert\frac{\mathrm{det}\big(H_{k_i}\big(x^n_j/\sqrt{4r}\big)\cdot \exp\big(x^n_j b_{n-i+1}/(2r)\big)_{1\le i, j\le n}}{\displaystyle\prod_{i<j} (x^n_j-x^n_i) }\right\vert
\end{equation*}
where $\underline{k} = (0\le k_i\le n-i)_{1\le i \le n}$ for all $(x^1_1, \cdots, x^n_n)\in \R^n_{\le}$.
\end{proposition}

In the following proposition, we now turn to re-express the ratios of determinants (the denominators are Vandermonde determinants) into determinants of divided differences by a simple induction argument.

\begin{proposition}\label{prop: div diff G}
Fix $n\in \N$ and a non-decreasing sequence $(b_{i})^n_{i=1}$ with $b_n=0$ and for $\underline{k} = (0\le k_i\le n-i)_{1\le i \le n}$ consider the function
\begin{equation*}
   G^{\underline{k}}(x^n_1,\cdots, x^n_n) = \left\vert\frac{\mathrm{det}\big(H_{k_i}\big(x^n_j/\sqrt{4r}\big)\cdot \exp\big(x^n_j b_{n-i+1}/(2r)\big)_{1\le i, j\le n}}{\displaystyle\prod_{i<j} (x^n_j-x^n_i) }\right\vert \; \mathrm{for}\; (x^1_1, \cdots, x^n_n)\in \R^n_{<}.
\end{equation*}
Then with $f_i: x\mapsto H_{k_i}\big(x/\sqrt{4r}\big) \exp\big(x \cdot b_{n-i+1}/(2r))$ for $1\le i \le n$. Denote by
$f_i[y_1, y_2, \cdots y_{n-1}, y_n]$ the $n-$th divided difference of $f_i$ for $1\le i \le n$ which are defined as
\begin{equation*}\label{eq: divided differences recursion}
    \begin{array}{cc}
         f_i[y] &= f_i(y), \quad y \in \R \\
         f_i[y_1, y_2, \cdots y_{n-1}, y_n, y_{n+1}] &= \frac{f_i[y_2, \cdots y_{n-1}, y_n, y_{n+1}]-f_i[y_1, y_2, \cdots y_{n-1}, y_n,]}{y_{n+1}-y_1},\quad y_1<\cdots <y_{n+1}.
    \end{array}
\end{equation*}
Then, we have that
\begin{equation*}
\begin{array}{cc}
     & G^{\underline{k}}(x^n_1,\cdots, x^n_n) = \left\vert\mathrm{det}\big(f_i[x^n_1, \cdots, x^n_j]\big)_{1\le i, j\le n}\right\vert \quad \mathrm{for}\; (x^1_1, \cdots, x^n_n) \\
     & = \left\vert\mathrm{det}\big(f_i^{(j)}(\xi_{ij})\big)_{1\le i, j\le n}\right\vert \quad \mathrm{for}\; (x^1_1, \cdots, x^n_n)\in \R^n_{<},
\end{array}
\end{equation*}
where $\xi_{ij}\in [x^n_1, x^n_j]$.
\end{proposition}

\begin{proof}
    First observe that by antisymmetry and multi-linearity, we have that
    \begin{equation*}
        \frac{\mathrm{det}(f_i(x^n_j))_{1\le i,j \le n}}{\prod^{n-1}_{j=1}(x^n_{j+1}-x^n_j)} = \mathrm{det}(\mathbf{1}_{j=1}f_i[x^n_1] + \mathbf{1}_{j>1}f_i[x^n_j, x^n_{j+1}])_{1\le i,j \le n}
    \end{equation*}
    and thus
    \begin{equation*}
    \begin{array}{cc}
         & \displaystyle\frac{\mathrm{det}(f_i(x^n_j))_{1\le i,j \le n}}{\prod^{n-1}_{j=1}(x^n_{j+1}-x^n_j)\cdot \prod^{n-2}_{j=1}(x^n_{j+2}-x^n_j)} = \frac{\mathrm{det}(\mathbf{1}_{j=1}f_i[x^n_1] + \mathbf{1}_{j>1}f_i[x^n_{j-1}, x^n_{j}])_{1\le i,j \le n}}{\prod^{n-2}_{j=1}(x^n_{j+2}-x^n_j)}\\[3ex]
        & = \mathrm{det}(\mathbf{1}_{j=1}f_i[x^n_1] + \mathbf{1}_{j=2}f_i[x^n_1,x^n_2] + \mathbf{1}_{j>2}[x^n_{j-2}, x^n_{j-1}, x^n_{j}])_{1\le i,j \le n}.
    \end{array}
    \end{equation*}
    Now, it becomes clear that using the recursive definition of the divided differences \eqref{eq: divided differences recursion}, we obtain the desired result. The final equality is a standard fact and can be found in any reference on the topic.
\end{proof}
\begin{remark}
    Thus, by the smoothness of the $f_i$, $1\le i \le n$, $G^{\underline{k}}$ has a continuous extension up to the boundary of $\R^n_{\le}$, $n\in \N$.
\end{remark}

In the following lemma, we use the previous proposition and Hadamard's inequality to estimate the above determinants of divided differences up to multiplicative constants by exponential and polynomial factors.

\begin{lemma}\label{lemma: estimates G}
    Fix $n\in \N$ and a non-decreasing sequence $(b_i)^n_{i=1}\in \R^n_{\ge}$ with $b_n=0$ and for $\underline{k} = (0\le k_i\le n-i)_{1\le i \le n}$ consider the function
\begin{equation*}
   G^{\underline{k}}(x^n_1,\cdots, x^n_n) = \left\vert\frac{\mathrm{det}\big(H_{k_i}\big(x^n_j/\sqrt{4r}\big)\cdot \exp\big(x^n_j b_{n-i+1}/(2r)\big)_{1\le i, j\le n}}{ \displaystyle\prod_{i<j} (x^n_j-x^n_i) }\right\vert \; \mathrm{for}\; (x^1_1, \cdots, x^n_n)\in \R^n_{<}.
\end{equation*}
Then, we have the following pointwise estimates
\begin{align*}
G^{\underline{k}}(x^n_1,\cdots, x^n_n) &\le \mathrm{e}^{O(n^2\log n)}\left(\frac{b_1}{2r}\lor 1\right)^{n^2}\cdot\exp\left(\frac{n(x^n_n)_+ \cdot b_1}{2r}\right)\left\langle \frac{(x^n_1)_- + (x^n_n)_+}{\sqrt{4r}}\right\rangle^{n^2}, \; 
\end{align*}
for $(x^n_1, \cdots, x^n_n)\in \R^n_{\le}$ where $\langle \cdot \rangle = \sqrt{\cdot ^2 + 1}$.
\end{lemma}

\begin{proof}
Using Lemma \ref{prop: div diff G}, we have 
\begin{equation*}
      G^{\underline{k}}(x^n_1,\cdots, x^n_n) = \left\vert\mathrm{det}\big(f_i^{(j)}(\xi_{ij})\big)_{1\le i, j\le n}\right\vert \quad \mathrm{for}\; (x^1_1, \cdots, x^n_n)\in \mathrm{int}\R^n_{\le},
\end{equation*}
where 
\begin{equation*}
    f_i(x) =  H_{k_i}\big(x/\sqrt{4r}\big) \exp\big(x \cdot b_{n-i+1}/(2r)\big),\quad x\in \R,  1\le i \le n.
\end{equation*}
Now, using the Leibniz rule, we further see that
\begin{equation*}
\begin{array}{cc}
     f_i^{(j)}(x) & =  \displaystyle\frac{\diff ^j}{\diff x^j} \left(H_{k_i}\big(x/\sqrt{4r}\big) \exp\big(x \cdot b_i/(2r)\big)\right),\quad x\in \R \\
     &  = \displaystyle\sum_{k=0}^j {\binom{j}{k}}\left(\frac{b_{n-i+1}}{2r}\right)^{j-k}\displaystyle\frac{\diff ^j}{\diff x^j} \left(H_{k_i}\big(x/\sqrt{4r}\big)\right)\cdot \exp\big(x \cdot b_{n-i+1}/(2r)\big).
\end{array}
\end{equation*}
Now, to estimate the $G$, we use Hadamard's inequality for determinants to obtain
\begin{equation*}
    \begin{array}{cc}
        &G^{\underline{k}}(x^n_1,\cdots, x^n_n) \le  \displaystyle\prod_{j=1}^n \left(\displaystyle\sum_{i=1}^n (f_i^{(j)}(\xi_{ij}))^2\right)^{\frac{1}{2}}\\
        &\le \displaystyle\prod_{j=1}^n \left(\displaystyle\sum_{i=1}^n \left(\displaystyle\sum_{k=0}^j {\binom{j}{k}}\left(\frac{b_i}{2r}\right)^{j-k}\displaystyle\frac{\diff ^k}{\diff \xi_{ij}^k} \left(H_{k_i}\big(\xi_{ij}/\sqrt{4r}\big)\right)\cdot \exp\big(\xi_{ij} \cdot b_{n-i+1}/(2r)\big)\right)^2\right)^{\frac{1}{2}}.
    \end{array}
\end{equation*}
Note that the Hermite polynomials $H_k, k\ge 0$ can be expressed explicitly as 
     \begin{equation*}
         H_n(x) = n!\displaystyle\sum_{n=0}^{\lfloor n/2\rfloor}\frac{(-1)^{n}}{n!(n-2m)!}(2x)^{n-2m}\, \quad x\in \R, n\ge 0.
     \end{equation*}
From which we obtain the elementary estimates
     \begin{equation*}
     \begin{array}{cc}
          & |H_n(x)|\le n!2^n \big(1+|x|\big)^n,\quad x\in \R, \\
          & \left|\frac{\diff^k}{\diff x^k} H_k(x)\right|\le 2^n \cdot n!\langle x\rangle^{n},\quad x\in \R, 0\le k\le n,
     \end{array}
     \end{equation*}
     where $\langle \cdot \rangle = \sqrt{\cdot ^2 + 1}$.

     Now using the non-negativity and monotonicity of the $b_i$, we obtain
     \begin{equation*}
    \begin{array}{cc}
        &G^{\underline{k}}(x^n_1,\cdots, x^n_n)\\
        &\le  \left(\frac{b_1}{2r}\lor 1\right)^{n^2}\exp\big(\frac{n(x^n_n)_+ \cdot b_1}{2r}\big)\displaystyle\prod_{j=1}^n \left(\displaystyle\sum_{i=1}^n \left(\displaystyle\sum_{k=0}^j {\binom{j}{k}}2^{k_i} \cdot k_i!\left\langle \frac{(x^n_1)_- + (x^n_n)_+}{\sqrt{4r}}\right\rangle^{k_i} \right)^2\right)^{\frac{1}{2}}\\
        & \le \left(\frac{b_1}{2r}\lor 1\right)^{n^2}\exp\big(\frac{n(x^n_n)_+ \cdot b_1}{2r}\big)\left\langle \frac{(x^n_1)_- + (x^n_n)_+}{\sqrt{4r}}\right\rangle^{n^2}\displaystyle\prod_{j=1}^n \left(\displaystyle\sum_{i=1}^n \left(\displaystyle 2^{k_i+j} \cdot k_i!\right)^2\right)^{\frac{1}{2}}\\
        & \le 2^{n^2} \cdot (n!)^n\cdot 2^{(n+1)n/2}\left(\frac{b_1}{2r}\lor 1\right)^{n^2}\exp\big(\frac{n(x^n_n)_+ \cdot b_n}{2r}\big)\left\langle \frac{(x^n_1)_- + (x^n_n)_+}{\sqrt{4r}}\right\rangle^{n^2} \\
        &\le \mathrm{e}^{O(n^2\log n)}\left(\frac{b_1}{2r}\lor 1\right)^{n^2}\exp\big(\frac{n(x^n_n)_+ \cdot b_1}{2r}\big)\left\langle \frac{(x^n_1)_- + (x^n_n)_+}{\sqrt{4r}}\right\rangle^{n^2}\,.
    \end{array}
\end{equation*}
\end{proof}

Having estimated the $G^{\underline{k}}$ terms defined in (\ref{eq: G^k}) in terms of polynomial and exponential factors, the integral estimates in Section \ref{app: int est} in the Appendix allow us to further estimate (\ref{eq: RN ratio}) in the following proposition.  

\begin{proposition}\label{prop: integral ratio estimate}
 Fix $n\in \N$ and let $\boldsymbol{K}$ denote the Gelfand-Tsetslin cone of  points ${\mathbf x}= \bigl(x^{1},x^{2}, \ldots x^{n}\bigr)$ with  $x^{k}=\bigl(x^{k}_1,x^{k}_2, \ldots ,x^{k}_k\bigr) \in {\R}^k$ satisfying 
the inequalities 
\begin{equation*}
 x^{k+1}_{i} \leq x^{k}_i \leq x^{k+1}_{i+1}.
\end{equation*}
Then, for a given $\ell> 0$ and with $f(x^n_1, \cdots, x^n_n) = \prod^{n-1}_{i=1}\phi_\ell(x^n_i)$, 
\begin{equation*}
    \begin{array}{cc}
         & \frac{ \displaystyle\int_{\boldsymbol{K}}\langle (x^n_1)_-\rangle ^N f (x^n_1, \cdots, x^n_n)\displaystyle\prod_{i<j} (x^n_j-x^n_i) \diff\, \left(x^k_j\, \mathrm{where}\, \left.\begin{cases}
       &2\le k \le n\\
       &1\le j \le k-1
   \end{cases}\right\} \right)}{ \displaystyle\int_{\boldsymbol{K}}f (x^n_1, \cdots, x^n_n)\displaystyle\prod_{i<j} (x^n_j-x^n_i) \diff\, \left(x^k_j\, \mathrm{where}\, \left.\begin{cases}
       &2\le k \le n\\
       &1\le j \le k-1
   \end{cases}\right\} \right)} \\
         & \le 
   N!(2\ell)^{N/2}O_{\ell}(\mathrm{e}^{dn^2\log n})\left( \exp\left(\frac{(n-1)((x^n_n)_+)}{(2\ell)^{1/2}}\big)+\exp\big(\frac{(x^1_1)_-}{(2\ell)^{1/2}}\right)\right)
   \end{array}
\end{equation*}
for some universal constant $d>0$.
\end{proposition}

\begin{proof}
    Now observe that for any strictly positive integrable function (rapidly decaying) $f(x^n_1, \cdots, x^n_n)$, we have by Fubini
\begin{equation*}\label{eq: int cone}
\begin{array}{cc}
     & \displaystyle\int_{\boldsymbol{K}}f (x^n_1, \cdots, x^n_n)\displaystyle\prod_{i<j} (x^n_j-x^n_i) \diff\, \left(x^k_j\, \mathrm{where}\, \left.\begin{cases}
       &2\le k \le n\\
       &1\le j \le k-1
   \end{cases}\right\} \right) \\
     & = \displaystyle\int_{x^n_1\le x^n_2\le \cdots\le x^n_{n-1}\le x^n_n}f (x^n_1, \cdots, x^n_n)\displaystyle\prod_{i<j} (x^n_j-x^n_i) \\
     & \cdot\displaystyle\int \mathbf{1}_K(\underline{x}_{\boldsymbol{K}})\diff\, \left(x^k_j\, \mathrm{where}\, \left.\begin{cases}
       &2\le k \le n-1\\
       &1\le j \le k-1
   \end{cases}\right\} \right)\diff x^n_1\cdots \diff x^n_{n-1}
\end{array}
\end{equation*}
where $\underline{x}_{\boldsymbol{K}}$ denotes a generic point in the Gelfand-Tsetslin cone $\boldsymbol{K}$. First observe that by expressing the ambient space $\R^{(n+1)n/2}$ into a union of totally ordered subsets $\R^{(n+1)n/2}_{\sigma}$, with $\sigma\in S^{(n+1)n/2}$ a permutation of indices, we obtain
\[
\boldsymbol{K} = \displaystyle\bigcup_{\sigma\in S^{(n+1)n/2}}\boldsymbol{K}\cap \R^{(n+1)n/2}_{\sigma}
\]
and observe that if for a fixed $\sigma\in $ $S^{(n+1)n/2}$, $\boldsymbol{K}\cap \mathrm{int}\R^{(n+1)n/2}_{\sigma}\neq \emptyset$, then $\mathrm{int}\R^{(n+1)n/2}_{\sigma}\subseteq\mathrm{int}\boldsymbol{K}$ since membership in $\boldsymbol{K}$ is completely determined by a partial order which is implied by the total order induced by $\sigma$, and taking the closure, $\R^{(n+1)n/2}_{\sigma}\subseteq\boldsymbol{K}$. We can thus express
\[
\boldsymbol{K} = \displaystyle\bigcup_{\sigma\in S^{(n+1)n/2}}\boldsymbol{K}_{\sigma}
\]
where $\boldsymbol{K}_{\sigma}$ is either $\R^{(n+1)n/2}_{\sigma}$ or a subset of the boundary $\partial\R^{(n+1)n/2}_{\sigma}$, which has zero Lebesgue measure.

Now with $x^1_1, \cdots, x^n_n$ fixed and strictly increasing, for any two $\sigma, \sigma^\prime \in S^{(n+1)n/2}$ distinct (preserving the above ordering), the induced Lebesgue measure of $\boldsymbol{K}_{\sigma}\cap \boldsymbol{K}_{\sigma^\prime}$ on the subset of $\R^{(n+1)n/2}$ with said $(x^i_i)_{i=1}^n$ fixed vanishes. To see this, by the above we can assume both $\boldsymbol{K}_{\sigma} = \R^{(n+1)n/2}_{\sigma}, \boldsymbol{K}_{\sigma^{\prime}} = \R^{(n+1)n/2}_{\sigma^\prime}$. Now, by a quick induction, for any two distinct permutations, of coordinates, there must be at least a pair of points that have been interchanged in the ordering (here this cannot be the `diagonal' terms $(x^i_i)_{i=1}^n$, being fixed and distinct). It could either be two non-diagonal diagonal entries or a combination of a diagonal entry and a non-diagonal entry. In either case, the intersection lies in a hyperplane with codimension at most one minus that of the ambient space ($\R^{(n+1)n/2-n}$) and so the induced Lebesgue measure is zero as claimed.  Hence, by inclusion-exclusion we see that the integral \ref{eq: int cone} can be expressed as
\begin{equation*}
\begin{array}{cc}
     &\displaystyle\int_{x^n_1\le x^n_2\le \cdots\le x^n_{n-1}\le x^n_n}f (x^n_1, \cdots, x^n_n)\displaystyle\prod_{i<j} (x^n_j-x^n_i) \\
     & \cdot\displaystyle\int \displaystyle\sum_{\stackrel{\sigma\in S^{(k+1)k/2},}{\mathrm{K_{\sigma}}}}\mathbf{1}_{\boldsymbol{K}_{\sigma}}\diff\, \left(x^k_j\, \mathrm{where}\, \left.\begin{cases}
       &2\le k \le n-1\\
       &1\le j \le k-1
   \end{cases}\right\} \right)\diff x^n_1\cdots \diff x^n_{n-1}.
\end{array}
\end{equation*}
Now each non-vanishing term
\begin{equation*}
    \displaystyle\int \mathbf{1}_{\boldsymbol{K}_{\sigma}}(\underline{x}_{\boldsymbol{K}})\diff\, \left(x^k_j\, \mathrm{where}\, \left.\begin{cases}
       &2\le k \le n-1\\
       &1\le j \le k-1
   \end{cases}\right\} \right)
\end{equation*}
can be partitioned into a sum over polynomial factors involving consecutive powers of $(x^n_{j+1}-x^n_j)$ for $1\le j \le n-1$ since all variables $x^k_i, 1\le i < k\le n-1$ are integrated out in every total order and each total order can be factorised into a product of indicators where variables are separated by the diagonal terms $x^i_i, 1\le i \le n $. Furthermore, we have that
\begin{equation*}
    \begin{array}{cc}
        \displaystyle\prod_{i<j} (x^n_j-x^n_i)  &= \displaystyle \prod_{i=1}^{n-1}\left(\displaystyle \sum_{j=i}^{n-1}(x^n_{j+1}-x^n_j)\right) \\
         & =\displaystyle \prod_{i=1}^{n-1}\left(\displaystyle \sum_{j=i}^{n-1}(x^n_{j+1}-x^n_j)\right)\,.
    \end{array}
\end{equation*}

Suppose now that $\prod_{i=1}^{n-1}f_i(x^n_i)$. Then the above allow us to express
\begin{equation*}
\begin{array}{cc}
     & \displaystyle\int_{\boldsymbol{K}}f (x^n_1, \cdots, x^n_n)\displaystyle\prod_{i<j} (x^n_j-x^n_i) \diff\, \left(x^k_j\, \mathrm{where}\, \left.\begin{cases}
       &2\le k \le n\\
       &1\le j \le k-1
   \end{cases}\right\} \right) \\
     & = \displaystyle \Bigg(\sum_{k=1}^{n-1}\sum_{\stackrel{m_1, m_2, \cdots, m_k}{\mathrm{admissible}}}\displaystyle\int_{x^n_1\le x^n_2\le \cdots\le x^n_k\le x^1_1}\prod_{i=1}^{k}f_i(x^n_i)\cdot(x^n_2-x^n_1)^{m_1}\cdot (x^n_3-x^n_2)^{m_2}\\
     & \cdots (b-x^n_k)^{m_k}\diff x^n_1\cdots \diff x^n_{k}\cdot \Omega^{\underline{m}}_k(x^n_{k+1}, \cdots, x^n_n; x^2_2, \cdots x^n_n)\Bigg)
\end{array}
\end{equation*}
for $k\le n-1$ and $m_1, \cdots, m_k\in \N$ such that $\sum_{j=1}^k m_j \le dn^2$ for some universal $d>0$,\\  $\Omega^{\underline{m}}_k(x^n_{k+1}, \cdots, x^n_n; x^2_2, \cdots x^n_n)>0$ in $\mathrm{int}\boldsymbol{K}$, $1\le k \le n-1$.

We will now need the elementary inequality for $(a_i)_{i=1}^n\in \R^n_{\ge 0}$, $(b_i)_{i=1}^n\in \R^n_{> 0}$
\begin{equation*}
    \frac{a_1+a_2+\cdots + a_n}{b_1+b_2+\cdots + b_n}\le \displaystyle\sum_{i=1}^n \frac{a_i}{b_i}
\end{equation*}
and can be easily shown using an induction argument. We thus estimate for all $N\in \N$
\begin{align*}
     & \frac{\displaystyle\int_{\boldsymbol{K}}\langle (x^n_1)_-\rangle ^N f (x^n_1, \cdots, x^n_n)\displaystyle\prod_{i<j} (x^n_j-x^n_i) \diff\, \left(x^k_j\, \mathrm{where}\, \left.\begin{cases}
       &2\le k \le n\\
       &1\le j \le k-1
   \end{cases}\right\} \right)}{\displaystyle\int_{\boldsymbol{K}}f (x^n_1, \cdots, x^n_n)\displaystyle\prod_{i<j} (x^n_j-x^n_i) \diff\, \left(x^k_j\, \mathrm{where}\, \left.\begin{cases}
       &2\le k \le n\\
       &1\le j \le k-1
   \end{cases}\right\} \right)} \\
     & \le \displaystyle \displaystyle \sum_{k=1}^{n-1}\sum_{\stackrel{m_1, m_2, \cdots, m_k}{\mathrm{admissible}}}\displaystyle\frac{\displaystyle\int_{x^n_1\le x^n_2\le \cdots\le x^n_k\le x^1_1}\langle (x^n_1)_-\rangle ^N \cdot \prod_{i=1}^{k}f_i(x^n_i)\cdot(x^n_2-x^n_1)^{m_1}\cdot \cdots \cdot (b-x^n_k)^{m_k}\diff x^n_1\cdots \diff x^n_{k}}{\displaystyle\int_{x^n_1\le x^n_2\le \cdots\le x^n_k\le x^1_1}\prod_{i=1}^{k}f_i(x^n_i)\cdot(x^n_2-x^n_1)^{m_1}\cdot \cdots \cdot (b-x^n_k)^{m_k}\diff x^n_1\cdots \diff x^n_{k}}
\end{align*}
Now, if we further stipulate that $f_i = \phi_{\ell}$, for $\ell>0$, $1\le i \le n-1$, we obtain
by repeated applications of Lemma \ref{lemma: integral estimate}, the estimate
\begin{equation*}
\begin{array}{cc}
     & \frac{\displaystyle\int_{\boldsymbol{K}}\langle (x^n_1)_-\rangle ^N f (x^n_1, \cdots, x^n_n)\displaystyle\prod_{i<j} (x^n_j-x^n_i) \diff\, \left(x^k_j\, \mathrm{where}\, \left.\begin{cases}
       &2\le k \le n\\
       &1\le j \le k-1
   \end{cases}\right\} \right)}{\displaystyle\int_{\boldsymbol{K}}f (x^n_1, \cdots, x^n_n)\displaystyle\prod_{i<j} (x^n_j-x^n_i) \diff\, \left(x^k_j\, \mathrm{where}\, \left.\begin{cases}
       &2\le k \le n\\
       &1\le j \le k-1
   \end{cases}\right\} \right)} \\
     & \le N!(2\ell)^{N/2}O(\mathrm{e}^{dn^2\log n})\left( \exp\left(\frac{n(x^n_n)_+}{(2\ell)^{1/2}}\big)+\exp\big(\frac{(x^1_1)_-}{(2\ell)^{1/2}}\right)\right) \displaystyle \sum_{k=1}^{n-1}\sum_{\stackrel{m_1, m_2, \cdots, m_k}{\mathrm{admissible}}}\mathbf{1}.
\end{array}
\end{equation*}
 Finally, observe that $\sum_{\stackrel{m_1, m_2, \cdots, m_k}{\mathrm{admissible}}}\mathbf{1}$ can be estimated from above by the number of partitions of $O(n^2)$, which famously has $O(\exp(c n))$ asymptotics, see \cite{hardy1918asymptotic}. We finally arrive at the estimate
\[\le N!(2\ell)^{N/2}O(\mathrm{e}^{dn^2\log n})\left( \exp\left(\frac{n(x^n_n)_+}{(2\ell)^{1/2}}\right)+\exp\left(\frac{(x^1_1)_-}{(2\ell)^{1/2}}\right)\right)
\]
for some universal constant $d>0$ (changing from line to line) using the bounds in Corollary \ref{cor: explicit constants good bound} and the remark therein, which concludes the proof.
\end{proof}

Now, combining Lemma \ref{lemma: estimates G} and Propositions \ref{prop: density estimate Warren} and  \ref{prop: integral ratio estimate}, we easily obtain the following proposition, where we are not in a position to estimate the ratio of densities in Theorem \ref{thm: Radon-Nikodym  density inhom BLPP} in terms of analytically tractable quantities, that is up to exponential and polynomial factors.

\begin{proposition}\label{prop: estimate density final}
Fix $r>0$, $n\in \N$, then for a sequence $\underline{b} = (b_\ell)_{\ell =1}^n \in \R^n_{\ge}$ such that $b_n = 0$ and let $(H_1, H_2, \cdots, H_m)(\cdot)$ be a Brownian TASEP started from initial data $\underline{b}$ and  let $(G_1, G_2, \cdots, G_m)(\cdot)$ be a Brownian TASEP started from the origin. As before, with $\boldsymbol{K}$ denote the Gelfand-Tsetslin cone of  points ${\mathbf x}= \bigl(x^{1},x^{2}, \ldots x^{n}\bigr)$ with  $x^{k}=\bigl(x^{k}_1,x^{k}_2, \ldots ,x^{k}_k\bigr) \in {\R}^k$ satisfying 
the inequalities 
\begin{equation*}
 x^{k+1}_{i} \leq x^{k}_i \leq x^{k+1}_{i+1}.
\end{equation*}
    Then with the densities $q_r$ as defined in Lemma \ref{lemma: density Warren}, we have the estimate
    \begin{equation*}
\begin{array}{cc}
    &\frac{q_r(x^n_n, \cdots x^1_1; b_{1:n})}{q_r(x^1_1, \cdots x^n_n; \underline{0})}\le \frac{\displaystyle\int_{\boldsymbol{K}} \boldsymbol{\nu}^n_{r}(\underline{x})\diff\, \left(x^k_j\, \mathrm{where}\, \left.\begin{cases}
       &2\le k \le n\\
       &1\le j \le k-1
   \end{cases}\right\} \right)}{\displaystyle\int_{\boldsymbol{K}} \boldsymbol{\mu}^n_r (\underline{x})\diff\, \left(x^k_j\, \mathrm{where}\, \left.\begin{cases}
       &2\le k \le n\\
       &1\le j \le k-1
   \end{cases}\right\} \right)}\\
    \le &O_r(\mathrm{e}^{dn^2\log n})\displaystyle\prod^n_{i=1}\exp\big(-b_i^2/(4r)\big)\left( \exp\left(\frac{n(x^n_n)_+}{(2r)^{1/2}}\big)+\exp\big(\frac{(x^1_1)_-}{(2r)^{1/2}}\right)\right)\\
   & \cdot\left(\frac{b_1}{2r}\lor 1\right)^{n^2}\exp\big(\frac{x^n_n \cdot \sum_{i=1}^n b_i}{2r}\big)\left( 1 + \frac{(x^n_n)_+}{\sqrt{4r}}\right)^{n^2}.
   \end{array}
\end{equation*}
\end{proposition}

The following lemma controls the Radon-Nikodym derivative of inhomogeneous BLPP against Brownian motion on compacts by chaining the estimates of the Radon-Nikodym derivative in Theorem \ref{thm: Radon-Nikodym  density inhom BLPP} to the estimates for Dyson Brownian motion in Proposition \ref{prop: top line pitman pointwise bound general case}.

\begin{lemma}\label{lemma: ptwise bound Radon-Nikodym  Warren}
Fix $r>0$, $m\ge 1$, a sequence $\underline{b}= (b_\ell)_{\ell =1}^m \in\R^m_{\ge}$ with $b_m = 0$ and for $k\in \llbracket 1, m \rrbracket$ let 
     $$H_k(y) = B^{\uparrow \underline{b}}_k(y)\, , \quad  y\in [0,\infty)\,.$$Furthermore, for $k\in \llbracket 1, m-1 \rrbracket$, let 
     $$G_k(y) = B^{\uparrow \underline{0}_n}_k(y)\, , \quad  y\in [0,\infty)\,.$$ Suppose that the almost sure pointwise bound holds
\begin{equation*}
      \frac{\diff \mathrm{Law}(H_i)_{i =1}^m(\ell)}{\diff\mathrm{Law}(G_i)_{i =1}^m(\ell)}(x_1, \cdots, x_m)  \le f((x_1)+, (x_m)_-) =  g((x_1)_+)\cdot \mathrm{e}^{d_m (x_m)_-}\cdot ((x_1)_++ (x_m)_-+b_1)^N
\end{equation*}
for some non-negative non-decreasing $g$, $N\in \N$ and $d_m>0$. Then, for any $0<\ell<r$, we have the following bound on the Radon-Nikodym  derivative of $H_1$ against rate two Brownian motion (starting from the origin) on paths on $[\ell, r]$ 
\begin{equation*}
    \begin{array}{cc}
     \le &c^N \cdot N^{N/2} \mathrm{e}^{d^2_m/2}\cdot g(\xi(\ell)_+)\cdot (\xi(\ell)_++ cm\sqrt{\ell}+b_1)^N \cdot  \frac{c^{m(m-1)}m^{m(m-1)}}{\prod_{j=1}^{m-1}j!}\\
    & \cdot (\xi(\ell)_+/\sqrt{\ell}+1)^m\cdot (\xi(r)_+/\sqrt{\ell}+1)^m\,.
    \end{array}
\end{equation*}
\end{lemma}
\begin{proof}
    Fix $A\subseteq C_{*,*}([\ell,r])$ Borel measurable and observe that since the ensemble $H_{1:m}^m(\cdot)$ and $G_{1:m}(\cdot)$ are Markov process with the same transition kernel, we compute using \cite[Theorem 7]{O2002representation} analogously to Proposition \ref{prop: top line pitman pointwise bound general case} and coupling the ensemble $G_{1:m}^m(\cdot)$ to the melon transform of independent Brownian motions generating it through the RSK correspondence:
    \[
    WB^m_m(\cdot) \le G_m(\cdot) \le G_1(\cdot) \le WB^m_1 (\cdot)
    \]
    to obtain
    \begin{equation*}
        \begin{array}{cc}
            \PP(H_1(\cdot)\in A) & =  \mathbb{E}\left[\frac{\diff \mathrm{Law}(H_i)_{i =1}^m(r)}{\diff\mathrm{Law}(G_i)_{i =1}^m(\ell)}(G_1(\ell), \cdots, G_m(\ell))\cdot\mathbf{1}(G_1(\cdot)\in A)\right]\\
            & \le  \mathbb{E}\left[f(G_1(\ell)_+, G_m(\ell)_-)\cdot\mathbf{1}(G_1(\cdot)\in A)\right]\\
            & \le  \mathbb{E}\left[f(B_1(\ell)_+, B_m(\ell)_-)\cdot C_\ell h(B_{1:m}(\ell))h(B_{1:m}(r))\cdot\mathbf{1}(B_1(\cdot)\in A)\right]
        \end{array}
    \end{equation*}
where
\begin{equation*}
    C_t = \left[t^{m(m-1)/2}\prod_{j=1}^{m-1}j!\right]^{-1}, \quad t> 0
\end{equation*}
and $h(x_1, x_2, \cdots, x_n) = \prod_{1\le i < j\le m}(x_i-x_j)_+$.\\
Thus, the Radon-Nikodym derivative of $H_n$ against standard Brownian motion on $[\ell,r]$ is pointwise bounded for $\mu$-a.a. paths $\xi$ by (essentially computing the marginal)

\begin{align*}
     &\quad  C_\ell\mathbb{E}_0\left[f(\xi(\ell)_+, B_m(\ell)_-)\cdot  h(\xi(\ell), B_2(\ell), \cdots, B_m(\ell))\cdot h(\xi(r), B_2(r), \cdots, B_m(r))\right] \\[2ex]
     & = \displaystyle C_\ell\mathbb{E}_0\left[f(\xi(\ell)_+, B_m(\ell)_-)\cdot \prod_{1 < j\le m}(\xi(\ell)-B_j(\ell))_+ \cdot \prod_{1 < j\le m}(\xi(r)-B_j(r))_+\right.\\
     &\left.\cdot\prod_{2\le i < k\le m}(B_i(\ell)-B_k(\ell))_+(B_i(r)-B_k(r))_+ \right] \\[2ex]
     & \le  \displaystyle C_\ell\mathbb{E}_0\left[f(\xi(\ell)_+, B_m(\ell)_-)\cdot \prod_{1 < j\le m}(\xi(\ell)_+ +B_j(\ell)_-) \right.\\
     & \left.\cdot \prod_{1 < j\le m}(\xi(r)_+-B_j(r)_-) \cdot\prod_{2\le i < k\le m}(B_i(\ell)_++B_k(\ell)_-)(B_i(r)_++B_k(r)_-) \right]\\[2ex]
     & = \displaystyle C_\ell\mathbb{E}_0\left[g((\xi(\ell))_+)\cdot \mathrm{e}^{d_m (B_m)(\ell)_-}\cdot ((\xi(\ell))_++ (B_m)(\ell)_-+b_1)^N\cdot \prod_{1 < j\le m}(\xi(\ell)_+ +B_j(\ell)_-)\right.\\
     &\displaystyle \left. \cdot \prod_{1 < j\le m}(\xi(r)_+-B_j(r)_-) \cdot\prod_{2\le i < k\le m}(B_i(\ell)_++B_k(\ell)_-)(B_i(r)_++B_k(r)_-) \right]\\[2ex]
     & \displaystyle  \le C_\ell \mathrm{e}^{d_m^2/2} g((\xi(\ell))_+)\mathbb{E}_0\left[\left(((\xi(\ell))_++ (B_m)(\ell)_-+b_1)^N\cdot \prod_{1 < j\le m}(\xi(\ell)_+ +B_j(\ell)_-)\right.\right.\\
     &  \displaystyle \left.\left. \cdot \prod_{1 < j\le m}(\xi(r)_+-B_j(r)_-) \cdot\prod_{2\le i < k\le m}(B_i(\ell)_++B_k(\ell)_-)(B_i(r)_++B_k(r)_-)\right)^2 \right]^{1/2}\\[2ex]
     &  \displaystyle \le \displaystyle C_\ell \mathrm{e}^{d_m^2/2} g((\xi(\ell))_+)\norm{\big(\xi(\ell)_+ + b_1+ B_m(\ell)_- \big)^N\cdot(\xi(\ell)_++ B_m(\ell)_-)}_{2m(m-1)})\\
     &  \displaystyle \cdot\prod_{1 < j < m}(\xi(\ell)_++\norm{B_j(\ell)_-}_{2m(m-1)})\cdot (\xi(r)_++\norm{B_j(r)_-}_{2m(m-1)})\\[2ex]
     &  \displaystyle \cdot \prod_{2\le i < k\le m}(\norm{B_i(\ell)_-}_{2m(m-1)}+\norm{B_k(\ell)_-}_{2m(m-1)})\cdot(\norm{B_i(r)_-}_{2m(m-1)}+\norm{B_k(r)_-}_{2m(m-1)})\\[2ex]
     &  \displaystyle \le  C_\ell \mathrm{e}^{d_m^2/2}\cdot c^N \cdot N^{N/2}\cdot g(\xi(\ell)_+)\cdot (\xi(\ell)_++ \norm{B_m(\ell)_-}_{2m(m-1)}+b_1)^N\cdot  (\xi(\ell)_++ \norm{B_m(\ell)_-}_{2m(m-1)})\\[1ex]
     & \displaystyle  \cdot\prod_{1 < j < m}(\xi(\ell)_++\norm{B_j(\ell)_-}_{2m(m-1)})\cdot (\xi(r)_++\norm{B_j(r)_-}_{2m(m-1)})\\[2ex]
     &  \displaystyle \cdot \prod_{2\le i < k\le m}(\norm{B_i(\ell)_-}_{2m(m-1)}+\norm{B_k(\ell)_-}_{2m(m-1)})\cdot(\norm{B_i(r)_-}_{2m(m-1)}+\norm{B_k(r)_-}_{2m(m-1)})\,,
\end{align*}
by generalised H\"{o}lder, where $\norm{\cdot}_{m2(m-1)}$ denotes the $L^{2m(m-1)}(\PP)$ norm and the fact that for a standard normal random variable $Z$, we have the estimates
\begin{equation*}
    \norm{(Z_-)^n}_{m} \le c^n \cdot n^{n/2}\norm{Z_-}_{m}^n\, , \quad n,m\in \N
\end{equation*}
for some universal constant $c> 0$ using the $O((\cdot)^{1/2})$ asymptotics of moments of $Z$. The latter also implies that for all $1\le j \le m$, 
\begin{equation*}
    \norm{B_j(\ell)_-}_{m(m-1)} \le cm\sqrt{\ell}\,,
\end{equation*}
where $c>0$ is a universal constant, and similarly for $B(r)$ (which follows from the asymptotics of the moments of the gaussian distribution). Thus, we have the further estimate for the Radon-Nikodym  derivative for $\mu$-a.a. paths $\xi$
\begin{equation*}
    \begin{array}{cc}
         & \le \displaystyle C_\ell\cdot c^N \cdot N^{N/2} \mathrm{e}^{d_m^2/2}\cdot g(\xi(\ell)_+)\cdot (\xi(\ell)_++ cm\sqrt{\ell}+b_1)^N\cdot   (\xi(\ell)_++ cm\sqrt{\ell})^m\cdot  (\xi(r)_++ cm\sqrt{\ell})^m\\
     & \cdot \prod_{2\le i < k\le m}(cm\sqrt{\ell}+cm\sqrt{\ell})\cdot(cm\sqrt{r}+cm\sqrt{r})\\[2ex]
     & \le c^N \cdot N^{N/2} \mathrm{e}^{d_m^2/2}\cdot g(\xi(\ell)_+)\cdot (\xi(\ell)_++ cm\sqrt{\ell}+b_1)^N \cdot  (\xi(\ell)_++ cm\sqrt{\ell})^m\cdot  (\xi(r)_++ cm\sqrt{\ell})^m\\
     & \cdot  \frac{c^{m(m-1)}m^{m(m-1)}}{\prod_{j=1}^{m-1}j!}\,.    
    \end{array}
\end{equation*}
\end{proof}

Now, we finally apply Lemma \ref{lemma: ptwise bound Radon-Nikodym  Warren} and Proposition \ref{prop: estimate density final} to obtain the desired pathwise and $L^{\infty-}$ comparison of inhomogeneous BLPP against Brownian motion on compacts, while also estimating from above the growth for all $p>1$ of $L^p$ norms of the Radon-Nikodym derivative as the number of lines tends to infinity. This is the content of the following Theorem.
\begin{theorem}\label{thm: tasep bm abs cont}
Fix $r>0$, $m\ge 1$, a sequence $\underline{b}=(b_\ell)_{\ell =1}^m \in\R^m_{\ge}$ with $b_m = 0$ and let $H(\cdot)$ denote inhomogeneous Brownian LPP started from $\underline{b}$. Then, for all $0<\ell<r<\infty$, we have that the Radon-Nikodym of the law of $H(\cdot)$ against a rate two Brownian motion starting from the origin $\mu$ on $[\ell,r]$ is in $L^{\infty-}(\mu\vert_{[\ell, r]})$. In particular, with $\xi_{\ell, r, m, \underline{b}}$ denoting the law of $H$ as defined above on $[\ell, r]$,
     \begin{equation*}
         \norm{\frac{\diff \xi_{\ell, r, m, \underline{b}}}{\diff\mu\vert_{[\ell, r]}}}_{L^p(\mu\vert_{[\ell, r]})} =  O_p(\mathrm{e}^{d_pm^2\log m}),\quad \mathrm{for all}\quad p>1.
     \end{equation*}
     for some universal in $m\in \N$ (though possibly $p$-dependent) constant $d_p>0$ for all $p>1$.
\end{theorem}
\begin{proof}

     First recall from Definition \ref{def: brownian tasep} that $H(\cdot)$ is the top line $H_1(\cdot)$ of Brownian TASEP $(H_1, H_2, \cdots, H_m)(\cdot)$ started from $\underline{b}$.
     
     Now, suppose first that $b_m = 0$. Then, using Lemma \ref{lemma: ptwise bound Radon-Nikodym  Warren} we conclude that the Radon-Nikodym  derivative of 
\begin{equation*}
    \begin{array}{cc}
         & \frac{\diff \mathrm{Law}(H_i)_{i =1}^m(\cdot)}{\diff\mathrm{Law}(G_i)_{i =1}^m(\cdot)}\le O(\mathrm{e}^{dm^2\log m})\left( \exp\left(\frac{m(\omega_1(\ell))_+}{(2\ell)^{1/2}}\big)+\exp\big(-\frac{(\omega_m(\ell))_-}{(2\ell)^{1/2}}\right)\right)\\
        & \cdot\left(\frac{b_m}{2\ell}\lor 1\right)^{m^2}\exp\big(\frac{m(\omega_m(\ell))_+ \cdot b_m}{2\ell}\big)\left( 1 + \frac{(\omega_m(\ell))_+}{\sqrt{4\ell}}\right)^{m^2}
    \end{array}
\end{equation*}
on paths $\omega\in C^{m}_{*,*}([\ell,r])$ such that $\omega_{i}(\cdot) \le \omega_j (\cdot)$ for all $1\le i < j\le m$ for some universal constant $d> 0$. Now, using Proposition \ref{prop: estimate density final} we obtain the following pointwise bound
    \begin{equation*}
        \begin{array}{cc}
             & \frac{\diff \xi_{\ell, r, m, \underline{b}}}{\diff\mu\vert_{[\ell, r]}} \le  O_\ell(\mathrm{e}^{dm^2\log m})\cdot\left(\frac{b_m}{2\ell}\lor 1\right)^{m^2}\cdot \exp\big(\frac{(\xi(\ell))_+ \cdot \sum_{i=1}^m b_i}{2\ell}\big)\cdot\left( \exp\left(\frac{m(\xi(\ell))_+}{(2\ell)^{1/2}}\big)+\mathrm{e}^{\frac{1}{4\ell}}\right)\right)\\
             & \cdot (\xi(\ell)_++ cm\sqrt{\ell})^{m^2}\prod_{1\le i < j\le m}(\xi(\ell)_+/\sqrt{\ell}+1)\cdot (\xi(r)_+/\sqrt{\ell}+1)
        \end{array}
    \end{equation*}
    on paths $\xi$ on $[\ell, r]$ for some universal constants $c, d> 0$. Now this clearly gives the desired growth estimates and $L^{\infty-}(\mu\vert_{[\ell, r]})$ control on the Radon-Nikodym  derivative.

    Now for the general case, observe that we obtain inductively using the definition of the Pitman transform that
     \[
H_{k-1}(y) = b_1 + W\tilde{L}_1(y)\,,\quad y\ge 0
    \]
    and $H_m = B_m-b_m$, where $\tilde{L} = (B_{k-1}-b_m, H^{k})$, $(H^{\ell})^m_{\ell=1}$ is the ensemble obtain by consecutively reflecting upwards the family $(B_\ell-b_m)_{\ell=1}^m$ of independent rate two Brownian motions starting from $(b_\ell-b_m)_{\ell=1}^m$. Thus, we need to chain the Radon-Nikodym  derivative of $\xi_{\ell, r, m, \underline{b}}$ against a rate two Brownian motion starting at $b_m$ with that of a rate two Brownian motion starting from the origin. We thus obtain as above
     \begin{equation*}
        \begin{array}{cc}
             & \frac{\diff \xi_{\ell, r, m, \underline{b}}}{\diff\mu\vert_{[\ell, r]}} \le  O_\ell(\mathrm{e}^{dm^2\log m})\cdot\displaystyle\prod^m_{i=1}\exp\big(-(b_i-b_m)^2/(4\ell)\big)\cdot \left(\frac{(b_1-b_m)}{2\ell}\lor 1\right)^{m^2}\\
             & \cdot \exp\big(\frac{(\xi(\ell)-b_m) \cdot\sum_{i=1}^m(b_i-b_m)}{2\ell}\big)\cdot \exp\left(\frac{m(\xi(\ell)-b_1)_+}{(2\ell)^{1/2}}\right)\\
             & \cdot ((\xi(\ell)-b_m)_++ m\sqrt{\ell})^{m^2+m}\cdot ((\xi(r)-b_m)_+/\sqrt{\ell}+1)^m\cdot \mathrm{e}^{\frac{\xi(\ell)b_m}{2\ell}}\exp\big(-b_m^2/(4\ell)\big)
        \end{array}
    \end{equation*}
    on paths $\xi$ on $[\ell, r]$ for some universal constants $c, d> 0$ which, again, gives the desired growth estimates and $L^{\infty-}(\mu\vert_{[\ell, r]})$ control on the Radon-Nikodym  derivative of the law of $H_1$ as defined above on $[\ell, r]$ against a rate two Brownian motion starting from the origin on $[\ell,r]$, concluding the proof.
\end{proof}
\begin{remark}
\begin{itemize}
    \item Note that by translation, the increment process $H(\cdot)-H(\ell)$ on $[\ell,r]$ has a Radon-Nikodym  derivative against rate two Brownian motion on $[\ell,r]$ that only depends on the values $(b_\ell-b_m)_{\ell = 1}^m$. In particular, applying Proposition \ref{prop: top line pitman pointwise bound general case} with $\tilde{\xi}_{\ell, r, m, \underline{b}}$ denoting the law of $H(\cdot+\ell)-H(\ell)$ on $[0,r-\ell]$ we obtain
     \begin{equation*}
     \begin{array}{cc}
         \norm{\frac{\diff \tilde{\xi}_{\ell, r, m, \underline{b}}}{\diff\mu}}_{L^p(\mu)} &= \displaystyle\prod^m_{i=1}\exp\big(-(b_i-b_m)^2/(4\ell)\big)\cdot \left(\frac{(b_1-b_m)}{2\ell}\lor 1\right)^{m^2}\\
         & \quad \cdot O_{p,\ell, r}\Bigg(\mathrm{e}^{d m^2\log m + c_{\ell} \big(\sum_{i=1}^m (b_i-b_m)\big)^2}\Bigg),
    \end{array}
     \end{equation*}
     for some constants $c_{\ell,r},d>0$ independent of $m\in \N$ and all $p>1$. The same can be said if one shifts downward, that is, start with initial conditions $(b_\ell-b_1)_{\ell = 1}^m$.
    \item In the special case of `almost' homogeneous boundary data, by a special coupling to Brownian motion in the Gelfand-Tsetslin cone, one can obtain much simpler pathwise estimates on the Radon-Nikodym derivative of inhomogeneous BLPP, see Proposition \ref{prop: top line pitman pointwise bound general case special coupling} in the Appendix.
\end{itemize}
\end{remark}

\section{Future directions and applications}\label{sec: future directions and applications}\label{sec: misc}

In this section, we will discuss some applications of the pathwise and $L^{\infty-}$ estimates obtained thus far for inhomogeneous Brownian LPP. A first application will be to consider a simplified model for the KPZ fixed point, see \cite{Sarkar2021Brownian} and obtain Radon-Nikodym derivative estimates; this is the content of Theorem \ref{thm: tasep bm abs cont}. This model is relevant because the KPZ fixed point can be realised as inhomogeneous \textbf{Airy} LPP of `random depth' and initial data, see \cite{Sarkar2021Brownian}. The definition of the above is the same as that for inhomogeneous Brownian LPP, save for the random environment which is the parabolic Airy line ensemble. One can exploit the Brownian nature of the parabolic Airy line ensemble obtained in \cite{Corwin2014Brownian} and study the KPZ fixed point within the framework of Brownian TASEP.

A first step in this direction providing some quantitative control for the actual KPZ fixed point starting from some appropriate initial data is the main result of \cite{tassopoulos2025quantitativebrownianregularitykpz}. In that paper, we use the estimates in Theorem \ref{thm: tasep bm abs cont} as a crucial technical input.

\subsection{RN derivative of BLPP with ``random depth'' with unbounded support}\label{subsec: BLPP rand depth RN}
We now obtain pathwise and $L^{\infty-}$ estimates for the Radon-Nikodym derivative of inhomogeneous BLPP of `random depth', which we use as a toy model for the KPZ fixed point as defined in \cite{Sarkar2021Brownian}. This model is inspired by a reduction of the KPZ fixed point to BLPP of ``random depth'' since one can exploit the Brownian structure of the random ensembles involved in its construction, particularly, the Brownian Gibbs property of the Airy line ensemble, see \cite{Corwin2014Brownian} to reduce the full problem to that of inhomogeneous BLPP incurring the cost given by an appropriate Radon-Nikodym derivative.

\begin{theorem}\label{thm: KPZ rn}
Fix $r>0$, and let $B = (B_1, B_2, \cdots)\in C^{\N}([0,\infty))$ be an ensemble of independent rate two Brownian motions starting from the origin. Furthermore, let $(G_\ell)_{\ell =1}^\infty$ be an almost surely non-increasing family of random variables independent from $B$. Suppose further that there exists a $\sigma\big(G_\ell: \ell\in \N\big)$-measurable random positive integer $L_0$ that is a stopping time with respect to the filtration  $\sigma\big(G_\ell: \ell\in \llbracket 1, n\rrbracket \big)_{n\in \N}$  and a universal constant $c>0$ such that 
\[
\displaystyle\sup_{r\in [0,\infty)}\mathrm{e}^{cr^3}\PP(L_0\ge r)<\infty\,.
\]
Moreover, suppose there exists an $\epsilon> 0$ such that for all $m>0$ on the event $\{L_0\le m\}$,
\begin{equation*}
\displaystyle\max_{1\le \ell\le m}(G_{\ell})_- \le Cm^{\frac{1}{2}-\epsilon},
\end{equation*}
for some universal positive constant $C>0$. Finally, suppose that
there exists another positive constant such that 
\begin{equation}\label{eq: geod conc}
\displaystyle\sup_{r\in [0,\infty)}\mathrm{e}^{dr^3}\PP((G_1)_+\ge r)<\infty\,.
\end{equation}
Now, suppose $H(\cdot)$ is a continuous stochastic process on the positive reals such that for all $m\in \N$, on the event $\{L_0\le m\}$
\[
H(y) = \displaystyle\max_{1\leq \ell\leq m}(G_\ell+B[(0,\ell)\to(y,1)]) \,,\quad y\ge 0\,.
 \]
Then, for all $0<\ell<r<\infty$, we have that the Radon-Nikodym 
     of the law of $H(\cdot)$ against a rate two Brownian motion starting from the origin $\mu$ on $[\ell,r]$ is in $L^{\infty-}(\mu)$.
\end{theorem}

\begin{proof}
Fix $A\subseteq C_{*,*}([\ell,r])$ Borel measurable. Then we estimate for all $m\in \N$ using Theorem \ref{thm: tasep bm abs cont}
\begin{equation*}
\begin{array}{cc}
     &\PP(H(\cdot)\in A)= \PP(H(\cdot)\in A, L_0\le m) + \PP(L_0\geq m+1) \\
     & \le \PP(H_m(\cdot)\in A, L_0\le m) + \PP(L_0\geq m)\\
     & \le \mathbb{E}\left[\norm{\frac{\diff \xi_{\ell, r, m} \,\mathrm{conditioning}\, (G_\ell)^m_{\ell=1}}{\diff\mu}}_{L^p(\mu)}\cdot \mathbf{1}(L_0\le m)\right]\mu(A)^{1-\frac{1}{p}} + \PP(L_0\geq m)\\
     & \le \mathbb{E}\left[ O_p(\mathrm{e}^{d_p \cdot  m^{3-\delta}})\right]\mu(A)^{1-\frac{1}{p}} + \PP(L_0\geq m)\\
\end{array}
\end{equation*}
for some universal in $m\in \N$ (though possibly $p$-dependent) constant $d_p>0$ for all $p>1$ and $\delta\in (0,3)$. Now, without loss of generality, assume that $\mu(A)>0$ and let $m\in \N$ be the unique positive integer such that $\mu(A)\in [1/\mathrm{e}^{m^{(3-\frac{\delta}{2})}}, 1/\mathrm{e}^{(m-1)^{(3-\frac{\delta}{2})}})$. Now fix $1<r<p$ and estimate, 
\begin{equation*}
    \begin{array}{cc}
     \PP(H(\cdot)\in A)&\le \mathbb{E}\left[ O_p(\mathrm{e}^{d_p \cdot  m^{3-\delta}})\right]\mu(A)^{1-\frac{1}{p}} + \mathrm{e}^{m^{3-\frac{\delta}{2}}}\cdot \PP(L_0\geq m)\cdot \mathrm{e}^{-m^{(3-\frac{\delta}{2})}} \\
     &\le  O_p(\mathrm{e}^{d_p \cdot  m^{3-\delta}})\mu(A)^{1-\frac{1}{p}} + \mathrm{e}^{m^{3-\frac{\delta}{2}}}\cdot \PP(L_0\geq m)\cdot \mu(A)\\
     &\le  O_p(\mathrm{e}^{d_p \cdot  m^{3-\delta}})\mu(A)^{\frac{p-r}{pr}} + \mathrm{e}^{m^{3-\frac{\delta}{2}}}\cdot \PP(L_0\geq m)\cdot \mu(A)^{1-\frac{1}{r}}\\
     &\le  O_p\Bigg(\displaystyle\sup_{m\in \N}\left\{\mathrm{e}^{d_p \cdot  m^{3-\delta}}\mathrm{e}^{-\frac{p-r}{pr}(m-1)^{(3-\frac{\delta}{2})}}+\mathrm{e}^{cm^3}\cdot \PP(L_0\geq m)\right\}\Bigg)\mu(A)^{1-\frac{1}{r}}\tag{*}.\\
\end{array}
\end{equation*}
Thus, for all $0<\ell<r<\infty$, we have that the Radon-Nikodym  derivative of the law of $H(\cdot)$ against a rate two Brownian motion starting from the origin $\mu$ on $[\ell,r]$ is in $L^{p-}(\mu)$. Since, $p$ was arbitrary then allows us to conclude the proof of the theorem.
\end{proof}

Note that here $L_0$, the `random depth' of the inhomogeneous BLPP in Theorem \ref{thm: tasep bm abs cont} is meant to stand for the input from geodesic geometry on the Airy line ensemble. More specifically, it is the intercept of semi-inifinite geodesics in the parabolic Airy line ensemble, see \cite[Definition 4.2 and Lemma 4.2]{tassopoulos2025quantitativebrownianregularitykpz}. At present, we are only able to obtain that these geodesic intercepts have exponentially stretched tails, and not the tails in (\ref{eq: geod conc}). However, we do believe that the latter is achievable and the tail bound is consistent with similar results in the KPZ universality class regarding transversal fluctuations of geodesics in other models, including \cite[Proposition 6.2]{basu2018coalescencegeodesicsexactlysolvable} for Poisson LPP and \cite[Lemma 3.13]{RahmanGeodInterface} in the directed landscape.

Notwithstanding, with the estimates obtained in Theorem \ref{thm: tasep bm abs cont}, we are still able to obtain a form of quantitative Brownian regularity for the increments of the KPZ fixed point started from finitary initial data in the following theorem. 

\begin{theorem}(Quantitative Brownian regularity, \cite[Theorem 6.6]{tassopoulos2025quantitativebrownianregularitykpz})\label{thm: KPZ reg finitary}
     Let $h_t(\cdot):=\mathcal{L}(t;h_0), t\ge 0$ be the KPZ fixed point started from initial data $h_0$ which is $t$-finitary. Then, for any fixed $\ell<r$ with $|\ell|+ |r|\le y_0$ for some $y_0>0$, there exists some universal $\theta> 0$ such that the estimates for all $A$ Borel measurable $A\subseteq C_{0,*}([0,r-\ell])$ with rate two Wiener measure $\mu(A) > 0$,
\begin{equation*}
\begin{array}{cc}
     &\PP(h_t(\cdot + \ell)-h_t(\ell)\in A)\le c_{t, y_0, h_0}\exp\left( -d_{t, y_0, h_0}\log^{\theta}\log\big(1/\mu(A)\big)\right)\,,
\end{array}
\end{equation*} 
hold for some $c_{t, y_0, h_0}, d_{t, y_0, h_0}>0$.
\end{theorem}

\section{Appendix}\label{sec: Appendix}

\subsection{Integral estimates}\label{app: int est}
In this subsection, we control the growth rate of `cumulants' of the integrands used to estimate (\ref{eq: nu}) where the flavour of most arguments is inductive, owing to the recursive nature of the iterated integrals under consideration.

\begin{lemma}
    Fix $k\in \N, \ell> 0$, $m_1, m_2, \cdots, m_k \in \N$ and $\lambda > 0$ and consider the functions
    \begin{equation*}
        \begin{array}{cc}
            h_{m_1, \cdots, m_k}(b) = \displaystyle\int_{x_1\le x_2\le \cdots\le x_k\le b}\mathrm{e}^{-\lambda x_1}\prod_{i=1}^{k}\mathrm{e}^{-\frac{x_i^2}{2}}\cdot(x_2-x_1)^{m_1}\cdot \cdots \cdot (b-x_k)^{m_k}\diff x_1\cdots \diff x_{k},\, b\in \R\,.
        \end{array}
    \end{equation*}
    and 
    \begin{equation*}
        \begin{array}{cc}
            g_{m_1, \cdots, m_k}(b) = \displaystyle\int_{x_1\le x_2\le \cdots\le x_k\le b}\prod_{i=1}^{k}\mathrm{e}^{-\frac{x_i^2}{2}}\cdot(x_2-x_1)^{m_1}\cdot \cdots \cdot (b-x_k)^{m_k}\diff x_1\cdots \diff x_{k},\, b\in \R\,.
        \end{array}
    \end{equation*}
    Then there exists a universal constant $d>0$ such that 
    \[
     h_{m_1, \cdots, m_k}(b)\le \mathrm{e}^{\lambda^2/2}\mathrm{e}^{-(k-1)\lambda^2/2} \cdot \mathrm{e}^{(k-1)\lambda b}\cdot g_{m_1, \cdots, m_k}(b+\lambda), \quad b\in \R\,.
    \]
\end{lemma}

\begin{proof}
    We proceed again by induction. The base case is clear upon performing a change of variables. Indeed, 
    \begin{align*}
            h_{m_1}(b) &= \displaystyle\int_{y\le b}\mathrm{e}^{-\lambda y}\mathrm{e}^{-\frac{y^2}{2}}\cdot(b-y)^{m_1}\diff y\\
            &= \mathrm{e}^{\lambda^2/2} \displaystyle\int_{y\le b}\mathrm{e}^{-\frac{(y+\lambda)^2}{2}}\cdot(b+\lambda-(y+\lambda))^{m_1}\diff y\\
            &= \mathrm{e}^{\lambda^2/2} \displaystyle\int_{y\le b+\lambda}\mathrm{e}^{-\frac{(y)^2}{2}}\cdot(b+\lambda-y)^{m_1}\diff y\\
            & = \mathrm{e}^{\lambda^2/2}g_{m_1}(b+\lambda),\quad b\in \R\,.
    \end{align*}
    Suppose the claim were true for some $k\ge 2$. Then we have that
    \begin{align*}
         h_{m_1, \cdots, m_{k+1}}(b) &= \displaystyle\int^b_{-\infty}\mathrm{e}^{-\frac{(y)^2}{2}} h_{m_1, \cdots, m_{k}}(y)(b-y)^{m_{k+1}}\diff y\\
         & \le \mathrm{e}^{\lambda^2/2}\mathrm{e}^{-(k-1)\lambda^2/2}\displaystyle\int^b_{-\infty}\mathrm{e}^{-\frac{(y)^2}{2}}  \mathrm{e}^{(k-1)\lambda y}\cdot g_{m_1, \cdots, m_k}(y+\lambda)(b-y)^{m_{k+1}}\diff y\\
         & \le  \mathrm{e}^{\lambda^2/2}\mathrm{e}^{-k\lambda^2/2}\mathrm{e}^{k\lambda b}\displaystyle\int^{b+\lambda}_{-\infty} \mathrm{e}^{-\frac{(y)^2}{2}}\cdot g_{m_1, \cdots, m_k}(y)(b+\lambda-y)^{m_{k+1}}\diff y\\
         &=  \mathrm{e}^{\lambda^2/2}\mathrm{e}^{-(k-1)\lambda^2/2}\mathrm{e}^{k\lambda b}\cdot g_{m_1, \cdots, m_{k+1}}(b+\lambda),\quad b\in \R.
    \end{align*}
    completing the induction.
\end{proof}

\begin{lemma}\label{lemma: difference bound} Fix $k\in \N, \ell> 0$ and $m_1, m_2, \cdots, m_k, m_{k+1} \in \N$ and consider the function
\begin{align*}
g_{m_1, \cdots, m_k}(b) = \displaystyle\int_{x_1\le x_2\le \cdots\le x_k\le b}\prod_{i=1}^{k}\mathrm{e}^{-\frac{x_i^2}{2}}\cdot(x_2-x_1)^{m_1}\cdot \cdots \cdot (b-x_k)^{m_k}\diff x_1\cdots \diff x_{k}\,, b\in \R.
\end{align*}
Then there exists a universal constant $d>0$ such that 
    $$
     g_{m_1, \cdots, m_k}(b+1)\le C_{k, m_1, \cdots, m_k} \cdot (1+\mathrm{e}^{-kb})\cdot g_{m_1, \cdots, m_k}(b), \quad b\in \R
    $$
    where $C_{k+1, m_1, \cdots, m_{k+1}} = 2^{m_{k+1}}(C_{k,\underline{m}}\mathrm{e}^{k+1}\lor \mathrm{e}^{d\sum_{i=1}^k m_i\log m_i})$, and $C_{1,m_1} = O(2^{m_1})$.
\end{lemma}

\begin{proof}
    Indeed, we proceed by induction. To see this note that for the inductive step
    \begin{equation*}
    \begin{array}{cc}
         g_{m_1, \cdots, m_{k+1}}(b) &= \displaystyle\int^b_{-\infty}\mathrm{e}^{-\frac{(y)^2}{2}} g_{m_1, \cdots, m_{k}}(y)(b-y)^{m_{k+1}}\diff y\\
         & = \displaystyle\int^\infty_{0}\mathrm{e}^{-\frac{(b-y)^2}{2}} g_{m_1, \cdots, m_{k}}(b-y)(y)^{m_{k+1}}\diff y,\quad b\in \R
    \end{array}
    \end{equation*}
    and
\begin{align*}
& \frac{g_{m_1, \cdots, m_{k+1}}(b+1)}{g_{m_1, \cdots, m_{k+1}}(b)} =  \frac{
    \displaystyle \int_{-\infty}^b \mathrm{e}^{-\frac{y^2}{2}} g_{m_1, \cdots, m_k}(y)
    (b+1 - y)^{m_{k+1}} \, \diff y
}{
    \displaystyle \int_{-\infty}^b \mathrm{e}^{-\frac{y^2}{2}} g_{m_1, \cdots, m_k}(y)
    (b - y)^{m_{k+1}} \, \diff y
}
\\
& + \frac{
    \displaystyle \int_b^{b+1} \mathrm{e}^{-\frac{y^2}{2}} g_{m_1, \cdots, m_k}(y)
    (b+1 - y)^{m_{k+1}} \, \diff y
}{
    \displaystyle \int_{-\infty}^b \mathrm{e}^{-\frac{y^2}{2}} g_{m_1, \cdots, m_k}(y)
    (b - y)^{m_{k+1}} \, \diff y
} \\
\le\ & 2^{m_{k+1} - 1} \frac{
    \displaystyle \int_{-\infty}^b \mathrm{e}^{-\frac{y^2}{2}} g_{m_1, \cdots, m_k}(y)
    \bigl(1 + (b - y)^{m_{k+1}}\bigr) \, \diff y
}{
    \displaystyle \int_{-\infty}^b \mathrm{e}^{-\frac{y^2}{2}} g_{m_1, \cdots, m_k}(y)
    (b - y)^{m_{k+1}} \, \diff y
} \\
& + \frac{
    \displaystyle \int_0^1 \mathrm{e}^{-\frac{(b+1 - y)^2}{2}} g_{m_1, \cdots, m_k}(b+1 - y)
    y^{m_{k+1}} \, \diff y
}{
    \displaystyle \int_0^\infty \mathrm{e}^{-\frac{(b - y)^2}{2}} g_{m_1, \cdots, m_k}(b - y)
    y^{m_{k+1}} \, \diff y
} \\
\le\ & 2^{m_{k+1} - 1} + 2^{m_{k+1} - 1} 
\frac{
    \displaystyle \int_{-\infty}^b \mathrm{e}^{-\frac{y^2}{2}} g_{m_1, \cdots, m_k}(y) \, \diff y
}{
    \displaystyle \int_{-\infty}^{-(b)_- - 1} \mathrm{e}^{-\frac{y^2}{2}} g_{m_1, \cdots, m_k}(y) \, \diff y
} \\
& + C_k \frac{
    \displaystyle \int_0^1 \mathrm{e}^{-\frac{(b+1 - y)^2}{2}} (1 + \mathrm{e}^{-k(b - y)}) g_{m_1, \cdots, m_k}(b - y)
    y^{m_{k+1}} \, \diff y
}{
    \displaystyle \int_0^\infty \mathrm{e}^{-\frac{(b - y)^2}{2}} g_{m_1, \cdots, m_{k-1}}(b - y)
    y^{m_{k+1}} \, \diff y
}\\
&\displaystyle\le 2^{m_{k+1}-1} + 2^{m_{k+1}-1}\frac{\left(\int^b_{-(b)_- -1}+\int^{-(b)_- -1}_{-\infty}\right)\mathrm{e}^{-\frac{(y)^2}{2}} g_{m_1, \cdots, m_{k}}(y)\diff y}{\int^{-(b)_- -1}_{-\infty}\mathrm{e}^{-\frac{(y)^2}{2}} g_{m_1, \cdots, m_{k}}(y)\diff y}\\
         &\displaystyle +C_k\cdot\frac{\mathrm{e}^{k+1}(1+\mathrm{e}^{-kb}) \mathrm{e}^{-b}\int^1_{0}\mathrm{e}^{by-\frac{y^2}{2}}g_{m_1, \cdots, m_{k}}(b-y)(y)^{m_{k+1}}\diff y}{\int^\infty_{0}\mathrm{e}^{by-\frac{y^2}{2}} g_{m_1, \cdots, m_{k-1}}(b-y)(y)^{m_{k+1}}\diff y}\\
         &\displaystyle\le 2^{m_{k+1}} +C_k\mathrm{e}^{k+1-b} (1+\mathrm{e}^{-kb})\\
         &\displaystyle+\mathbf{1}_{b\le 0}2^{m_{k+1}-1}C_k\mathrm{e}^{k-b} (1+\mathrm{e}^{-k(b-1)})+\mathbf{1}_{b\ge0}\frac{2^{m_{k+1}-1}\int^b_{ -1}\mathrm{e}^{-\frac{(y)^2}{2}} g_{m_1, \cdots, m_{k}}(y)\diff y}{\int^{-1}_{-\infty}\mathrm{e}^{-\frac{(y)^2}{2}} g_{m_1, \cdots, m_{k}}(y)\diff y}\\
         &\displaystyle\le 2^{m_{k+1}} + C_k\mathrm{e}^{k+1-b} (1+\mathrm{e}^{-kb})+ \mathbf{1}_{b\le 0}2^{m_{k+1}-1}C_k\mathrm{e}^{k-b} (1+\mathrm{e}^{-k(b-1)})\\
         &+ 2^{m_{k+1}-1}\mathbf{1}_{b\ge0}\frac{\int^\infty_{ -\infty}\mathrm{e}^{-\frac{(y)^2}{2}} g_{m_1, \cdots, m_{k}}(y)\diff y}{\int^{-1}_{-\infty}\mathrm{e}^{-\frac{(y)^2}{2}} g_{m_1, \cdots, m_{k}}(y)\diff y}\,. 
\end{align*}

    Now observe that we can estimate
    \begin{equation*}
    \begin{array}{cc}
        &\displaystyle\int_{\R}\mathrm{e}^{-\frac{(y)^2}{2}} g_{m_1, \cdots, m_{k}}(y)\diff y\\
        & = \displaystyle \int_{\R}\int_{x_1\le x_2\le \cdots\le x_k\le y}\mathrm{e}^{-\frac{y^2}{2}}\prod_{i=1}^{k}\mathrm{e}^{-\frac{x_i^2}{2}}\cdot(x_2-x_1)^{m_1}\cdot \cdots \cdot (y-x_k)^{m_k}\diff x_1\cdots \diff x_{k}\diff y\\
        & \le \displaystyle \int_{\R^{k+1}}\mathrm{e}^{-\frac{y^2}{2}}\prod_{i=1}^{k}\mathrm{e}^{-\frac{x_i^2}{2}}\cdot|x_2-x_1|^{m_1}\cdot \cdots \cdot |y-x_k|^{m_k}\diff x_1\cdots \diff x_{k}\diff y\\
        & \le 2^{\sum_{i=1}^k m_i }\displaystyle \int_{\R^{k+1}}\prod_{i=1}^{k+1}\mathrm{e}^{-\frac{x_i^2}{2}}\cdot\big(|x_2|^{m_1}+|x_1|^{m_1}\big)\cdot \cdots \cdot \big(|x_{k+1}|^{m_k}+|x_k|^{m_k}\big)\diff x_1\cdots \diff x_{k+1}\\
        & \le 2^{\sum_{i=1}^k m_i + k}\mathrm{e}^{d\sum_{i=1}^k m_i\log m_i}\\
    \end{array}
    \end{equation*}
    for a universal constant $d>0$ using independence and the asymptotics of moments of a gaussian random variable. Furthermore, repeated applications of Lemma \ref{lemma: integral estimate} give the lower bound 
    \begin{equation*}
    \begin{array}{cc}
          &C_{k, \underline{m}}\le \displaystyle\int^{-1}_{-\infty}g_{m_1, \cdots, m_{k}}(y) \diff y\,,
    \end{array}
    \end{equation*}
for $C_{k, \underline{m}} = c^k\prod_{i=1}^k\frac{1}{i^{m_{i}}}$ for some universal constant $c>0$. Combining the above, we arrive at
\begin{equation*}
\begin{array}{cc}
     & \frac{g_{m_1, \cdots, m_{k+1}}(b+1)}{g_{m_1, \cdots, m_{k+1}}(b)}\le 2^{m_{k+1}} + \mathbf{1}_{b\le 0}2^{m_{k+1}-1}C_k\mathrm{e}^{k-b} (1+\mathrm{e}^{-k(b-1)})\\
     & +C_k\mathrm{e}^{k+1-b} (1+\mathrm{e}^{-kb}) + 2^{m_{k+1}-1}\mathbf{1}_{b\ge 0}c^k \cdot \mathrm{e}^{d\sum_{i=1}^k m_i\log m_i} \\
     & \le C_{k+1}(1+\mathrm{e}^{-(k+1)b})
\end{array}
\end{equation*}
    for universal constants $c,d>0$ where $C_{k+1} = 2^{m_{k+1}}(C_{k}\mathrm{e}^{k+1}\lor \mathrm{e}^{d\sum_{i=1}^k m_i\log m_i})$. This concludes the induction.
\end{proof}

As an easy corollary of Lemma \ref{lemma: difference bound} we obtain the following lemma. 
\begin{corollary}
     Fix $k\in \N, \ell> 0$ and $m_1, m_2, \cdots, m_k \in \N$ and consider the functions
    \begin{equation*}
        \begin{array}{rr}
            h_{m_1, \cdots, m_k}(b) &= \displaystyle\int_{x_1\le x_2\le \cdots\le x_k\le b}\mathrm{e}^{-x_1}\prod_{i=1}^{k}\mathrm{e}^{-\frac{x_i^2}{2}}\cdot(x_2-x_1)^{m_1}\cdot \cdots \cdot (b-x_k)^{m_k}\diff x_1\cdots \diff x_{k},\, b\in \R.
        \end{array}
    \end{equation*}
    and 
    \begin{equation*}
        \begin{array}{rr}
            g_{m_1, \cdots, m_k}(b) &= \displaystyle\int_{x_1\le x_2\le \cdots\le x_k\le b}\prod_{i=1}^{k}\mathrm{e}^{-\frac{x_i^2}{2}}\cdot(x_2-x_1)^{m_1}\cdot \cdots \cdot (b-x_k)^{m_k}\diff x_1\cdots \diff x_{k},\, b\in \R.
        \end{array}
    \end{equation*}
    Then there exists a universal constant $d>0$ such that 
    \[
     \frac{h^{\ell}_{m_1, \cdots, m_k}}{g^{\ell}_{m_1, \cdots, m_k}}(b)\le  C_{k,\underline{m}} \cdot (\mathrm{e}^{(k-1)b}+\mathrm{e}^{-b}), \quad b\in \R.
    \]
    where $C_{k+1, \underline{m},m_{k+1}} = 2^{m_{k+1}}(C_{k,\underline{m}}\mathrm{e}^{k+1}\lor \mathrm{e}^{d\sum_{i=1}^k m_i\log m_i})$ and $C_{1, m_1} = O(2^{m_1})$.
\end{corollary}

We also quickly deduce the following corollary by scaling. 
\begin{corollary}\label{cor: 1 int estimates}
      Fix $k\in \N, \ell> 0$, $\ell>0$ and $m_1, m_2, \cdots, m_k \in \N$ and consider the functions
    \begin{equation*}
        \begin{array}{rr}
            h^\ell_{m_1, \cdots, m_k}(b) &= \displaystyle\int_{x_1\le x_2\le \cdots\le x_k\le b}\mathrm{e}^{-\frac{x_1}{2\ell}}\prod_{i=1}^{k}\mathrm{e}^{-\frac{x_i^2}{4\ell}}\cdot(x_2-x_1)^{m_1}\cdot \cdots \cdot (b-x_k)^{m_k}\diff x_1\cdots \diff x_{k},\, b\in \R.
        \end{array}
    \end{equation*}
    and 
    \begin{equation*}
        \begin{array}{rr}
            g^{\ell}_{m_1, \cdots, m_k}(b) &= \displaystyle\int_{x_1\le x_2\le \cdots\le x_k\le b}\prod_{i=1}^{k}\mathrm{e}^{-\frac{x_i^2}{4\ell}}\cdot(x_2-x_1)^{m_1}\cdot \cdots \cdot (b-x_k)^{m_k}\diff x_1\cdots \diff x_{k},\, b\in \R.
        \end{array}
    \end{equation*}
    Then there exists a universal constant $d>0$ such that 
    \[
     \frac{h^{\ell}_{m_1, \cdots, m_k}}{g^{\ell}_{m_1, \cdots, m_k}}(b)=\frac{h^{1/2}_{m_1, \cdots, m_k}}{g^{1/2}_{m_1, \cdots, m_k}}(b/(2\ell)^{1/2}) \le  C_{k,\underline{m}} \cdot \left(\mathrm{e}^{\frac{(k-1)b}{(2\ell)^{1/2}}}+\mathrm{e}^{-\frac{b}{(2\ell)^{1/2}}}\right), \quad b\in \R.
    \]
    where $C_{k+1, \underline{m},m_{k+1}} = 2^{m_{k+1}}(C_{k,\underline{m}}\mathrm{e}^{k+1}\lor \mathrm{e}^{d\sum_{i=1}^k m_i\log m_i})$ and $C_{1, m_1} = O(2^{m_1})$.
\end{corollary}

\begin{proof}
    By a change of variables in both integrals and Corollary \ref{cor: 1 int estimates} we have 
    \[
     \frac{h^{\ell}_{m_1, \cdots, m_k}}{g^{\ell}_{m_1, \cdots, m_k}}(b)=\frac{h^{1/2}_{m_1, \cdots, m_k}}{g^{1/2}_{m_1, \cdots, m_k}}(b/(2\ell)^{1/2}) \le  N!(2\ell)^{N/2} \cdot C_{k,\underline{m}} \cdot \left( \mathrm{e}^{\frac{(k-1)b}{(2\ell)^{1/2}}}+\mathrm{e}^{-\frac{b}{(2\ell)^{1/2}}}\right), \quad b\in \R.
    \]
    where $C_{k+1, \underline{m},m_{k+1}} = 2^{m_{k+1}}(C_{k,\underline{m}}\mathrm{e}^{k+1}\lor \mathrm{e}^{d\sum_{i=1}^k m_i\log m_i})$ and $C_{1, m_1} = O(2^{m_1})$.
\end{proof}

Combining the above integral estimates gives the following corollary.

\begin{corollary}\label{cor: explicit constants good bound}
 Fix $N,k\in \N, \ell> 0$, and $m_1, m_2, \cdots, m_k \in \N$ and consider the functions
    \begin{equation*}
        \begin{array}{rr}
            f^\ell_{m_1, \cdots, m_k}(b) &= \displaystyle\int_{x_1\le x_2\le \cdots\le x_k\le b}|x_1|^N\prod_{i=1}^{k}\mathrm{e}^{-\frac{x_i^2}{4\ell}}\cdot(x_2-x_1)^{m_1}\cdot \cdots \cdot (b-x_k)^{m_k}\diff x_1\cdots \diff x_{k},\, b\in \R.
        \end{array}
    \end{equation*}
    and 
    \begin{equation*}
        \begin{array}{rr}
            g^{\ell}_{m_1, \cdots, m_k}(b) &= \displaystyle\int_{x_1\le x_2\le \cdots\le x_k\le b}\prod_{i=1}^{k}\mathrm{e}^{-\frac{x_i^2}{4\ell}}\cdot(x_2-x_1)^{m_1}\cdot \cdots \cdot (b-x_k)^{m_k}\diff x_1\cdots \diff x_{k},\, b\in \R.
        \end{array}
    \end{equation*}
    Then there exists a universal constant $d>0$ such that 
    \[
     \frac{f^{\ell}_{m_1, \cdots, m_k}}{g^{\ell}_{m_1, \cdots, m_k}}(b) \le  N!(2\ell)^{N/2} \cdot C_{k,\underline{m}} \cdot \left( \exp\left(\frac{kb}{(2\ell)^{1/2}}\big)+\exp\big(-\frac{b}{(2\ell)^{1/2}}\right)\right), \quad b\in \R.
    \]
    where $C_{k+1, \underline{m},m_{k+1}} = 2^{m_{k+1}}(C_{k,\underline{m}}\mathrm{e}^{k+1}\lor \mathrm{e}^{d\sum_{i=1}^k m_i\log m_i})$ and $C_{1, m_1} = O(2^{m_1})$.
\end{corollary}
\begin{proof}
First observe that
\begin{equation*}
    |x|^N \le N!(2\ell)^{N/2} \cdot \mathrm{e}^{|x|/(2\ell)^{1/2}}\le N!(2\ell)^{N/2} \cdot \big(\mathrm{e}^{x/(2\ell)^{1/2}}+\mathrm{e}^{-x/(2\ell)^{1/2}}\big),\, x\in \R.
\end{equation*}
And so we estimate
    \begin{equation*}
        \begin{array}{cc}
            \displaystyle\frac{f^\ell_{m_1, \cdots, m_k}}{g^\ell_{m_1, \cdots, m_k}(b)}(b) \le N!(2\ell)^{N/2} \cdot \displaystyle\frac{\mathrm{e}^{b/(2\ell)^{1/2}}g^\ell_{m_1, \cdots, m_k}(b) + h^\ell_{m_1, \cdots, m_k}(b)}{g^\ell_{m_1, \cdots, m_k}(b)}
        \end{array}
    \end{equation*}
    with $ h^\ell_{m_1, \cdots, m_k}(b)$ as in Corollary \ref{cor: 1 int estimates} and so we have
    \begin{equation*}
        \begin{array}{cc}
            &\displaystyle\frac{f^\ell_{m_1, \cdots, m_k}}{g^\ell_{m_1, \cdots, m_k}}(b) \le N!(2\ell)^{N/2} \cdot \big( \mathrm{e}^{b/(2\ell)^{1/2}} + \frac{h^\ell_{m_1, \cdots, m_k}(b)}{g^\ell_{m_1, \cdots, m_k}(b)}\\
            & \le  N!(2\ell)^{N/2} \cdot \big(\mathrm{e}^{b/(2\ell)^{1/2}} + C_{k,\underline{m}} \cdot (\mathrm{e}^{\frac{(k-1)b}{(2\ell)^{1/2}}}+\mathrm{e}^{-\frac{b}{(2\ell)^{1/2}}})\big)\\
            & \le N!(2\ell)^{N/2} \cdot C_{k,\underline{m}} \cdot \big( \exp\big(\frac{kb}{(2\ell)^{1/2}}\big)+\exp\big(-\frac{b}{(2\ell)^{1/2}}\big)\big)
        \end{array}
    \end{equation*}
    where $C_{k+1, \underline{m},m_{k+1}} = 2^{m_{k+1}}(C_{k,\underline{m}}\mathrm{e}^{k+1}\lor \mathrm{e}^{d\sum_{i=1}^k m_i\log m_i})$ and $C_{1, m_1} = O(2^{m_1})$.
\end{proof}
\begin{remark}
    Observe that $C_k = O(\mathrm{e}^{ck^2\log k})$ for some universal constant $c>0$ given the constraint that $\sum_{i=1}^k m_i= O(k^2)$.
\end{remark}

The following lemma is a stability result for the structural form of pointwise estimates of terms that appear in estimates of the inhomogeneous transition densities \ref{eq: nu}.
\begin{lemma}\label{lemma: integral estimate}
Fix $m\in \N$, let $g:\R\to \R $ be a smooth function such that there exists an $N\in \N$ with the following estimate
\begin{equation*}\label{eq: poly-gaussian bounds}
    C_1(\theta, \beta)\langle y\rangle^{-1}\cdot \mathrm{e}^{-\beta(y)_-}\cdot \mathrm{e}^{-\theta(y)^2_-} \le g(y) \le C_2(N, \theta, \beta) \langle y\rangle^{N}\cdot \mathrm{e}^{-\theta(y)^2_-},\quad y\in \R.
\end{equation*}
for some $\theta >0$, $C_1(\theta, \beta), C_2(N, \theta, \beta)>0$. Then, we have the upper and lower bounds
    \begin{align*}
          &C^\prime_1(\theta, \beta)\langle (x)_-\lor 1 \rangle ^{-1} \mathrm{e}^{-(\theta + 1/2)(x)_-^2}\mathrm{e}^{-(x)_-\big(\beta+3\big)}\le \displaystyle\int^x_{-\infty}\mathrm{e}^{-\frac{(y)^2}{2}} g(y)(x-y)^m\diff y \\
         & \le C^\prime_2(N, m, \theta, \beta)\langle x\rangle ^{N+m+1}\mathrm{e}^{-(\theta+1/2)(x)_-^2},\quad x\in \R.
    \end{align*}
for $C^\prime_1(\theta, \beta) = C_1(\theta, \beta)\frac{c}{(2\theta+1)^m}\frac{2^{1/2}}{\frac{\beta+3}{2\theta +1} +  2^{1/2}}\mathrm{e}^{(\theta + 1/2)(\frac{\beta+2}{2\theta +1})^2}\mathrm{e}^{-(\theta + 1/2)\big(\frac{\beta+3}{2\theta +1}\big)^2}$ and\\ $C^\prime_2(N, m, \theta, \beta) =  d^{(m+N)\log (m+N)}\cdot C_2(N, \theta, \beta)$, for a universal constant $c>0$, that is of the same form as \ref{eq: poly-gaussian bounds}. Furthermore, for all $\lambda>0$, $K\in \N$
\begin{align*}
    \displaystyle\int^x_{-\infty} \langle y \rangle^K \mathrm{e}^{\lambda(y)_-}\mathrm{e}^{-\frac{y^2}{2}} g(y)(x-y)^m\diff y &\le C_{N,K, M, m, \theta, \beta, \lambda}  \langle x\rangle ^{N+K+m+1}\mathrm{e}^{\big(\beta +3\big)\cdot (x)_-}\\
    & \cdot\displaystyle\int^x_{-\infty}\mathrm{e}^{-\frac{(y)^2}{2}} g(y)(x-y)^m\diff y\, \quad x\in \R.
\end{align*}
where 
\begin{align*}
         & C_{N,K,M,m,\theta, \beta, \lambda}=C_1(\theta, \beta) \cdot C_2(N, \theta, \beta) d^{(N+K+m)\log(N+K+m)}\cdot(2\theta+1)^m  \\
         & \cdot \mathrm{e}^{\frac{\lambda^2}{2\theta + 1}} \mathrm{e}^{(\theta + 1/2)\big(\frac{\beta+3}{2\theta +1}\big)^2} \mathrm{e}^{-(\theta + 1/2)\big(\frac{\beta+2}{2\theta +1}\big)^2}
    \end{align*}
for a universal constant $d>0$ and $\langle\cdot \rangle = (\cdot ^2 + 1)^{1/2}$. 
\end{lemma}

\begin{proof}
Using the pointwise bounds in \ref{eq: poly-gaussian bounds}
    \begin{align*}
             & C_1(\theta, \beta)\displaystyle\int^x_{-\infty}\mathrm{e}^{-\frac{y^2}{2}}\langle y\rangle^{-1}\cdot\mathrm{e}^{-\beta(y)_-}\cdot \mathrm{e}^{-\theta(y)^2_-} (x-y)^m\diff y\\
             & \le\displaystyle\int^x_{-\infty}\mathrm{e}^{-\frac{y^2}{2}}g(y)(x-y)^m\diff y\le C_2(N, \theta, \beta)\displaystyle\int^x_{-\infty}\langle y\rangle ^N\mathrm{e}^{-\frac{y^2}{2}}\cdot \mathrm{e}^{-\theta(y)^2_-} (x-y)^m\diff y\,.
    \end{align*}
The lower bound can be further estimated from below
    \begin{align*}
         &\displaystyle\int^x_{-\infty}\mathrm{e}^{-\frac{y^2}{2}}\langle y\rangle^{-1}\cdot \mathrm{e}^{-\beta(y)_-}\cdot\mathrm{e}^{-\theta(y)^2_-} (x-y)^m\diff y \\
         & \ge \frac{1}{(2\theta+1)^m}\displaystyle\int^{-(x)_- -1/(2\theta + 1)}_{-\infty}\langle y\rangle^{-1}\cdot \mathrm{e}^{\beta y}\cdot \mathrm{e}^{-(\theta + 1/2)(y)^2_-}\diff y\\
         &\ge 2^{1/2} \frac{1}{(2\theta+1)^m}\displaystyle\int_{(x)_- +1/(2\theta + 1)}^{\infty} \frac{1}{y\lor 1}\mathrm{e}^{-\beta y}\cdot\mathrm{e}^{-(\theta + 1/2)y^2}\diff y\\
         & \ge2^{3/2}\frac{c}{(2\theta+1)^m}\displaystyle\int_{(x)_- + 1/(2\theta + 1)}^{\infty}\mathrm{e}^{-(\beta + 2) y} \mathrm{e}^{-(\theta + 1/2)(y)^2}\diff y\\
         & \ge \frac{c}{(2\theta+1)^m}\mathrm{e}^{(\theta + 1/2)(\frac{\beta+2}{2\theta +1})^2}\displaystyle\int_{(x)_- }^{\infty}\mathrm{e}^{-(\theta + 1/2)\big(y +\frac{\beta+3}{2\theta +1}\big)^2}\diff y\\
         & \ge \frac{c}{(2\theta+1)^m}\mathrm{e}^{(\theta + 1/2)(\frac{\beta+2}{2\theta +1})^2}\langle \big((x)_-+\frac{\beta+3}{2\theta +1}\big)\lor 1 \rangle ^{-1} \mathrm{e}^{-(\theta + 1/2)\big((x)_- +\frac{\beta+3}{2\theta +1}\big)^2}\\
         & = \frac{c}{(2\theta+1)^m}\mathrm{e}^{(\theta + 1/2)(\frac{\beta+2}{2\theta +1})^2}\mathrm{e}^{-(\theta + 1/2)\big(\frac{\beta+3}{2\theta +1}\big)^2}\cdot \frac{2^{1/2}}{\frac{\beta+3}{2\theta +1} +  2^{1/2}}\langle (x)_-\lor 1 \rangle ^{-1} \mathrm{e}^{-(\theta + 1/2)(x)_-^2}\mathrm{e}^{-(x)_-\big(\beta+3\big)}
    \end{align*}
    for some universal constant $c>0$, using the asymptotics of the error function, see Section \ref{app: erf asymptotics} in the Appendix, and the easily checked fact 
    \[
    \langle \big(x+\alpha\big)\lor 1 \rangle ^{-1}\ge \frac{2^{1/2}}{\alpha +  2^{1/2}}\langle x\lor 1 \rangle ^{-1} ,\; x,\alpha \ge 0\,,
    \]
    in the last inequality. Now the upper bound is also estimated
    \begin{align*}
         &\displaystyle\int^x_{-\infty}\mathrm{e}^{-\frac{y^2}{2}}\langle y\rangle^{N}\cdot \mathrm{e}^{-\theta(y)^2_-}(x-y)^m\diff y \le  \displaystyle\int^x_{-\infty}\langle y\rangle^{N}\cdot \mathrm{e}^{-(\theta+1/2)(y)^2_-}(x-y)^m\diff y\\
         & \le \displaystyle\int^x_{-\infty}\langle y\rangle^{N}\cdot \mathrm{e}^{-(\theta+1/2)(y)^2_-}(\langle x\rangle+\langle y\rangle)^m\diff y\\
         & \le 2^m\displaystyle \sum_{k=0}^m \langle x\rangle^{m-k}\displaystyle\int^x_{-\infty}\langle y\rangle^{N+k}\cdot \mathrm{e}^{-(\theta+1/2)(y)^2_-}\diff y\\
         & \le 2^m\displaystyle \sum_{k=0}^m \langle x_-\rangle^{m-k}\left(\displaystyle\int^{-(x)_-}_{-\infty}+\displaystyle\int^x_{-(x)_-}\right)\langle y\rangle^{N+k}\cdot \mathrm{e}^{-(\theta+1/2)(y)^2_-}\diff y\\
         & \le 2^m\displaystyle \sum_{k=0}^m \langle x\rangle^{m-k}\left(\displaystyle\int^{-(x)_-}_{-\infty}\langle y\rangle^{N+k}\cdot \mathrm{e}^{-(\theta+1/2)(y)^2_-}\diff y+ \mathbf{1}_{x\ge 0}\langle x\rangle^{N+k+1}\right) \\
         & \le 2^{2m+N}\displaystyle \sum_{k=0}^m \langle x\rangle^{m-k}\left(\displaystyle\int^{\infty}_{(x)_-}(y+1)^{N+k}\cdot \mathrm{e}^{-(\theta+1/2)(y)^2}\diff y + \mathbf{1}_{x\ge 0}\langle x\rangle^{N+k+1}\right) \\
         &\le d^{(m+N)\log (m+N)} \langle x\rangle^{N+m+1} \mathrm{e}^{-(\theta+1/2)(x)_-^2}\\
    \end{align*}
      for a universal $d>0$, where in the last line the comparison $(|y|\lor1)\le \langle y\rangle \le 2^{1/2}(|y|\lor 1)$ and repeated integration by parts. Thus, we have the upper and lower bounds
    \begin{align*}
          & C^\prime(\theta, \beta)\langle (x)_-\lor 1 \rangle ^{-1} \mathrm{e}^{-(\theta + 1/2)(x)_-^2}\mathrm{e}^{-(2\theta + 1)(x)_-\big(\frac{\beta}{2\theta +1}+ 2\big)}\\
          & \le\displaystyle\int^x_{-\infty}\mathrm{e}^{-\frac{y^2}{2}}g(y)(x-y)^m\diff y. \\
         & \le C^\prime_2(N, m, \theta, \beta)\langle x\rangle ^{N+m+1}\mathrm{e}^{-(\theta+1/2)(x)_-^2},\quad x\in \R.
    \end{align*}
for \[C^\prime(\theta, \beta) = C_1(\theta, \beta)\frac{1}{(2\theta+1)^m}\mathrm{e}^{(\theta + 1/2)(\frac{\beta+2}{2\theta +1})^2}\mathrm{e}^{-(\theta + 1/2)\big(\frac{\beta+3}{2\theta +1}\big)^2}
\]
and 
\[
C^\prime_2(N, m, \theta, \beta) =  d^{(m+N)\log (m+N)}\cdot C_2(N, \theta, \beta)\,,
\]
for a universal constant $c>0$. The bounds on $g$ and some analogous manipulations give 
    \begin{align*}
         &\displaystyle\int^x_{-\infty} \mathrm{e}^{\lambda (y)_-}\mathrm{e}^{-\frac{y^2}{2}}g(y)(x-y)^m\diff y \le C_2(N, \theta, \beta) \displaystyle\int^x_{-\infty} \langle y\rangle ^N\mathrm{e}^{\lambda (y)_-}\mathrm{e}^{-\frac{y^2}{2}}\mathrm{e}^{-\theta(y)^2_-}(x-y)^m\diff y \\
         &\le C_2(N, \theta, \beta)\displaystyle\int^x_{-\infty}\langle y\rangle ^N \mathrm{e}^{\lambda (y)_-}\mathrm{e}^{-(\theta+1/2)(y)^2_-}(x-y)^m\diff y\\
         & = C_2(N, \theta, \beta)\cdot\mathrm{e}^{\frac{\lambda^2}{2\theta + 1}}\displaystyle\int^x_{-\infty}  \langle y\rangle ^N\mathrm{e}^{-(\theta+1/2)(y_-+\frac{\lambda}{2\theta + 1})^2}(x-y)^m\diff y\\
         & = C_2(N, \theta, \beta)\cdot\mathrm{e}^{\frac{\lambda^2}{2\theta + 1}}\displaystyle\left(\int^{-(x)_-}_{-\infty} + \int^x_{-(x)_-}\right)  \langle y\rangle ^N\mathrm{e}^{-(\theta+1/2)(y_-+\frac{\lambda}{2\theta + 1})^2}(x-y)^m\diff y\\
         & = C_2(N, \theta, \beta)\cdot\mathrm{e}^{\frac{\lambda^2}{2\theta + 1}}\displaystyle\left(\int^{\infty}_{(x)_-} \langle y\rangle ^N\mathrm{e}^{-(\theta+1/2)(y+\frac{\lambda}{2\theta + 1})^2}(x+y)^m\diff y\right.\\
         &\left. + \mathbf{1}_{x\ge 0}\int^x_{-(x)_-}  \langle y\rangle^N \mathrm{e}^{-(\theta+1/2)(y+\frac{\lambda}{2\theta + 1})^2}(x-y)^m\diff y\right)\\
         & \le d^{(N+m)\log(N+m)}C_2(N, \theta, \beta)\cdot\mathrm{e}^{\frac{\lambda^2}{2\theta + 1}}\langle (x)_+\rangle ^{N+m}\langle (x)_-\rangle ^{N+m}\mathrm{e}^{(x)_-}\mathrm{e}^{-(\theta+1/2)(x)_-^2}\\
         & = d^{(N+m)\log(N+m)}C_2(N, \theta, \beta)\cdot\mathrm{e}^{\frac{\lambda^2}{2\theta + 1}} \mathrm{e}^{(x)_-}\mathrm{e}^{(\theta + 1/2)\big(\frac{\beta}{2\theta +1}+ 2\big)^2}\\
         & \cdot \langle (x)_+\rangle ^{N+m}\langle (x)_-\rangle ^{N+m+1}\langle (x)_-\lor 1 \rangle ^{-1}\mathrm{e}^{(2\theta+1)\big(\frac{\beta+3}{2\theta +1}\big) \cdot (x)_-}\mathrm{e}^{(x)_-}\mathrm{e}^{-(\theta + 1/2)\big((x)_- +\frac{\beta+3}{2\theta +1}\big)^2}\\
         &\le d^{(N+m)\log(N+m)}C_2(N, \theta, \beta)\cdot\mathrm{e}^{\frac{\lambda^2}{2\theta + 1}} \mathrm{e}^{(\theta + 1/2)\big(\frac{\beta+3}{2\theta +1}\big)^2}\\
         & \cdot\frac{(2\theta+1)^m}{c\mathrm{e}^{(\theta + 1/2){(\frac{\beta+2}{2\theta +1})^2}}} \langle x\rangle ^{N+m+1}\mathrm{e}^{\big(\beta + 4\big)\cdot (x)_-} \displaystyle\int^x_{-\infty}\mathrm{e}^{-\frac{y^2}{2}}\langle y \rangle^{-1} \mathrm{e}^{-\beta (y)_-}\mathrm{e}^{-\theta(y)^2_-}(x-y)^m\diff y\\
         &\le C_{N,M, m, \theta, \beta, \lambda}  \langle x\rangle ^{N+m+1}\mathrm{e}^{\big(\beta+4\big)\cdot (x)_-} \displaystyle\int^x_{-\infty}\mathrm{e}^{-\frac{y^2}{2}}g(y)(x-y)^m\diff y
    \end{align*}   
    for some universal constant $d>0$, and  
    \[C_{N,M, m, \theta, \beta, \lambda} = C_1(\theta, \beta) \cdot C_2(N, \theta, \beta) d^{(N+m)\log(N+m)}\cdot(2\theta+1)^m\cdot \mathrm{e}^{\frac{\lambda^2}{2\theta + 1}} \mathrm{e}^{(\theta + 1/2)\big(\frac{\beta+3}{2\theta +1}\big)^2} \mathrm{e}^{-(\theta + 1/2)\big(\frac{\beta+2}{2\theta +1}\big)^2}.
    \]
\end{proof}

\subsection{Uniform pathwise Radon-Nikodym deriative estimates for homogeneous BLPP}\label{sec: uniform path blpp hom}
Using the pathwise estimates from Proposition \ref{prop: top line pitman pointwise bound general case}, we are now in a position to prove a uniform of spatial increments of homogeneous BLPP against Brownian motion on compacts; this is the content of the following proposition.

\begin{proposition}\label{prop: uniform pathwise est blpp}
Fix a time horizon $T > 0$ and depth $m\in \N$ and consider the Brownian LPP process on $[\ell,r]$ for $0<\ell<r$
\[
B[(-T, m)\rightarrow (\cdot, 1)].
\]
Now, the increment process 
\[
h^T(\cdot+\ell) = B[(-T, m)\rightarrow (\cdot+\ell, 1)]-B[(-T, m)\rightarrow (\ell, 1)]
\]
 has a Radon-Nikodym derivative $X_T$ on $[0,r-\ell]$ against a rate two Brownian motion with almost sure pointwise bound 
    \begin{equation*}
    \begin{array}{cc}
    \frac{c^{m(m-1)}n^{m(m-1)}}{\prod_{j=1}^{m-1}j!}&\cdot \mathbb{E}_{Z}\Bigg[(Z_+ +1)^{m(m-1)/2}\cdot ((\xi(r-\ell)+\sqrt{T+\ell}Z)_+/\sqrt{T+r}+1)^{m(m-1)/2}\Bigg]
    \end{array}
\end{equation*}
on paths $\xi$ on $[0,r-\ell]$ where $Z$ is an independent centred variance $2$ Gaussian random variable for some universal constant $c>1$.

Furthermore, it follows that there is a non-negative $Y\in L^{\infty-}(\mu)$ such that $\sup_{T\ge 0} X_T \le Y$ a.s. In other words, the family of Radon-Nikodym Derivatives $(X_T)_{T\le 0}$ is tight with respect to the rate two Wiener measure on $[0,r-\ell]$.
\end{proposition}

\begin{proof}
First observe that using Proposition \ref{prop: top line pitman pointwise bound general case} the process $B[(0, m)\rightarrow (\cdot, 1)]$ on $[T+\ell,T+r]$ has a Radon-Nikodym  against rate two Brownian motion with pointwise bound
\begin{equation*}
    \frac{c^{n(n-1)}n^{n(n-1)}}{\prod_{j=1}^{n-1}j!}\cdot (\xi(T+\ell)_+/\sqrt{T+\ell}+1)^{m(m-1)/2}\cdot (\xi(T+r)_+/\sqrt{-T+r}+1)^{m(m-1)/2}
\end{equation*}
for some universal constant $c>1$ on paths $\xi$ on $[T+\ell, T+r]$. Now, using Proposition \ref{prop: 3} the process $h^T(\cdot)$ on $[0,r-\ell]$ has a Radon-Nikodym  against rate two Brownian motion with pointwise bound
 \begin{equation*}
    \begin{array}{cc}
    \frac{c^{m(m-1)}n^{m(m-1)}}{\prod_{j=1}^{m-1}j!}&\cdot \mathbb{E}_{Z}\Bigg[(Z_+ +1)^{m(m-1)/2}\cdot ((\xi(r-\ell)+\sqrt{T+\ell}Z)_+/\sqrt{T+r}+1)^{m(m-1)/2}\Bigg]
    \end{array}
\end{equation*}
where $Z$ is an independent centred variance $2$ Gaussian random variable for some universal constant $c>1$, as required.
\end{proof}
\subsection{Increment regularisation}\label{sec: inc reg}
We record a regularisation lemma for the Radon-Nikodym derivatives of a continuous process against the Wiener measure on compacts away from zero and its increments. In particular for $0<\ell<r$, one sees that under the map 
$$\Phi : C_{*,*}([\ell,r])\to  C_{0,*}([0,r-\ell]) : \xi(\cdot) \mapsto \xi(\cdot + \ell)-\xi (\ell)\label{eq: pushforward}$$
the induced map mapping Radon-Nikodym derivatives is `contractive' under suitable assumptions, in the sense described below.

\begin{proposition}(Increment regularisation)\label{prop: 3} Let $0<\ell<r$ and $b\in \R$. Let $B$ be a continuous process on $[0,\infty)$ such that the law of $B$ is absolutely continuous against rate two Brownian motion starting from the origin away from zero. Suppose furthermore that for $0<\ell<r$, the Radon-Nikodym derivative is pointwise almost surely bounded by
$$
f(\xi(\ell), \xi(r))
$$
for some non-negative measurable function $f:\R^2 \to \R$ on paths $\xi$ on $[\ell, r]$. Then, the joint law of $(B(\ell), B(\cdot + \ell)-B(\ell))$ restricted to $[0,r-\ell]$ is absolutely continuous with respect to the measure $\lambda\times \mu$, where $\lambda$ is the one-dimensional Lebesgue measure and $\mu$ denotes the law of a rate two Brownian motion starting from $(0,0)$ restricted to $[0,r-\ell]$ with a.e.-pointwise bound on points $(y, \xi)\in \R_\ge\times C_{*,*}([0,r-\ell])$,
$$
g(y,\xi(r-\ell))
$$
for some non-negative measurable function $g$ (non-decreasing in its last argument if $f$ is). Furthermore, the following `contractivity' is observed for all $p>1$
\[
\norm{\frac{\diff \Phi_* B}{\diff \Phi_*\mu}}_{L^p(\Phi_*\mu)}\le \norm{\frac{\diff  B}{\diff \mu}}_{L^p(\mu)},\qquad p>1\,,
\]
where $\Phi_*$ denotes the pushforward under the map $\Phi$ as in (\ref{eq: pushforward}).
\end{proposition}

\begin{proof}
Fix $A\subseteq \R_\ge\times C_{*,*}([0,r-\ell])$ Borel measurable and estimate
\begin{equation*}
    \begin{array}{ll}
     & \PP((B(\ell), B(\cdot + \ell)-B(\ell))\in A) \\
     & \le \displaystyle\int_{C_{*,*}([\ell,r])} \mathbf{1}(\xi(\ell), \xi(\cdot+\ell)-\xi(\ell))\in A)\cdot f(\xi(\ell), \xi(r))\mu(\diff \xi)
\end{array}
\end{equation*}
And by the Markov property (independent increments) enjoyed by Brownian motion we have
\begin{equation*}
    \begin{array}{ll}
     & \PP((B(\ell), B(\cdot + \ell)-B(\ell))\in A) \\
     & \le \displaystyle\int_{C_{*,*}([\ell,r])} \mathbf{1}(\xi(\ell), \xi(\cdot+\ell)-\xi(\ell))\in A)\cdot f(\xi(\ell), \xi(r))\mu(\diff \xi)\\
     & = \displaystyle\int_{C_{*,*}([0,r-\ell])}\int_{\R} \mathbf{1}(x, \xi(\cdot))\in A)\cdot f(x, \xi(r-\ell)+x)\phi_{\ell}(x)\lambda(\diff x)\mu(\diff \xi)\\
\end{array}
\end{equation*}
and we thus arrive at the a.e.-pointwise bound on points $(y, \xi)\in \R_\ge\times C_{*,*}([0,r-\ell])$,
$$
g(y,\xi(r-\ell)) = f(y, \xi(r-\ell)+y)\phi_{\ell}(y)
$$
for some function $g$ non-decreasing in its last argument. This means that we have the following norm estimates for all $p>1$
\begin{equation*}
\begin{array}{ll}
    &\norm{\frac{\diff \Phi_* B}{\diff \Phi_*\mu}}_{L^p(\Phi_*\mu)} = \left(\mathbb{E}_0[\mathbb{E}_{\hat B (r-\ell)}[f(B(\ell), \hat B(r-\ell)+ B(\ell))]^p]\right)^{1/p}\\
    &\norm{\frac{\diff  B}{\diff \mu}}_{L^p(\mu)} = \left(\mathbb{E}_0[f(B(\ell), B(r))^p]\right)^{1/p}\,.
\end{array}
\end{equation*}
Now, by H\"{o}lder, we observe that
\[
\norm{\frac{\diff \Phi_* B}{\diff \Phi_*\mu}}_{L^p(\Phi_*\mu)}\le \norm{\frac{\diff  B}{\diff \mu}}_{L^p(\mu)}\,,\qquad p>1\,,
\]
which concludes the proof.
\end{proof}

Note that Proposition \ref{prop: 3} suggests the estimates in Section \ref{sec: BM tasep rn estimates} can be improved, though in the case of inhomogeneous BLPP it is not clear how one might proceed, as one needs a refinement of the ratios of densities coming from \ref{lemma: density Warren}.

\subsection{Monotonicity properties of Radon-Nikodym derivatives of inhomogeneous BLPP}\label{subsec: monotonicity properties of RN of reflections}
We now briefly consider inhomogeneous Brownian LPP with `almost' homogeneous initial data (that is only the first entry is non-vanishing). In the following proposition, we obtain Radon-Nikodym derivative estimates of inhomogeneous BLPP with the above data against Brownian motion on compacts. This argument notably bypasses the technical proof of Theorem \ref{thm: Radon-Nikodym  density inhom BLPP}, relying only on a coupling of the Brownian motion in the Gelfand-Tsetslin cone to a larger collection of interlaced diffusions.  

\begin{proposition}\label{prop: top line pitman pointwise bound general case special coupling} Let $0<\ell<r$ and $B_{1}$ be a Brownian motion starting from $b>0$ and let $B_2,B_2\cdots, B_{n+1}$ be Brownian motions on $[0,\infty)$ starting from the origin where $B_{1:n+1}$ are mutually independent. Set $\underline{g}=(b,0,\cdots, 0)\in\R^{n+1}_\ge$ and let $H(\cdot)$ denote an inhomogeneous Brownian LPP started from initial data $\underline{g}$. Then, for all $\ell>0$, the joint law of $(H_1(\ell/2), H_2 (\ell/2), H_2(\cdot + \ell/2)-H_2(\ell/2))$ restricted to $[0,r-\frac{\ell}{2}]$ is absolutely continuous with respect to the measure $\lambda^2\times \mu$, where $\lambda$ is the one-dimensional Lebesgue measure and $\mu$ denotes the law of a rate two Brownian motion starting from the origin restricted to $[0,r-\frac{\ell}{2}]$ with a.e.-pointwise bound on points $(y, z,  \xi)\in \R^2_\ge\times C_{*,*}([0,r-\ell/2])$,
$$
g(y,z,\xi(r-\ell/2))
$$
for some function $g$ non-decreasing in its last argument.
\end{proposition}
\begin{proof}

Fix $C\subseteq \R^2_\ge$, $A\subseteq C_{*,*}([0,r-\ell/2])$ Borel measurable and upon conditioning estimate using an argument analogous to that in Proposition \ref{prop: top line pitman pointwise bound general case}
\begin{equation*}
    \begin{array}{cc}
         & \PP((H_1(\ell/2), H_2 (\ell/2))\in C, (H_2(\cdot + \ell/2)-H_2(\ell/2))\in A) \\
         & =\PP((H_1(\ell/2), H_2 (\ell/2))\in C, (H_2(\cdot + \ell/2)-H_2(\ell/2))\in A, \mathrm{NoInt}) \\
         & +  \PP((H_1(\ell/2), H_2 (\ell/2))\in C, (H_2(\cdot + \ell/2)-H_2(\ell/2))\in A, \mathrm{NoInt}^c) \\
    \end{array}
\end{equation*}
where the event $\mathrm{NoInt}$ is the event that $\{B_{1}>H_2\, \mathrm{on}\, [0,\ell/2]\}$. On $\mathrm{NoInt}^c$, $H_1(\ell/2)$ can be coupled with the homogeneous ensemble generated by $H_{2:n+1}$ and a rate two Brownian motion starting from the origin (in place of $B_{1}$); on $\mathrm{NoInt}$, $H_1(\ell/2) = B_{1}(\ell/2)$ and we thus have the estimate 
\begin{equation*}
    \begin{array}{cc}
         & \PP((H_1(\ell/2), H_2 (\ell/2))\in C, (H_2(\cdot + \ell/2)-H_2(\ell/2))\in A) \\
         & \le \PP((B_{n+1}(\ell/2), H_2 (\ell/2))\in C, (H_2(\cdot + \ell/2)-H_2(\ell/2))\in A) \\
         & +  \PP((H_1(\ell/2), H_2 (\ell/2))\in C, (H_2(\cdot + \ell/2)-H_2(\ell/2))\in A) \\
    \end{array}
\end{equation*}
where in the second term, we take $b=0$ (in a slight abuse of notation). Concentrating on the second term, we obtain using the fact that $H_2$ can be realised as the top line of a Dyson Brownian motion (which under another abuse of notation in the conditioning below, we will write as $H_{2:n+1}$)
\begin{equation*}
    \begin{array}{cc}
         & = \mathbb{E}\left[\mathbf{1}((H_1(\ell/2), H_2 (\ell/2))\in C)\cdot\mathbb{E}\left[\mathbf{1}((H_2(\cdot + \ell/2)-H_2(\ell/2))\in A)\lvert H_{2:n+1}|_{[0, \ell/2]}\right]\right]\\
         & \le \mathbb{E}\left[\mathbf{1}(H_1(\ell/2), H_2 (\ell/2))\in C)\cdot\mathbb{E}_{H_{2:n+1}(\ell/2)}\left[\frac{h(\hat{B}+H_{2:n+1}(\ell/2))}{h(H_{2:n+1}(\ell/2))}\mathbf{1}(\hat{B}_1(\cdot))\in A)\right]\right]\,,
    \end{array}
\end{equation*}
where $\hat{B}_{1:n}(\cdot)$ are independent rate two Brownian motions starting from the origin. We further estimate using independence
\begin{equation*}
    \begin{array}{cc}
         & \PP((H_1(\ell/2), H_2 (\ell/2))\in C, (H_2(\cdot + \ell/2)-H_2(\ell/2))\in A) \\
         & \le \mathbb{E}\left[\mathbf{1}((H_1(\ell/2), H_2 (\ell/2))\in C)\cdot\mathbb{E}_{\hat{B}_{2:n}}\left[\frac{h(\tilde{B}+H_{2:n+1}(\ell/2))}{h(H_{2:n+1}(\ell/2))}\right]\mathbf{1}(\hat{B}_1(\cdot))\in A)\right]\,.
    \end{array}
\end{equation*}

Here we take the coupling of $(W(B_{1}, H_2)_1, H_2)(\cdot)$ on paths given in \cite{warren2007dyson} under the measure $Q^{n,+}_{0,0}$ on paths on $W^{n+1, n}=\{ (x,y) \in {\R}^{n+1} \times {\R}^n: x_{1} \leq y_1 \leq x_2\leq \ldots \leq y_n \leq x_{n+1} \}$ where the canonical coordinate process $(X,Y)$ evolves as interlaced Brownian motions such that, crucially for us, we have that $(X_{n+1}, Y_{n})\stackrel{d}{=} (W(B_{1}, \Gamma^n_n)_1, \Gamma^n_n)$ on paths (using the deterministic result of \cite[lemma 2.1]{revuz2013continuous}), where $\Gamma^n_n$ denotes the top curve of an $n-$dimensional Dyson Brownian motion starting from the origin. The reason this is done is there is a nice form for the entry law of $(X,Y)$ under $Q^{n,+}_{0,0}$ and will be exploited below.

Particularly, we can compute the density of $(H_1(\ell/2), H_{2:n+1} (\ell/2))$ to be equal to 
\begin{equation*}
q^n_{\ell/2}(x,y)=\frac{n!}{Z_{n+1}} (\ell/2)^{-(n+1)^2/2} \exp \left\{ -\sum_{i} x_i^2/(2\ell)\right\} \left\{ \prod_{i<j} (x_j-x_i) \right\}\left\{ 
\prod_{i<j} (y_j-y_i) \right\},
\end{equation*}
where the normalisation constant $Z_{n+1}= (2\pi)^{(n+1)/2}\prod_{j<n+1} j!$. For ease of notation, we define the function $f: \R^{n+1}_{\ge}\to \R_{\ge 0}$
\[
h(x_{1:n+1}) = \prod_{1\le i<j\le n+1} (x_{j}-x_{i}), \quad (x_1, \cdots, x_{n+1})\in \R^{n+1}_\ge\,.
\]
We are now able to estimate the above probability using equation (\ref{eq: dyson rn basic estimate}) and thus obtain the contribution to the pointwise upper bound on the density of the joint law (absolute continuity has already been established, see \cite{Sarkar2021Brownian})
\begin{equation*}
    \begin{array}{cc}
         & \PP((H_1(\ell/2), H_2 (\ell/2))\in C, (H_2(\cdot + \ell/2)-H_2(\ell/2))\in A) \\
         & \le \displaystyle\int_{A}\int_{W^{n+1, n}}\mathbf{1}((x_{n+1}, y_n)\in C)\frac{n!}{Z_{n+1}} (\ell/2)^{-(n+1)^2/2}\prod_{i=1}^{n+1}\phi_{\ell/2}(x_i)\cdot h(x_{1:n+1})\\
         & \cdot\mathbb{E}_{\hat{B}_{2:n}}\left[h((\xi, \hat{B}_{2:n})(r-\ell/2)+y_{1:n})\right]\diff x_{1:n}\diff y_{1:n}\mu(\diff\xi)\\
         & = \displaystyle\int_{A}\int_{C}g(y,x,\xi(r-\ell/2))\diff x_{1:n}\diff y_{1:n}\mu(\diff\xi)
    \end{array}
\end{equation*}
$g$ is non-decreasing in its last argument and $\hat{B}_{1:n}$ is a Brownian motion starting from the origin. The contribution from the non-intersection term is analogously derived by estimating and using independence
\begin{equation*}
    \begin{array}{cc}
         & \PP((B_{1}(\ell/2), H_2 (\ell/2))\in C, (H_2(\cdot + \ell/2)-H_2(\ell/2))\in A) \\
         & \le \mathbb{E}\left[\mathbf{1}((B_{1}(\ell/2), H_2 (\ell/2))\in C)\cdot\mathbb{E}_{\hat{B}_{2:n}}\left[\frac{h(\hat{B}(\ell/2)+H_{2:n+1}(\ell/2))}{h(H_{2:n+1}(\ell/2))}\right]\mathbf{1}(\hat{B}_1(\cdot))\in A)\right]\\
         & \le C_{\ell/2}\mathbb{E}\left[\mathbf{1}((B_{1}(\ell/2), W_1 (\ell/2))\in C)\cdot h(W_{1:n}(\ell/2))\cdot\mathbb{E}_{\hat{B}_{2:n}}\left[h(\hat{B}(\ell/2)+W_{1:n}(\ell/2))\right]\mathbf{1}(\hat{B}_1(\cdot))\in A)\right]\\
         & = \displaystyle\int_{A}\int_{\R^{n+1}}\mathbf{1}((y, x_n)\in C)h(x_{1:n})\phi_{\ell/2}(y-b)\displaystyle\prod_{i=1}^n\phi_{\ell/2}(x_i)\cdot\mathbb{E}\left[h((\xi,\hat{B}_{2:n})(r-\ell/2)+x_{1:n})\right]\\
         & \diff x_{1:n}\diff y_{1:n}\mu(\diff\xi)\\
         & = \displaystyle\int_{A}\int_{C}g'(y,x,\xi(r-\ell/2))\diff x_{1:n}\diff y_{1:n}\mu(\diff\xi)\,,
    \end{array}
\end{equation*}
where $C_{\ell/2}$ is as in proposition \ref{prop: top line pitman pointwise bound general case}, for $g'$ non-decreasing in its last argument and $W_{1:n}, \hat{B}_{1:n}$ are Brownian motion starting from the origin. Combining the above, and using that sets of the form $C\times A$ generate the product Borel sigma algebra we conclude.
\end{proof}

\subsection{Error function asymptotics}\label{app: erf asymptotics}
Set  
\begin{equation*}
    \mathrm{erf}(x) = \frac{2}{\sqrt{\pi}}\displaystyle\int_{0}^x\mathrm{e}^{-x^2}dx, \quad x\in \R.
\end{equation*}
Now observe first that for all $z_2\in \R$, $\ell > 0$
\begin{equation*}
    \frac{\sqrt{\pi \ell}}{2} z_2  \left[1-\operatorname{erf}\left(\frac{ -z_2}{2 \sqrt{\ell}}\right)\right]+ \ell\mathrm{e}^{- \frac{z_2^{2}}{4L} } > 0.
\end{equation*}
We also have by the asymptotics of the error function (see \cite{abramowitz1948handbook}) for $r>0$
\begin{equation*}\label{eq: asymptotics erf}
    \frac{\sqrt{\pi \ell}}{2}  \left[1-\operatorname{erf}\left(\frac{ r}{2 \sqrt{\ell}}\right)\right] =  \frac{\sqrt{\pi \ell}}{2} \left[\frac{2\sqrt{\ell}}{\sqrt{\pi}r}\mathrm{e}^{-\frac{r^2}{4L}} + R\left(\frac{r}{2\sqrt{\ell}}\right) \right]
\end{equation*}
where $R(x)$ is a remainder bounded by
\begin{equation*}
    |R(x)|\le \frac{3}{4\sqrt{\pi}}\frac{\mathrm{e}^{-x^2}}{x^3}
\end{equation*}
for all $x>0$. A quick computation shows that for $r^2\ge 6L$
\begin{equation*}
    \frac{\sqrt{\pi \ell}}{2}  \left[1-\operatorname{erf}\left(\frac{ r}{2 \sqrt{\ell}}\right)\right] \ge  \frac{\ell}{2r}\mathrm{e}^{-\frac{r^2}{4L}}.
\end{equation*}

Note that we can further refine the asymptotic expansion \ref{eq: asymptotics erf} for the error function to obtain for $r>0$
\begin{equation*}
    \frac{\sqrt{\pi \ell}}{2}  \left[1-\operatorname{erf}\left(\frac{ r}{2 \sqrt{\ell}}\right)\right] =  \frac{\sqrt{\pi \ell}}{2} \left[\frac{2\sqrt{\ell}}{\sqrt{\pi}r}\mathrm{e}^{-\frac{r^2}{4L}} -\frac{8L^{\frac{3}{2}}}{2\sqrt{\pi}r^3}\mathrm{e}^{-\frac{r^2}{4L}}+  \tilde{R}\left(\frac{r}{2\sqrt{\ell}}\right) \right]
\end{equation*}
where $\tilde{R}(x)$ is a remainder bounded by
\begin{equation*}
    |\tilde{R}(x)|\le \frac{15}{8\sqrt{\pi}}\frac{\mathrm{e}^{-x^2}}{x^5}
\end{equation*}
for all $x>0$. A quick computation shows that for $r^2\ge 30L$
\begin{equation*}
    -\frac{\sqrt{\pi \ell}}{2} r  \left[1-\operatorname{erf}\left(\frac{ -r}{2 \sqrt{\ell}}\right)\right]+ \ell\mathrm{e}^{- \frac{r^{2}}{4L} }\ge \frac{\ell^2}{r^2}\mathrm{e}^{-\frac{r^2}{4L}}.
\end{equation*}

\bibliographystyle{alpha} 
\bibliography{refs} 
\end{document}